\documentclass[11pt,a4paper]{article}
\usepackage[left=3cm, right=3cm, top=3.5cm]{geometry}
\usepackage{amssymb,amsmath, amsthm, amsfonts, rotating}
\usepackage{mathrsfs,color,bbold,url}
\usepackage{stmaryrd}
\usepackage{tikz}
\usepackage[all]{xypic}
\usepackage{enumerate}

\newcommand{\nmdash}{\mathord{{\left|\!\sim\!\right.}}}

\newcommand{\atom}{\mathbb{at}}
\newcommand{\AmA}{A\mid A}

\newtheorem{theorem}{Theorem}[section]
\newtheorem{lemma}[theorem]{Lemma}
\newtheorem{corollary}[theorem]{Corollary}
\newtheorem{proposition}[theorem]{Proposition}
\newtheorem{claim}{Claim}

\theoremstyle{definition}
\newtheorem{definition}[theorem]{Definition}
\newtheorem{remark}[theorem]{Remark}
\newtheorem{example}[theorem]{Example}

\newcommand{\msim}{\mathord\sim}

\newcommand{\llb}{\llbracket}
\newcommand{\rrb}{\rrbracket}

\begin{document}

\title{Boolean algebras of conditionals, probability and logic}
\author{Tommaso Flaminio$^1$, Lluis Godo$^1$, Hykel Hosni$^2$\\ [2mm] %
{\small $^{1}$ Artificial Intelligence Research Institute (IIIA) - CSIC, Barcelona, Spain}\\
{\small \tt email: {\{tommaso,godo\}}@iiia.csic.es}\\[1mm]
{\small $^2$Department of Philosophy, University of Milan, Milano, Italy}\\
{\small \tt email: hykel.hosni@unimi.it}
}
%\address{$^1$Artificial
%Intelligence Research Institute (IIIA - CSIC) \\ Campus UAB,
%Bellaterra 08193, Spain \\ {\tt email: tommaso@iiia.csic.es}}
%\author{Lluis Godo$^1$}
%\address{$^1$Artificial
%Intelligence Research Institute (IIIA - CSIC) \\ Campus UAB,
%Bellaterra 08193, Spain \\ {\tt email: godo@iiia.csic.es}}
%\author{Hykel Hosni$^3$}
%\address{$^3$Department of Philosophy, University of Milan, \\ Via Festa del Perdono 7 - 20122 Milano, Italy \\
%{\tt email: hykel.hosni@unimi.it}}

\date{}
\maketitle

\begin{abstract}
This paper presents an investigation on the structure of conditional
events 
and on the probability measures
which arise naturally in this context.  In particular we introduce a construction which defines a (finite) {\em Boolean algebra of conditionals} from any (finite) Boolean algebra of events. 
By doing so we 
distinguish the properties of conditional events which depend on probability and
those which are intrinsic to the logico-algebraic structure of conditionals. Our
main result provides a way to regard standard two-place conditional 
probabilities as one-place probability functions
on conditional events. We also consider a logical counterpart of our
Boolean algebras of conditionals with links to preferential consequence
relations for non-monotonic reasoning. 
The overall framework of this paper 
provides a novel perspective on  the rich
interplay between logic and probability in the representation of conditional knowledge.\vspace{.2cm}

\end{abstract}

\section{Introduction and motivation}

Conditional expressions are pivotal in representing knowledge and
reasoning abilities of intelligent agents. Conditional reasoning
features in a wide range of areas spanning non-monotonic reasoning,
causal inference, learning, and more generally reasoning under uncertainty.

This paper 
proposes an algebraic structure for conditional events
which serves as a logical basis to analyse the concept of conditional
probability -- a fundamental tool in Artificial Intelligence. 

At least since the seminal work of Gaifman \cite{Gaifman1964}, who in turn develops the
initial ideas of his supervisor Alfred Tarski \cite{Tarski1948}, it has
been considered natural to investigate the conditions under which
Boolean algebras -- i.e. classical logic -- played the role of the
logic of events for probability. The
point is clearly made in \cite{Gaifman1982}:

\begin{quote}
Since events are always described in some language they can be
identified with the sentences that describe them and the probability
function can be regarded as an assignment of values to sentences. The
extensive accumulated knowledge concerning formal languages makes such a project feasible. 
\end{quote}

We are interested in pursuing the same idea, but taking
\emph{conditional probability} as a primitive notion and obtain
unconditional probability by specialisation.  Taking conditional probability as primitive has a long tradition which
dates back at least to \cite{DeFinetti1935} and
includes \cite{Jeffreys1961,Popper,Renyi,vanFraassen76}. The key
justification for doing this lies in the methodological view that no assessment of
probability takes place in a vacuum. On the contrary, each
probabilistic evaluation must be done in the light of all and only the
available evidence. In this sense, any probabilistic assessment of uncertainty is always conditional.

The first step in achieving our goal is to clarify how
\emph{conditional} knowledge and information should be
represented. To do this we put forward a structure for representing
conditional events, taken as the primitive objects of uncertainty quantification. In other words we aim to capture the logic/algebra which plays the role 
of classical logic when the focus of probability theory is shifted on
conditional probability. In our preliminary investigations \cite{FGH15,Flaminio2017} on the
subject we suggested taking the   methodological approach of
asking the following questions:

\begin{itemize}
\item[(i)] which properties of conditional probabilities depend on
  properties of the {\em measure} and do not depend on the logical
  properties of conditional events?
\item[(ii)] which properties do
  instead depend on the {\em logic} -- whatever it is -- of
  conditional events?
\end{itemize}

Bruno de Finetti was the first not to take the notion of conditional
events for granted and argued that they cannot be described by
truth-functional classical logic. He expressed this by referring to conditional  events as  \emph{trievents}
\cite{DeFinetti1935,DeFinetti2008}, with the following
motivation. Since, intuitively, conditional events of the form ``$a$
given $b$'' express some form
of hypothetical assertion -- the assertion of the consequent $a$  based
\emph{on the supposition} that the antecedent $b$ is satisfied --  the
logical evaluation of a conditional amounts to a two-step procedure. We first  check the antecedent. If this is not satisfied, the conditional ceases
to mean anything at all. Otherwise we move on to evaluating the
consequent and the conditional event takes the same value as the
consequent. 

This interpretation allowed de Finetti to use the classical notion
of uncertainty resolution for conditional events implicitly assumed
by Hausdorff and Kolmogorov, except for the fact that de Finetti
allowed the evaluation of conditional events to be a \emph{partial}
function. This is illustrated clearly by referring to the \emph{betting
  interpretation} of subjective probability, which indeed can be
extended to a number of coherence-based measures of uncertainty \cite{Flaminio2014,
  Flaminio2015}. To illustrate this, fix an
{\em uncertainty resolving valuation} $v$, or in other words a two-valued classical logic
valuation. Then de Finetti interprets
conditional events ``$\theta$ given $\phi$'' as follows:
\begin{equation*}
\text{ a bet on  ``}\theta \text{ given } \phi  \text {'' is }
\begin{cases}
\text{ won  } & \text{ if }  v(\phi)=v(\theta)=1;\\
\text{ lost } & \text{ if  }  v(\phi)=1 \text { and } v(\theta)=0;\\
\text{ called-off }  & \text{ if  }  v(\phi)=0.
\end{cases}
\end{equation*}

This idea has been developed in uncertain reasoning, with links with
non monotonic reasoning, in \cite{DP94, Kern, Kern2014,Kern2018}.  In the context of probability logic, this approach has been pursued in detail
in \cite{CS02}. Note that this latter approach is
measure-theoretically oriented, and yet the semantics of conditional
events is three-valued. The algebra of conditional
  events developed in the present paper, on the contrary, will be
  a Boolean algebra. Hence, as we will point out in due time, the
three-valued semantics of conditional events is not incompatible with
requiring that conditional events form a Boolean algebra. What makes
this possible is that uncertainty-resolving valuations no longer
correspond  to classical  logic valuations,  as in de Finetti's work.
Rather, as it will be clear from our algebraic analysis,  they will
correspond to \emph{finite total orders} of valuations of
classical logics. This crucially allows for the formal representation of the ``gaps'' in uncertainty resolution which arise when the antecedent of a
conditional is evaluated to 0, forcing the bet to be called off.

Some readers may be familiar with the copious and multifarious
literature spanning philosophical logic, linguistics and psychology
which seeks to identify, sometimes probabilistically, how
``conditionals'' depart from Boolean (aka material)
implication. Within this literature emerged a view according to which
conditional probability can be viewed as the probability of a suitably
defined conditional. A detailed comparison with a proposal, due
to Van Fraassen, in this spirit will be done in Subsection \ref{Sub:Comparison2}. However it may be
pointed out immediately that a key contribution of this literature has been the
very useful argument, due to David Lewis \cite{Lewis1976}, according to which the
conditioning operator ``$|$'' cannot be taken, on pain of trivialising
probability functions, to be a Boolean connective, and in particular
material implication. This clearly
reinforces the view, held since de Finetti's early contributions,
that conditional events have their own algebra and logic. A key
contribution of this paper is to argue that this role can be played by
what we term Boolean Algebras of Conditionals (BAC).  Armed with these
algebraic structures, we can proceed to investigate the relation between
conditional probabilities and (plain) 
probability measures on Boolean Algebras of Conditionals. 
In particular we construct, for each
positive probability measure on a finite Boolean algebra, its
canonical extension to a Boolean Algebra of Conditionals which
coincides with the conditional probability on the starting algebra. Hence we
provide a formal setting in which the probability
of conditional events can be regarded as conditional probability. This contributes to a long-standing question which
has been put forward, re-elaborated and discussed by many authors
along the years, and whose general form can be roughly stated as
follows: {\em conditional probability is the probability of
  conditionals} \cite{Adamas1965, Lewis1976,Stalnaker1970,vanFraassen76,GN94,Kau}.\footnote{
To some extent it can be regarded as a simplified version of Adams's  thesis \cite{Adamas1965,Adams1975}, claiming that the {\em assertibility} of a conditional $ a \Rightarrow b$ {\em correlates} 
%or moderately 
%64
%correlates 
with the conditional probability $P(b \mid a)$ of the consequent $b$
given the antecedent $a$. A more concrete statement of this thesis was
put forward by Stalnaker by equating Adam's notion of assertability
with that of probability:  $P(a\Rightarrow b) = P(a\wedge b)/{P(a)}$, whenever $P(a)>0$, known in
the literature as \emph{Stalnaker's thesis} \cite{Stalnaker1970}.
}

In the late 1960 logic-based Artificial Intelligence started to encompass
qualitative uncertainty. First through the notion of negation-as-failure in logic programming,
then with the rise of non-monotonic logics, a field which owes
substantially to the 1980 double special issue of \emph{Artificial
Intelligence} edited by D.G. Bobrow. Much of the
following decade was devoted to identifying \emph{general patterns} in
non monotonic reasoning, in the felicitous turn of phrase due to David
Makinson \cite{Makinson1994}. One prominent such pattern emerged from
the semantical approach put forward by Shoham
\cite{Shoham1987}. According to it, a sentence $\theta$ is a
non-monotonic consequence of a sentence $\phi$ if $\theta$ is
(classically) satisfied by all \emph{preferred or most normal models} of
$\phi$. This equipped the syntactic notion of defaults --
i.e. conditionals which are taken to be defeasibly true -- with a
natural semantics: defaults are conditionals which are ``normally"
true, where normality is captured by suitably ordering classical models. 
Ordered models have then been the key to providing remarkable unity \cite{LM92}  to non-monotonic reasoning, which by the early
2000s encompassed not only a variety of default logics
\cite{Makinson2005}, but also AGM-style theory revision
\cite{Makinson1993} and social choice theory \cite{Rott2001}.

In light of all this, it is noteworthy that the Boolean Algebras of
Conditionals lend themselves to an axiomatisation which turns
out to be sound and complete with respect to a class of preferential
structures. The details are deferred to  Section \ref{sec:logic},
where in addition we show that the logic of boolean conditionals
therein defined satisfies
 the properties of preferential non-monotonic consequence relations, in the sense pioneered by the seminal paper
\cite{KLM} and refined by \cite{LM92}.  In spite of its technical
simplicity, we think this result is methodologically very
significant for it provides strong reasons in support of the very
definition of the Boolean Algebra of Conditionals. In other words, the
fact that its logical counterpart leads to the most widely
investigated framework for nonmonotonic reasoning, justifies our
interpretation of the algebra investigated in Section \ref{sec:BAC} as
the algebra of \emph{conditionals}.

\subsection*{Structure and summary of contributions of the paper}
\label{sec:summary}

The paper is structured as follows. After this introduction and recalling some basic facts about Boolean algebras in Section \ref{sec:algpre}, we present 
in Section \ref{sec:BAC}  the main construction which allows us to
define, starting from any Boolean algebra ${\bf A}$ of events, a
corresponding {\em Boolean algebra of conditional events}
$\mathcal{C}({\bf A})$. These algebras $\mathcal{C}({\bf A})$, whose
elements are  objects of the form $(a \mid b)$ for $a,b\in A$ and their boolean
combinations,  are finite if the original algebras $\bf A$ are so.

Section \ref{sec:atoms} is dedicated to 
the atomic structure of Boolean algebras of conditionals. The main
result is a full  characterization of the atoms of each finite
$\mathcal{C}({\bf A})$ in terms of the atoms of ${\bf A}$. This
characterization is a fundamental step for the
rest of 
the paper. Further elaborating on the atomic
structure, Section \ref{sec:tree-like} presents two {\em tree-like}
representations for the set of atoms of an algebra of conditionals
which  will be decisive in establishing the main result of the  paper in Section \ref{sec:measures0}.

In fact, Section \ref{sec:measures0} introduces  
probability measures on Boolean algebras of conditionals and presents 
our  main result to the effect  
that every positive probability $P$ on a finite Boolean algebra ${\bf
  A}$ can be canonically extended to a positive probability $\mu_P$ on
$\mathcal{C}({\bf A})$ which agrees with the ``conditionalised''
version of the former.  
That is, we prove that, for every $a, b \in  A$ with $b\neq \bot$, 
$$
\mu_P(`(a \mid b)\textrm{\rq})  = P(a \land b)/P(b).
$$

As a welcome consequence of our investigation we provide an
alternative, finitary, solution 
to the problem known in the literature as {\em the strong conditional
  event problem}, introduced and solved in the infinite setting of the Goodman and Nguyen's
Conditional Event Algebras of   \cite{GN94}.

Although our 
Boolean algebras of conditionals do not allow for an equational description, the characterizing properties of these algebras are expressible in an expansion of the language of classical propositional logic and hence they give rise naturally to a simple logic of (non-nested) conditionals, that we name LBC (for {\em Logic of Boolean Conditionals}). 
This is investigated in  Section \ref{sec:logic}, where we axiomatize the logic and prove  soundness and completeness with respect to a class of preferential structures. Moreover, we show that LBC satisfies
 the properties of preferential non-monotonic consequence relations, in the sense pioneered by the seminal paper
\cite{KLM} and refined by \cite{LM92}.  
Finally, Section \ref{sec:related} draws some detailed comparisons
between our contributions and the research on Measure-free
conditionals (Subsection \ref{sec:measure-free}) and with Conditional
Event Algebras  (Subsection \ref{Sub:Comparison2}).
Section \ref{sec:conclusion} outlines a set of key issues for future work. 

To facilitate the reading of the paper, most proofs
are relegated to an appendix.

\section{Preliminaries: Boolean algebras in a nutshell}\label{sec:algpre}
The algebraic framework of this paper is that of Boolean algebras and hence its logical setting is that of classical propositional logic (CPL). Here, we will briefly recap on some needed notions and basic results about Boolean algebras and CPL, for a more exhaustive introduction about this subject we invite the reader to consult \cite[\S IV]{BS}, and \cite{CoriLascar,GH09,Halmos}.

Given a countable (finite or infinite) set $V$ of {\em propositional variables}, the CPL language $\mathsf{L}(V)$ (or simply $\mathsf{L}$ when $V$ will be clear by the context)  is the 
smallest set containing $V$ and closed under the usual connectives $\wedge, \vee, \neg, \bot,$ and $\top$ of type  $(2,2,1,0,0)$. Along this paper we will use the notation  $\varphi, \psi$, etc (with possible subscript) for formulas. Further, we shall adopt the following abbreviations:
\begin{center}
$\varphi\to \psi=\neg \varphi\vee \psi$, $\varphi\leftrightarrow \psi=(\varphi\to \psi)\wedge(\psi\to \varphi)$. 
\end{center}
We shall denote by $\vdash_{CPL}$ the provability relation of CPL, in particular we will write $\vdash_{CPL}\varphi$ to denote that $\varphi$ is a theorem.

A {\em logical valuation} (or simply a {\em valuation}) of $\mathsf{L}$ is a map from $v:V \to \{0,1\}$, which uniquely extends to a function, that we denote by the same symbol $v$, from $\mathsf{L}$ to $\{0,1\}$ 
in accordance with the usual Boolean truth functions, i.e.
 $v(\varphi\wedge\psi)=\min\{v(\varphi),v(\psi)\}$, $v(\bot)=0$, $v(\neg \varphi)=1-v(\varphi)$, etc. We shall denote by $\Omega$ the set of all valuations of $\mathsf{L}$. {For a given formula $\varphi$ and a given valuation $v\in \Omega$, we will write $v\models \varphi$  whenever $v(\varphi)= 1$.}

We will broadly adopt, analogously to the above recalled logical frame, the signature $(\wedge, \vee, \neg, \bot, \top)$ of type $(2,2,1,0,0)$ for the algebraic language upon which Boolean algebras are defined. Thus, the same conventions and abbreviations of $\mathsf{L}$ can be adopted also in the algebraic setting. Further, in every Boolean algebra ${\bf A} = (A, \wedge, \vee, \neg, \bot, \top)$
we shall write $a\leq b$, whenever $a\to b=\top$. The relation $\leq$ is indeed the lattice-order in ${\bf A}$. Thus, $a\leq b$ iff $a\wedge b=a$ iff $a\vee b=b$.

 Along this paper, in order to distinguish an algebra from its universe, we will denote the former by ${\bf A}$, ${\bf B}$ etc, and the latter by $A$, $B$ etc, respectively. 
 
Recall that a map $h:{\bf A}\to{\bf B}$ between Boolean algebras is a {\em homomorphism} if $h$ commutes with the operations of their language, that is, $h(\top_{\bf A})=\top_{\bf B}$, $h(\neg_{\bf A}a)=\neg_{\bf B} h(a)$, $h(a\wedge_{\bf A}b)=h(a)\wedge_{\bf B}h(b)$ etc, (notice that we adopt subscripts to distinguish the operations of ${\bf A}$ from those of ${\bf B}$). Bijective (or 1-1)  homomorphisms are called {\em isomorphisms} and if there is a isomorphism between ${\bf A}$ and ${\bf B}$, they are said to be {\em isomorphic} (and we write ${\bf A}\cong{\bf B}$).

A {\em congruence} of a Boolean algebra ${\bf A}$ is an equivalence relation $\equiv$ on $A$ which is compatible with its operations (see \cite[\S17]{GH09}), that is, 
for every $a, a', b, b' \in A$, if $a \equiv a'$ and $ b\equiv b'$ 
then $\neg a\equiv \neg a', a\wedge b \equiv  a'\wedge b'$ and  $a\vee b \equiv a'\vee b'$. The compatibility property allows us to equip the set $A/_\equiv = \{ [a] \mid a \in A\}$ of equivalence classes with operations inherited from ${\bf A}$, endowing ${A}/_\equiv$ with a structure of Boolean algebra, written ${\bf A}/_\equiv$, and which is called the {\em quotient} of ${\bf A}$ {\em modulo} $\equiv$. For a later use, we further recall that for all $a, a'\in A$ such that $a\equiv a'$, the equality $[a]=[a']$ holds in ${\bf A}/_\equiv$.
Recall that for any subset $X \subseteq A \times A$, the {\em congruence generated by $X$} is the smallest congruence $\equiv_X$ which contains $X$. The congruence $\equiv_X$ always exists \cite[\S5]{BS}.

Boolean algebras form a variety, i.e.\ an equational class,  in which, for any (countable) set $V$, the {\em free $V$-generated} algebra ${\bf Free}(V)$ 
(see \cite[\S II]{BS}) is isomorphic to the Lindenbaum algebra of CPL over a language whose propositional variables belong to $V$ (see for instance \cite{BP}). Since these structures will play a quite important role in the main construction we will introduce in Section \ref{sec:BAC}, let us briefly recap on them. Given any set $V$ of propositional variables, we denote by $\mathsf{L}(V)/_{\equiv}$ the set of equivalence classes of formulas of the language $\mathsf{L}(V)$ %generated by  $V$ 
modulo the congruence relation $\equiv$ of {\em equi-provability}, i.e.,  two formulas $\varphi$ and $\psi$ are equi-provable iff $\vdash_{CPL}\varphi\leftrightarrow \psi$.  The algebra ${\bf L}(V)=(\mathsf{L}(V)/_{\equiv}, \wedge, \vee, \neg, \bot, \top)$ is a Boolean algebra called the Lindenbaum algebra of CPL over the language $\mathsf{L}(V)$. Therefore, a map $v: \mathsf{L}(V) \to \{0, 1\}$ is a valuation iff it is a homomorphism of ${\bf L}(V)$ into ${\bf 2}$ (where $\bf 2$ denotes the Boolean algebra of two elements $\{0,1\}$).

\begin{definition}\label{def:atomsAlgebra}
An element $a$ of a Boolean algebra ${\bf A}$ is said to be an {\em atom} of ${\bf A}$ if $a>\bot$ and for any other element $b \in A$ such that $a\geq b\geq \bot$, either $a=b$ or $b=\bot$.  
\end{definition}
For every algebra ${\bf A}$, we shall henceforth denote by $ \atom({\bf A})$ the set of its atoms and we will denote its elements by $\alpha,\beta,\gamma$ etc. If $ \atom({\bf A}) \neq \emptyset$, ${\bf A}$ is called {\em atomic}, otherwise  
${\bf A}$ is said to be {\em atomless}. If $\bf A$ is finite, then it is atomic.
In particular, if $V$ is finite, the Lindenbaum algebra ${\bf L}(V)$ is finite as well and  thus atomic \cite{Bez}. In fact, if $| \cdot |$ denotes the cardinality map, if  $| V | = n$ then $| \atom({\bf L}(V)) | = 2^n$, and $|{\bf L}(V) | = 2^{2^n}$.
The following proposition collects well-known and needed facts about atoms (see e.g. \cite[\S I]{BS} and \cite[\S 16]{Halmos}). It recalls, among other things, that  $ \atom({\bf A})$  is  a  partition of an atomic algebra $\bf A$. Recall that a {\em partition} of a Boolean algebra is a collection of pairwise disjoint elements different from $\bot$ whose supremum is $\top$. 

\begin{proposition}\label{prop:atomsAlgebra}
Every finite Boolean algebra is atomic. Further, 
for every finite Boolean algebra ${\bf A}$ the following hold:
\begin{enumerate}[(i)]
\item for every $\alpha,\beta\in \atom({\bf A})$, $\alpha\wedge\beta=\bot$;
\item for every $a\in A$, $a=\bigvee_{\alpha\leq a}\alpha$. Thus, in particular, $\bigvee_{\alpha\in \atom(\bf A)} \alpha=\top$;
\item for each $\alpha\in \atom({\bf A})$, the map $h_\alpha: {\bf A}\to {\bf 2}$ such that $h_\alpha(a)= 1$ if $a\geq \alpha$ and 
$h_\alpha(a)= 0$ otherwise,
is a homomorphism. Furthermore, the map $\lambda:\alpha \mapsto h_\alpha$ is a 1-1 correspondence between $\atom({\bf A})$ and the set of homomorphisms of ${\bf A}$ in ${\bf 2}$;
\item if ${\bf A} = {\bf L}(V)$ with $V$ finite, the map $\lambda$ as in {\rm(iii)} is a 1-1 correspondence between the atoms of ${\bf L}(V)$ and the set $\Omega$ of valuations of $V$.
\end{enumerate}
Moreover, a subset $X=\{x_1,\ldots, x_m\} \subseteq A$ coincides with $\atom({\bf A})$ iff the following two conditions are satisfied:
\begin{enumerate}
\item[(a)] $X$ is a partition of $\bf A$ (i.e.\ $x_i\wedge x_j=\bot$ if $i\neq j$, and $\bigvee_{i=1}^m x_i=\top$);
\item[(b)] every $x_i\in X$ is such that $\bot<x_i$ and there is no $b\in A$ such that $\bot<b<x_i$.
\end{enumerate}
\end{proposition}

\section{Boolean algebras of conditionals}\label{sec:BAC}
In this section we introduce the notion of Boolean algebras of conditionals and prove some basic properties.  For any Boolean algebra ${\bf A}$, the construction we are going to present builds a Boolean algebra of conditionals that we shall denote by $\mathcal{C}({\bf A})$. 
In the following, given a Boolean algebra ${\bf A}$, we will write $A'$ for $A\setminus \{\bot\}$.

Intuitively, in a Boolean algebra of conditionals over ${\bf A}$ we will allow {\em basic conditionals}, i.e.\ objects of the form $(a \mid b)$ for $a \in A$ and $b \in A'$, to be freely combined with the usual Boolean operations 
up to certain extent.
Recall from the introduction that our main
  goal is to distinguish, as far as this is possible, the properties
  of the uncertainty measure from the algebraic properties of
  conditionals. This means that we must  pin down properties which make sense in the context of
  conditional reasoning under uncertainty. Those properties are summed up in the
  following four informal requirements, which guide our construction. 
 
\begin{enumerate}
\item[R1] For every $b\in A'$, the conditional $(b \mid b)$ will be the top element of $\mathcal{C}({\bf A})$, while $(\neg b \mid b)$ will be the bottom; 
\item[R2] Given $b \in A'$, the set of conditionals $A \mid b = \{(a \mid b) : a \in A\}$ will be the domain of a Boolean subalgebra of $\mathcal{C}({\bf A})$, and in particular when $b = \top$, this subalgebra will be isomorphic to ${\bf A}$;
\item[R3] In a conditional $(a \mid b)$ we can replace the consequent $a$ by $a \land b$, that is,  we require the conditionals $(a \mid b)$ and $(a \land b \mid b)$ to represent the same element of $\mathcal{C}({\bf A})$; 
\item[R4] For all $a\in A$ and all $b,c\in A'$, if $a \leq b \leq c$, then the result of conjunctively combining the conditionals $(a \mid b)$ and $ (b \mid c)$ must yield the conditional $(a \mid c)$. 
\end{enumerate} 

Whilst conditions R1-R3 do not require delving into particular justifications,
  it is worth noting that R4  encodes a sort of restricted chaining of conditionals and it is inspired by the chain rule of conditional probabilities: $P(a \mid b) \cdot P(b \mid c) = P(a \mid c)$ whenever $a \leq b \leq c$. 

Given these four requirements, the formal construction of the algebra $\mathcal{C}({\bf A})$ 
 is done in three steps described next.\footnote{Our construction is
   inspired by the one by Mundici for \emph{algebraic tensor products} \cite{mundici}.}

The first one is to consider the set of objects $\AmA  = \{ (a \mid b) : a \in A, b \in A' \}$ and the  algebra
$${\bf Free}(\AmA) = (Free(\AmA), \sqcap,\sqcup,\sim,\bot^{*},\top^{*}). $$
Recall from Section \ref{sec:algpre} that ${\bf Free}(\AmA)$ is (up to isomorphism) the Boolean algebra whose elements are equivalence classes (modulo equi-provability) of Boolean terms generated by all pairs $(a \mid b)\in \AmA$ taken as propositional variables.
In other words, in ${\bf Free}(\AmA)$ two Boolean terms can be identified (i.e.\ they belong to the same class) only if one term can be rewritten into the other one by  using only the laws of Boolean algebras. For instance $(a \mid b) \sqcap (c \mid b)$ and $(c \mid b) \sqcap (a \mid b)$ clearly belong to the same class in ${\bf Free}(\AmA)$, but $(a \land c \mid b)$ does not, a fact that is not in agreement with requirement R2.

Therefore, in a second step, in order to accommodate the requirements R1-R4 above we need to identify more classes in ${\bf Free}(\AmA)$. In particular,
we would like $(a \land c \mid b)$, $(a \mid b) \sqcap (c \mid b)$ and $(c \mid b) \sqcap (a \mid b)$ to represent the same element in the algebra $\mathcal{C}({\bf A})$.
Thus, to enforce this and all the other desired identifications in ${\bf Free}(\AmA)$, we consider the congruence relation on ${\bf Free}(\AmA)$ generated by the subset 
$\mathfrak{C} \subseteq Free(\AmA) \times Free(\AmA)$ 
containing the following pairs of terms: 
\begin{itemize}
\item[(C1)] $((b \mid b),\top^*)$, for all $b\in A'$;
\item[(C2)] $((a_1 \mid b)\sqcap (a_2 \mid b), (a_1\wedge a_2 \mid b))$, for all $a_1, a_2\in A$, $b\in A'$;
\item[(C3)] $(\sim\!(a \mid b), (\neg a \mid b))$, for all $a\in A$, $b\in A'$;
\item[(C4)] $((a\wedge b \mid b), (a \mid b))$, for all $a\in A$, $b\in A'$;
\item[(C5)] $((a\mid b) \sqcap (b \mid c), (a \mid c))$, for all $a\in A$, $b,c\in A'$ such that $a \leq b \leq c$.
\end{itemize}
Note that (C1)-(C5) faithfully account for the requirements R1-R4 where, in particular,  (C2) and (C3) account for R2. 
In particular, observe that, continuing the discussion above, now the elements $(a \land c \mid b)$, $(a \mid b) \sqcap (c \mid b)$ and $(c \mid b) \sqcap (a \mid b)$ belong to the same class under the equivalence $\equiv_{\mathfrak{C}}$.

Then,  we finally propose the following definition.
\begin{definition}\label{def:BAC}
For every Boolean algebra ${\bf A}$, we define the {\em Boolean algebra of conditionals} of ${\bf A}$ as the quotient structure 
$$
\mathcal{C}({\bf A})={\bf Free}(\AmA)/_{\equiv_\mathfrak{C}}.
$$ 
\end{definition}

Note that, by construction, if $\bf A$ is finite, so is $\mathcal{C}({\bf A})$. For the sake of an unambiguous notation, we will henceforth distinguish the operations of ${\bf A}$ from those of $\mathcal{C}({\bf A})$ by adopting the following signature: 
$$
\mathcal{C}({\bf A})=(\mathcal{C}(A),\sqcap,\sqcup,\sim,\bot_{\mathfrak{C}},\top_{\mathfrak{C}}).
$$

\begin{remark}[Notational convention]\label{notation}
Since $\mathcal{C}({\bf A})$ is a {\em quotient} of ${\bf Free}(\AmA)$, its generic element is a class $[t]_{\equiv_{\mathfrak{C}}}$, for $t$ being  a Boolean term, whose members are  equivalent to $t$ under $\equiv_{\mathfrak{C}}$. 
For the sake of a clear notation and without danger of confusion, we will henceforth identify $[t]_{\equiv_{\mathfrak{C}}}$ with one of its representative elements and, in particular, by  $t$ itself. 
Given two elements $t_1, t_2$ of $\mathcal{C}({\bf A})$, we will write $t_1=t_2$ meaning that $t_1$ and $t_2$ determine the same equivalence class of $\mathcal{C}({\bf A})$ or, equivalently, that $t_1\equiv_{\mathfrak{C}} t_2$. \qed
\end{remark}

It is then clear that, using the above notation convention, the following equalities, which correspond to (C1)--(C5) above, hold in any Boolean algebra of conditionals $\mathcal{C}({\bf A})$.
\begin{proposition}\label{prop2} 
Any Boolean algebra of conditionals $\mathcal{C}({\bf A})$ satisfies the following properties for all $a, a'\in A$ and $b,c\in A'$: 
\begin{enumerate}[(i)]
\item\label{e1} $(b\mid b)=\top_\mathfrak{C}$;
\item\label{e2}  $(a\mid b)\sqcap  (c\mid b)=(a\wedge c\mid b)$;
\item\label{e3} $\sim\! (a\mid b)=(\neg a\mid b)$;
\item\label{e4}  $(a\wedge b\mid b)=(a\mid b)$;
\item\label{e8} if $a\leq b\leq c$, then $(a\mid b)\sqcap (b\mid c)=(a\mid c)$. 
\end{enumerate}
\end{proposition}

Straightforward consequences of (\ref{e4}) and (\ref{e8}) above are the following. 

\begin{corollary}\label{cor:props} \mbox{ }
\begin{enumerate}[(i)]
\item $(b \to a \mid b) = (a \mid b)$;
\item $(a \land b \mid \top)=(a\mid b)\sqcap (b\mid \top)$;
\item $(a \land b \mid c) = (a\mid b \land c) \sqcap (b \mid c)$.
\end{enumerate}
\end{corollary}
Notice that (iii) above corresponds to the qualitative version of axiom CP3 of \cite[Definition 3.2.3]{Halpern2003}.

It is convenient to distinguish the elements of $\mathcal{C}({\bf A})$ in {\em basic} and {\em compound} conditionals. The former are expressions of the form $(a\mid b)$, while 
the latter are those terms $t$ which are (non trivial) Boolean combination of basic conditionals but which are not equivalent  modulo $\equiv_{\mathfrak{C}}$ (and hence not {\em equal}), to any element of $\mathcal{C}({\bf A})$ of the form $(a\mid b)$.  
For instance, if $b_1\neq b_2\in A'$ there is no general rule, among (C1)--(C5) above, which allows us to identify in $\mathcal{C}({\bf A})$ the term $(a\mid b_1)\sqcap(a\mid b_2)$ with a basic conditional of the form $(x\mid y)$ whilst, 
the term $(a_1\mid b)\sqcap(a_2\mid b)$ coincides in $\mathcal{C}({\bf A})$ with the basic conditional $(a_1\wedge a_2\mid b)$, as required by (C2). 

\begin{example}\label{ex:fourelements}
Let us consider the four elements Boolean algebra ${\bf A}$ 
whose domain is $\{\top,a,\neg a,\bot\}$. 
Then, $\AmA=\{(\top,\top)$, $(\top,a)$, $(\top,\neg a)$, $(a,\top)$, $ (a,a)$, $(a,\neg a)$, $(\neg a,\top)$, $(\neg a,a)$, $ (\neg a,\neg a)$, $(\bot,\top)$, $(\bot,a)$, $(\bot,\neg a)\}$ has
cardinality $12$ and $\mathbf{Free}(\AmA)$ is the free Boolean algebra of  $2^{2^{12}}$ elements, i.e. the finite Boolean algebra of $2^{12}$ atoms.
However, 
in $\mathcal{C}({\bf A})$ the following equations hold (and the conditionals below are hence identified): 
\begin{enumerate}
\item $\top_{\mathfrak{C}} = {(\top\mid \top)}= (a \mid \top) \sqcup (\neg a \mid \top) = (\top \mid a)=  {(a\mid a)}={(\neg a\mid \neg a)}$; 
\item $({\top\mid \top})\sqcap ({a\mid \top})={(\top\wedge a\mid\top)}=({a\mid \top}) = \;\sim\!(\neg a \mid \top)$; 
\item $({\top\mid \top})\sqcap ({\neg a\mid \top})={(\top\wedge \neg a\mid\top)}=({\neg a\mid \top}) = \;\sim\!(a \mid \top)$; 
\item $\bot_{\mathfrak{C}} = \msim ({\top\mid\top})= ({\bot\mid\top}) =(a\mid \top)\sqcap(\neg a\mid \top)=(a\mid \top)\sqcap\msim(a\mid\top)$. 
\end{enumerate}

Thus, it is easy to see that $\mathcal{C}({\bf A})$ contains only four elements that are not redundant under $\equiv_\mathfrak{C}$: $(\top\mid\top), (a\mid\top),(\neg a\mid\top),(\bot\mid\top)$. 
As we will show in Section \ref{sec:atoms} (see Theorem \ref{th:atoms}) 
$\mathcal{C}(\bf A)$ has $2$ atoms and it is indeed isomorphic to ${\bf A}$. \qed
\end{example}

Next, we present some further basic properties of Boolean algebras of conditionals which are not immediate from the construction. However, since their proofs are essentially trivial, we also omit them.

\begin{proposition}\label{prop3}
The following conditions hold in every Boolean algebra of conditionals $\mathcal{C}({\bf A})$:
\begin{enumerate}[(i)]
\item\label{e5} for all $a,c\in A$,    $(a\mid\top)=(c\mid\top)$ iff $a= c$;
\item\label{e6} for all $b\in A'$, $(\neg b\mid b)=\bot_\mathfrak{C}$; 
\item\label{e7} for all $a,c\in A$, and $b\in A'$, $(a\mid b)\sqcup (c\mid b)=(a\vee c\mid b)$;
\end{enumerate}
\end{proposition}

For every fixed $b\in A'$, we can now consider the set $A\mid b=\{(a\mid b)\mid a\in A\}$  of all conditionals having $b$ as antecedent. 
The following  is  an immediate consequence of (i-iii) of Proposition \ref{prop2} and Proposition \ref{prop3} \eqref{e7} above.

\begin{corollary}\label{subalgebra}
 For every  algebra $\mathcal{C}({\bf A})$ and for every $b\in A'$ the structure ${\bf A} \mid b=(A \mid b, \sqcap, \sqcup, \neg, \bot_{\mathfrak{C}}, \top_{\mathfrak{C}})$ is a Boolean subalgebra of $\mathcal{C}({\bf A})$. In particular, the algebra ${\bf A}\mid \top$ is isomorphic to ${\bf A}$.
\end{corollary}

As in any Boolean algebra, the lattice order relation in $\mathcal{C}({\bf A})$, denoted by $\leq$,  is defined as follows: for every $t_1, t_2\in \mathcal{C}({\bf A})$,
$$
t_1\leq t_2\mbox{ iff }t_1\sqcap t_2=t_1\mbox{ iff }t_1\sqcup t_2=t_2.
$$

The following propositions  collect some general properties related to the lattice order $\leq$ defined above. 
Nevertheless, some further and stronger properties on the $\leq$-relation between basic conditionals will be provided at the end of Section \ref{sec:atoms}, once the atomic structure of the algebras of conditionals $\mathcal{C}({\bf A})$ will be characterised in that section.

\begin{proposition}\label{prop:ordine}
In every  algebra $\mathcal{C}({\bf A})$ the following properties hold for  every $a,c\in A$ and $b\in A'$:
\begin{enumerate}[(i)]
\item\label{p1} $(a\mid b)\geq (b\mid b)$ iff $a\geq b$;
\item\label{p2}  if $a\leq c$, then $(a\mid b)\leq (c\mid b)$; in particular $a\leq c$ iff $(a\mid \top)\leq (c\mid\top)$; 
\item \label{p3} if $a \leq b \leq d$, then $ (a \mid b) \geq (a \mid d)$; in particular $(a \mid b) \geq (a \mid a \lor b)$; 
\item\label{p5}  if $(a\mid b)\neq(c\mid b)$, then $a\land b \neq c \land b$;
\item \label{p6}  $(a\wedge b\mid \top)\leq(a \mid b) \leq (b \to a \mid \top)$;
\item\label{p7} if  $a \land d = \bot$ and $ \bot < a \leq b$, then $(a \mid \top) \sqcap  (d \mid b) = \bot_\mathfrak{C}$; 
\item\label{p8}  $(b \mid \top) \sqcap (a \mid b) \leq (a \mid \top) $;
\end{enumerate}
\end{proposition}
\begin{proof}
See Appendix. 
\end{proof}

Some properties in the proposition above have a clear logical reading. For instance, (\ref{p6}) tells us that in a Boolean algebra of conditionals, a basic conditional $(a\mid b)$ is a weaker construct than the conjunction $a\wedge b$ but  stronger than the material implication $b \to a$, in accordance to previous considerations in the literature, see e.g.\ \cite{DP94}. As a consequence, this suggests that  
a conditional $(a \mid b)$ can be evaluated to {\em true} when both $b$ and $a$ are so (i.e. when $a \land b$ is true), while $(a \mid b)$ can be evaluated as {\em false} when $a$ is false and $b$ is true (i.e. when falsifying $b \to a$). Furthermore, \eqref{p8} can be read as a form of {\em modus ponens} with respect to conditional expressions: from $b$ and $(a \mid b)$ it follows $a$. 
We refer this discussion on logical issues of conditionals to Section \ref{sec:logic} in which we will introduce and study a logic of conditionals and where we will propose a formal definition of {\em truth} for them.

We now end this section presenting a few further properties of Boolean algebras of conditionals regarding the disjunction in the antecedents.
\begin{proposition}\label{prop5}
In every  algebra $\mathcal{C}({\bf A})$ the following properties hold for all $a, a'\in A$ and $b, b'\in A'$:
\begin{enumerate}[(i)]
\item\label{p12} $(a \mid b) \sqcap (a \mid b') \leq (a \mid b \lor b')$; 
in particular, $(a \mid b) \sqcap (a \mid \neg b) \leq (a \mid \top)$;  
\item \label{p11}  if $a \leq b \land b'$, then $(a \mid b) \sqcap (a \mid b') = (a \mid b \lor b')$; 
\item\label{p13} $(a\mid b)\leq (b\to a\mid b\vee b')$;
\item\label{p14} $(a\mid b)\sqcap (a'\mid b')\leq ((b\to a) \wedge (b'\to a') \mid b\vee b')$.

\end{enumerate}
\end{proposition}

\begin{proof} See Appendix.
 \end{proof}

Observe that the logical reading of  property \eqref{p12} above is the
well-known {\em OR-rule}, typical of nonmonotonic reasoning (see
\cite{Adams1975,DP94}). This fact, although not being particularly
surprising, will be further strengthened in Section \ref{sec:logic}
where we will show that, indeed, Boolean algebras of conditionals
provide a sort of algebraic semantics for a nonmonotonic logic
related to System P. Further,  \eqref{p14} shows that the algebraic
conjunction of two basic conditionals is stronger than the operation
of {\em quasi-conjunction} introduced in the setting of measure-free
conditionals (see \cite{Adams1975} and \cite[Lemma 2]{DP94}) and recalled in Subsection \ref{sec:measure-free}. Also notice that, the point (iv) above, in the special case in which $a'=b$, $b'=c$ and $a\leq b\leq c$ actually gives $(a\mid b)\sqcap (b\mid c)=(a\mid c)=((b\to a)\wedge(c\to b)\mid b\vee c)$.  Therefore,  the requirement R4 is in agreement with the definition  of quasi-conjunction.

\section{The atoms of a Boolean algebra of conditionals}\label{sec:atoms}

As we already noticed, if ${\bf A}$ is finite, $\mathcal{C}({\bf A})$ is finite as well and hence atomic.
This section is devoted to investigate the atomic structure of finite Boolean algebras of conditionals. In particular, in Subsection \ref{sub41} we  provide a characterization of the atoms of $\mathcal{C}({\bf A})$ in terms of the atoms of ${\bf A}$. That characterization will be employed in Subsections \ref{sec:belowbasic} and \ref{sub43} to give, respectively, a full description of the atoms which stand below a basic conditional $(a\mid b)$ and to prove results concerning equalities  and inequalities among conditionals which improve those of Section \ref{sec:BAC}.

In this section and in rest of the paper, we will only deal with finite Boolean algebras.

\subsection{The atomic structure of $\mathcal{C}({\bf A})$}\label{sub41}

Let us recall the notation introduced in Section \ref{sec:algpre}: for every  Boolean algebra ${\bf A}$, we denote by $\atom({\bf A})$ the set of its atoms, that will be denoted by lower-case greek letters, $\alpha, \beta, \gamma$ etc.

 \begin{proposition} \label{lemma:repres} In a conditional algebra $\mathcal{C}({\bf A})$, the following   hold:
\begin{enumerate}[(i)]

\item each element $t$ of $\mathcal{C}(\bf A)$ is of the form $t = \bigsqcap_i (\bigsqcup_j (a_{i_j}\mid b_{i_j}))$; 

\item each basic conditional is of the form $(a \mid b) = \bigsqcup_{\alpha \leq a} (\alpha \mid b)$; 

\item in particular, every element of $\mathcal{C}({\bf A})$ is a $\sqcap$-$\sqcup$ combination of basic conditionals in the form $(\alpha\mid \bigvee X)$ where $\alpha\in \atom({\bf A})$ and $X\subseteq \atom({\bf A})$. 
\end{enumerate}
\end{proposition}

\begin{proof}
(i). It readily follows by recalling that (1) every element $t$ of $\mathcal{C}({\bf A})$ is, by construction, a Boolean combination of basic conditionals, (2) it can be expressed in conjunctive normal form, and (3) the negation of a basic conditional $(a \mid b)$ is the basic conditional $(\neg a \mid b)$. 
\vspace{.2cm}

\noindent (ii). The claim directly follows from  Proposition \ref{prop3} \eqref{e7} taking into account that $a = \bigvee_{\alpha \leq a} \alpha$ (recall Proposition \ref{prop:atomsAlgebra} (ii)). 
\vspace{.2cm}

\noindent (iii). It is a direct consequence of (i) and (ii). 
\end{proof}

 Now, let ${\bf A}$ be a Boolean algebra with $n$ atoms, i.e.\ $| \atom({\bf A}) | = n$. 
  For each $i \leq n-1$, let us define $Seq_i({\bf A})$ to be the set of sequences $\langle \alpha_1, \alpha_2, \ldots, \alpha_{i} \rangle$ of $i$ pairwise different elements of $\atom({\bf A})$. 
Thus, for every $\overline{\alpha}=\langle \alpha_1, \alpha_2, \ldots, \alpha_{i} \rangle\in Seq_i({\bf A})$, let us consider the compound conditional of $\mathcal{C}({\bf A})$ defined in the following way:
\begin{equation}\label{eq:atomsDef}
\omega_{\overline{\alpha}}= (\alpha_1 \mid \top) \sqcap  (\alpha_2 \mid \neg \alpha_1) \sqcap \ldots \sqcap
 (\alpha_{i} \mid \neg\alpha_1 \land \ldots \land \neg \alpha_{i-1}) . 
 \end{equation}

Intuitively, such a conjunction of conditonals encodes a sort of chained `defeasible' conditional statements about a set of mutually disjoint events: in principle $\alpha_1$ holds, but if $\alpha_1$ turns out to be false then in principle $\alpha_2$ holds, but if besides $\alpha_2$ turns out to be false as well, then in principle $\alpha_3$ holds, and so on \ldots 

These conjunctions of basic  conditionals will play an important role in describing the atomic structure of $\mathcal{C}({\bf A})$ and enjoy suitable properties. 
To begin with, let us consider sets of those compound conditionals of a given length: for each $1 \leq i \leq n-1$, let  
$$  Part_i(\mathcal{C}({\bf A})) =  \{ \omega_{\overline \alpha}  \mid   \overline{\alpha} \in Seq_i({\bf A}) \}. $$
\begin{example}\label{ex:neww}
Let ${\bf A}$ be the Boolean algebra with $4$ atoms, $\atom({\bf A})=\{\alpha_1,\ldots, \alpha_4\}$. For $i=1$, the set $Part_1(\mathcal{C}(\bf A))$ is easily built by considering all sequences of length 1 of atoms of ${\bf A}$, $Seq_1(\mathcal{C}({\bf A}))= \{\langle \alpha_1\rangle, \langle \alpha_2\rangle, \langle \alpha_3\rangle, \langle \alpha_4\rangle\}$, and hence: 
$$
Part_1(\mathcal{C}({\bf A}))=\{\omega_{\langle \alpha_1\rangle}, \dots, \omega_{\langle \alpha_4\rangle}\}=\{(\alpha_1\mid \top),\ldots, (\alpha_4\mid \top)\}.
$$
For $i = 2$, we have to consider sequences of atoms of length 2, i.e. $$Seq_2(\mathcal{C}({\bf A}))= \{\langle \alpha_1, \alpha_2 \rangle, \langle \alpha_1, \alpha_3,\rangle, \langle \alpha_1, \alpha_4 \rangle, \langle \alpha_2, \alpha_3 \rangle, \ldots \},$$ and then the corresponding set $Part_2(\mathcal{C}({\bf A}))$ is composed by 12 Boolean terms like
\begin{itemize}
\item[] $\omega_{\langle \alpha_1, \alpha_2 \rangle}=(\alpha_1\mid \top)\sqcap (\alpha_2\mid \neg \alpha_1)$; \quad $\omega_{\langle \alpha_1, \alpha_3 \rangle}=(\alpha_1\mid \top)\sqcap (\alpha_3\mid \neg \alpha_1)$;
\item[] $\omega_{\langle \alpha_1, \alpha_4 \rangle}=(\alpha_1\mid \top)\sqcap (\alpha_4\mid \neg \alpha_1)$; \quad $\omega_{\langle \alpha_2, \alpha_3 \rangle}=(\alpha_2\mid \top)\sqcap (\alpha_3\mid \neg \alpha_2)$;
\item[] \ldots
\end{itemize} 
Finally, for $i = 3$, consider the set sequences of atoms of length 3, $$Seq_3(\mathcal{C}({\bf A}))=\{\langle \alpha_1, \alpha_2, \alpha_3 \rangle, \langle \alpha_1, \alpha_2, \alpha_4 \rangle, \langle \alpha_2, \alpha_1, \alpha_3 \rangle,\ldots\}.$$ Therefore, the set $Part_3(\mathcal{C}({\bf A}))$ contains 24 Boolean terms:
\begin{itemize}
\item[] $\omega_{\langle \alpha_1, \alpha_2, \alpha_3 \rangle}=(\alpha_1\mid \top)\sqcap (\alpha_2\mid \neg \alpha_1)\sqcap(\alpha_3\mid \neg \alpha_1\wedge \neg \alpha_2)$;
\item[] $\omega_{\langle \alpha_1, \alpha_2, \alpha_4 \rangle}=(\alpha_1\mid \top)\sqcap (\alpha_2\mid \neg \alpha_1)\sqcap(\alpha_4\mid \neg \alpha_1\wedge \neg \alpha_2)$;
\item[] $\omega_{\langle \alpha_2, \alpha_1, \alpha_3 \rangle}=(\alpha_2\mid \top)\sqcap (\alpha_1\mid \neg \alpha_2)\sqcap(\alpha_3\mid \neg \alpha_2\wedge \neg \alpha_1)$;
\item[] \ldots  \qed
\end{itemize} 
\end{example}

The following result shows that, for each $i$, 
$Part_i(\mathcal{C}({\bf A}))$ is a partition of  $\mathcal{C}({\bf A}) $(recall Section \ref{sec:algpre}), and the higher  the index $i$, the more refined the partition is.
 
 \begin{proposition} \label{useful} 
$Part_i(\mathcal{C}({\bf A}))$ is a {\em partition} of $\mathcal{C}({\bf A})$. 
\end{proposition}

\begin{proof}  See Appendix.
\end{proof}

As we already saw in Example \ref{ex:neww},  
if $\atom({\bf A})=\{\alpha_1,\ldots, \alpha_n\}$, then $Part_1(\mathcal{C}({\bf A}))=\{(\alpha_1\mid \top),\ldots,( \alpha_n\mid \top)\}$ and it gives   
the coarsest partition among those  that we denoted by $Part_i(\mathcal{C}({\bf A}))$. 

In the following, if $| \atom({\bf A}) | = n$,  for simplicity, we will write $Seq({\bf A})$  and $Part(\mathcal{C}({\bf A}))$ instead of $Seq_{n-1}({\bf A})$ and $Part_{n-1}(\mathcal{C}({\bf A}))$ respectively. 
Note that in this case, for every $\overline{\alpha}=\langle \alpha_1, \alpha_2, \ldots, \alpha_{n-1} \rangle\in Seq({\bf A})$, the compound conditional $\omega_{\overline{\alpha}}$ defined as in (\ref{eq:atomsDef}), can  be equivalently written as
$$ 
\omega_{\overline{\alpha}}  = (\alpha_1 \mid \top) \sqcap  (\alpha_2 \mid \alpha_2 \lor \dots \lor \alpha_n) \sqcap  (\alpha_3 \mid  \alpha_3 \lor \ldots \lor \alpha_n) \sqcap \ldots \sqcap
 (\alpha_{n-1} \mid \alpha_{n-1} \lor \alpha_n).
 $$ 

 Next theorem shows that these conditionals are in fact the atoms of $\mathcal{C}({\bf A})$.
 
 \begin{theorem}\label{th:atoms} 
Let ${\bf A}$ be a Boolean algebra such that $|\atom({\bf A})|=n$. Then, 
 $$ 
 \atom(\mathcal{C}({\bf A})) = \{ \omega_{\overline{\alpha}} \mid  \overline{\alpha} \in Seq({\bf A}) \}=Part(\mathcal{C}({\bf A})).
 $$ 
As a consequence, $| \atom(\mathcal{C}({\bf A})) | = n!$ and  $|\mathcal{C}({\bf A}) | = 2^{n!}$. 
\end{theorem}

 \begin{proof} 
To show that $Part(\mathcal{C}({\bf A}))$ coincides with $\atom(\mathcal{C}({\bf A}))$, by  Proposition \ref{prop:atomsAlgebra}, we have to prove the following two conditions:
 
\vspace{.2cm}
 
\noindent(a) {\em $Part(\mathcal{C}({\bf A}))$ is a partition of $\mathcal{C}({\bf A})$. }

\vspace{0.2cm}

\noindent(b) {\em Any  $\omega_{\overline \alpha} \in Part(\mathcal{C}({\bf A}))$ is such that $\bot_\mathfrak{C}<\omega_{\overline{\alpha}}$ and there is no  $t \in \mathcal{C}({\bf A})$ such that $\bot_\mathfrak{C}  < t <  \omega_{\overline \alpha}$. }
 
\vspace{0.2cm} 

\noindent  It is clear that (a) is just the case $i= n-1$ in Proposition \ref{useful}. As for (b), 
note first that each $\omega_{\overline \alpha} \in Part(\mathcal{C}({\bf A}))$ is different from $\bot_{\mathfrak{C}}$. Indeed, it follows from (a) and a symmetry argument on the elements of $ Part(\mathcal{C}({\bf A}))$.
Thus,
it is enough to show that, if $t$ is any element of $\mathcal{C}({\bf A})$ which is not $\bot_\mathfrak{C}$, then either $t\sqcap \omega_{\overline{\alpha}}=\omega_{\overline{\alpha}}$, or $t\sqcap \omega_{\overline{\alpha}}=\bot_\mathfrak{C}$. We show this claim by cases on the form of $t$:

\begin{itemize}
\item[(i)] Assume $t$ is basic conditional of the form $t = (\gamma \mid b)$ with $\gamma \in \atom({\bf A})$.  Let $\overline{\alpha}=\langle \alpha_1, \alpha_2, \ldots, \alpha_{n-1} \rangle$. Since $\gamma \in \atom({\bf A})$, then $\gamma = \alpha_i$ for some $1 \leq i \leq n$. Then we have two cases: either  $b =  \alpha_i \lor \dots \lor \alpha_{n}$, and in that case $\omega_{\overline \alpha} \sqcap   (\gamma \mid b) = \omega_{\overline \alpha} $, or otherwise $b$ is of the form $b = \alpha_i \lor \alpha_k \lor c$, for some $k < i$  and some $c\in A$. If the latter is the case, we have $(\gamma \mid b) \sqcap (\alpha_k \mid \alpha_k \lor \ldots \lor \alpha_n) = (\alpha_i \mid  \alpha_i \lor \alpha_k \lor c) \sqcap  (\alpha_k \mid \alpha_k \lor \ldots \lor \alpha_n) \leq  (\alpha_i \mid  \alpha_i \lor \alpha_k) \sqcap (\alpha_k \mid \alpha_k \lor \alpha_i) = \bot_\mathfrak{C}$, and hence  $(\gamma \mid b) \sqcap \omega_{\overline \alpha} = \bot_\mathfrak{C}$ as well. 

\item[(ii)] Assume $t$ is  a basic conditional $t = (a \mid b)$. By (ii) of Proposition\ \ref{lemma:repres}, we can express $(a \mid b) = \bigsqcup_{\alpha \in  \atom({\bf A}): \alpha  \leq a} (\alpha \mid b)$. Hence,  $\omega_{\overline \alpha} \sqcap   (a \mid b) = \bigsqcup_{\gamma} \omega_{\overline \alpha} \sqcap (\gamma \mid b)$, but by (i), for each $\gamma$, $\omega_{\overline \alpha} \sqcap (\gamma \mid b)$ is either $\omega_{\overline \alpha}$ or $\bot_{\mathfrak{C}}$. So this is also the case  for $\omega_{\overline \alpha} \sqcap  t$. 

\item[(iii)] Finally, assume $t$ is an arbitrary element of $\mathcal{C}({\bf A})$. By Proposition \ref{lemma:repres} above, $t$ is a $\sqcap$-$\sqcup$ combination of basic conditionals, i.e.\ it can be displayed as $t = \bigsqcap_i (\bigsqcup_j (a_{i_j}\mid b_{i_j}))$. Then we have
$$
\omega_{\overline{\alpha}} \sqcap t =  \omega_{\overline{\alpha}} \sqcap \left (\bigsqcap_i \left(\bigsqcup_j (a_{i_j}\mid b_{i_j})\right)\right )  = 
\bigsqcap_i \left(\bigsqcup_j \omega_{\overline{\alpha}} \sqcap (a_{i_j}\mid b_{i_j})\right). 
$$ 
By (ii), each  $\omega_{\overline{\alpha}} \sqcap (a_{i_j}\mid b_{i_j})$ is either equal to $\bot_{\mathfrak{C}}$ or  to $\omega_{\overline{\alpha}}$, and hence so is $\omega_{\overline{\alpha}} \sqcap t$. 
\end{itemize}

Therefore, we have proved that $Part(\mathcal{C}({\bf A})) = \atom(\mathcal{C}({\bf A}))$.  
 \end{proof}

\begin{figure}[ht] 
\begin{center}
\definecolor{qqqqff}{rgb}{0.,0.,1.}
\begin{tikzpicture}
 \coordinate[label=left: ${\alpha_1}$] (a) at (2.389196370707718,2.7932465230629733);
 \coordinate[label=left: ${\alpha_2}$] (b) at (4.176028518092002,2.947505413484638);
 \coordinate[label=right: ${\alpha_3}$] (c) at (5.114436768157129,2.8703759682738057);
 \coordinate[label=below: ${\bot}$] (d) at (3.764671476967562,1.4049165092679916);
 \coordinate[label=right: ${\neg \alpha_1}$] (e) at (5.525793809281566,4.412964872490453);
 \coordinate[label=left: ${\neg \alpha_3}$] (f) at (2.800553411832157,4.33583542727962);
 \coordinate[label=right: ${\neg \alpha_2}$] (h) at (3.738961661897287,4.258705982068787);
 \coordinate[label=above: ${\top}$] (g) at (4.150318703021727,5.801294886285434);
\draw (2.389196370707718,2.7932465230629733)-- (3.764671476967562,1.4049165092679916);
\draw (3.764671476967562,1.4049165092679916)-- (5.114436768157129,2.8703759682738057);
\draw (2.389196370707718,2.7932465230629733)-- (3.738961661897287,4.258705982068787);
\draw (3.738961661897287,4.258705982068787)-- (5.114436768157129,2.8703759682738057);
\draw (2.800553411832157,4.33583542727962)-- (4.176028518092002,2.947505413484638);
\draw (4.176028518092002,2.947505413484638)-- (5.525793809281566,4.412964872490453);
\draw (2.800553411832157,4.33583542727962)-- (4.150318703021727,5.801294886285434);
\draw (4.150318703021727,5.801294886285434)-- (5.525793809281566,4.412964872490453);
\draw (4.150318703021727,5.801294886285434)-- (3.738961661897287,4.258705982068787);
\draw (2.800553411832157,4.33583542727962)-- (2.389196370707718,2.7932465230629733);
\draw (4.176028518092002,2.947505413484638)-- (3.764671476967562,1.4049165092679916);
\draw (5.525793809281566,4.412964872490453)-- (5.114436768157129,2.8703759682738057);
\begin{scriptsize}
\draw [fill=black] (2.389196370707718,2.7932465230629733) circle (1.5pt);
\draw [fill=black] (3.764671476967562,1.4049165092679916) circle (1.5pt);
\draw [fill=black] (5.114436768157129,2.8703759682738057) circle (1.5pt);
\draw [fill=black] (3.738961661897287,4.258705982068787) circle (1.5pt);
\draw [fill=black] (2.800553411832157,4.33583542727962) circle (1.5pt);
\draw [fill=black] (4.176028518092002,2.947505413484638) circle (1.5pt);
\draw [fill=black] (5.525793809281566,4.412964872490453) circle (1.5pt);
\draw [fill=black] (4.150318703021727,5.801294886285434) circle (1.5pt);
\end{scriptsize}
\end{tikzpicture}
\end{center}
\caption{\small{The Boolean algebra with 3 atoms and 8 elements.}}
\label{fig:threeAtoms}
\end{figure}
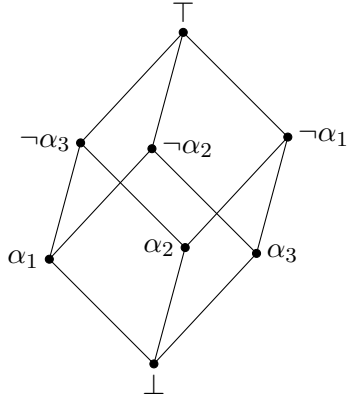

\begin{example}\label{Ex:atoms}
Let ${\bf A}$ be the Boolean algebra of 3 atoms $\alpha_1, \alpha_2, \alpha_3$, and 8 elements, see Figure \ref{fig:threeAtoms}. 
Theorem \ref{th:atoms} tells us that the atoms of the conditional algebra $\mathcal{C}({\bf A})$ are as follows:
$$
\atom(\mathcal{C}({\bf A}))=\{(\alpha_i\mid \top)\sqcap  (\alpha_j\mid \neg \alpha_i) : i,j=1,2,3 \mbox{ and }i\neq j\}. 
$$ 
Therefore, $\mathcal{C}({\bf A})$ has six atoms and $2^6=64$ elements (see Figure \ref{fig:sixAtoms}). In particular, 
the  atoms are: 
\begin{itemize}
\item[] \hfill $\omega_1=(\alpha_1\mid \top)\sqcap (\alpha_2\mid \neg \alpha_1)$, 
\hspace{0.2cm} $\omega_2=(\alpha_1\mid \top)\sqcap (\alpha_3\mid \neg \alpha_1)$, \hfill \mbox{}
\item[] \hfill $\omega_3=(\alpha_2\mid \top)\sqcap (\alpha_1\mid \neg \alpha_2)$, 
\hspace{0.2cm} $\omega_4=(\alpha_2\mid \top)\sqcap (\alpha_3\mid \neg \alpha_2)$, \hfill \mbox{}
\item[] \hfill $\omega_5=(\alpha_3\mid \top)\sqcap (\alpha_1\mid \neg \alpha_3)$,  
\hspace{0.2cm} $\omega_6=(\alpha_3 \mid \top)\sqcap (\alpha_2\mid \neg \alpha_3)$. \hfill \mbox{}
\end{itemize}
Notice that $\omega_1\sqcup  \omega_2= (\alpha_1\mid\top)$. Indeed, $((\alpha_1\mid \top)\sqcap  (\alpha_2\mid \neg \alpha_1)) \sqcup  ((\alpha_1\mid \top)\sqcap  (\alpha_3\mid \neg \alpha_1))= (\alpha_1\mid \top) \sqcap  (\alpha_2\vee \alpha_3\mid \neg \alpha_1)=(\alpha_1\mid \top)\sqcap  (\neg \alpha_1\mid \neg \alpha_1)= (\alpha_1\mid \top)\sqcap  \top_\mathfrak{C} = (\alpha_1\mid\top)$. Analogously, we can also derive $\omega_3\sqcup  \omega_4 = (\alpha_2\mid\top)$ and $\omega_5\sqcup  \omega_6 = (\alpha_3\mid \top)$. 

Now, let us consider the conditional $t= (\alpha_1\mid \neg \alpha_3)$. Obviously, $t=\bigsqcup \{\omega_i: \omega_i\leq t\}$ and thanks to Proposition \ref{prop:ordine}, we can check that 
$$
t=\omega_1\sqcup  \omega_2\sqcup  \omega_5.
$$

As a matter of fact, since $\omega_1\sqcup  \omega_2 =  (\alpha_1\mid\top)$, we have 
$$\omega_1\sqcup  \omega_2\sqcup  \omega_5=(\alpha_1\mid\top)\sqcup((\alpha_3\mid \top)\sqcap(\alpha_1\mid \neg\alpha_3))=((\alpha_1\mid \top)\sqcup (\alpha_3\mid \top))\sqcap((\alpha_1\mid\top)\sqcup(\alpha_1\mid\neg \alpha_3)).
$$ Now, $(\alpha_1\mid \top)\sqcup (\alpha_3\mid \top)=(\alpha_1\vee\alpha_3\mid \top)=(\neg\alpha_3\to \alpha_1\mid \top)$, while Proposition \ref{prop:ordine} (iii) implies that $(\alpha_1\mid\top)\sqcup(\alpha_1\mid\neg \alpha_3)=(\alpha_1\mid\neg \alpha_3)$ because $\alpha_1\leq\neg\alpha_3\leq \top$. Finally, by Proposition \ref{prop:ordine} (v) $(\alpha_1\mid \neg\alpha_3)\leq (\neg\alpha_3\to\alpha_1\mid \top)$ and hence $\omega_1\sqcup  \omega_2\sqcup  \omega_5=(\alpha_1\mid\neg\alpha_3)=t$. 
\qed
\end{example}

\begin{figure}[h!]

\begin{center}
\mbox{} \\
\definecolor{qqqqff}{rgb}{0,0,1}
\begin{tikzpicture}[scale=.7]
\draw [line width=0.2pt] (9.5,1.6)-- (12,0.6);
\draw [line width=0.2pt] (10.5,1.6)-- (12,0.6);
\draw [line width=0.2pt] (11.5,1.6)-- (12,0.6);
\draw [line width=0.2pt] (12.5,1.6)-- (12,0.6);
\draw [line width=0.2pt] (13.5,1.6)-- (12,0.6);
\draw [line width=0.2pt] (14.5,1.6)-- (12,0.6);
\draw [line width=1.2pt] (9.5,1.6)-- (5,3.5);
\draw [line width=0.2pt] (5,3.5)-- (10.5,1.6);
\draw [line width=0.2pt] (9.5,1.6)-- (6,3.5);
\draw [line width=0.2pt] (9.5,1.6)-- (8,3.5);1
\draw [line width=0.2pt] (9.5,1.6)-- (10,3.5);
\draw [line width=0.2pt] (9.5,1.6)-- (12,3.5);
\draw [line width=1.2pt] (10.5,1.6)-- (5,3.5);
\draw [line width=0.2pt] (10.5,1.6)-- (9,3.5);
\draw [line width=0.2pt] (10.5,1.6)-- (13,3.5);
\draw [line width=0.2pt] (10.5,1.6)-- (14,3.5);
\draw [line width=0.2pt] (11.5,1.6)-- (6,3.5);
\draw [line width=0.2pt] (11.5,1.6)-- (7,3.5);
\draw [line width=0.2pt] (11.5,1.6)-- (11,3.5);
\draw [line width=0.2pt] (11.5,1.6)-- (14,3.5);
\draw [line width=0.2pt] (11.5,1.6)-- (15,3.5);
\draw [line width=0.2pt] (12.5,1.6)-- (8,3.5);
\draw [line width=0.2pt] (12.5,1.6)-- (9,3.5);
\draw [line width=0.2pt] (12.5,1.6)-- (11,3.5);
\draw [line width=0.2pt] (12.5,1.6)-- (17,3.5);
\draw [line width=0.2pt] (12.5,1.6)-- (18,3.5);
\draw [line width=0.2pt] (13.5,1.6)-- (13,3.5);
\draw [line width=1.2pt] (13.5,1.6)-- (10,3.5);
\draw [line width=0.2pt] (13.5,1.6)-- (15,3.5);
\draw [line width=0.2pt] (13.5,1.6)-- (17,3.5);
\draw [line width=0.2pt] (13.5,1.6)-- (19,3.5);
\draw [line width=0.2pt] (14.5,1.6)-- (19,3.5);
\draw [line width=0.2pt] (14.5,1.6)-- (18,3.5);
\draw [line width=0.2pt] (14.5,1.6)-- (16,3.5);
\draw [line width=0.2pt] (11.5,1.6)-- (16,3.5);
\draw [line width=0.2pt] (14.5,1.6)-- (12,3.5);
\draw [line width=0.2pt] (2.5,6)-- (5,3.5);
\draw [line width=0.2pt] (2.5,6)-- (6,3.5);
\draw [line width=0.2pt] (2.5,6)-- (7,3.5);
\draw [line width=0.2pt] (3.5,6)-- (5,3.5);
\draw [line width=0.2pt] (3.5,6)-- (8,3.5);
\draw [line width=0.2pt] (3.5,6)-- (9,3.5);
\draw [line width=0.2pt] (4.5,6)-- (6,3.5);
\draw [line width=0.2pt] (4.5,6)-- (9,3.5);
\draw [line width=0.2pt] (4.5,6)-- (11,3.5);
\draw [line width=0.2pt] (5.5,6)-- (7,3.5);
\draw [line width=0.2pt] (5.5,6)-- (9,3.5);
\draw [line width=0.2pt] (5.5,6)-- (11,3.5);
\draw [line width=1.2pt] (6.5,6)-- (5,3.5);
\draw [line width=1.2pt] (6.5,6)-- (10,3.5);
\draw [line width=0.2pt] (6.5,6)-- (13,3.5);
\draw [line width=0.2pt] (7.5,6)-- (6,3.5);
\draw [line width=0.2pt] (7.5,6)-- (10,3.5);
\draw [line width=0.2pt] (7.5,6)-- (13,3.5);
\draw [line width=0.2pt] (8.5,6)-- (7,3.5);
\draw [line width=0.2pt] (8.5,6)-- (13,3.5);
\draw [line width=0.2pt] (8.5,6)-- (15,3.5);
\draw [line width=0.2pt] (9.5,6)-- (8,3.5);
\draw [line width=0.2pt] (9.5,6)-- (10,3.5);
\draw [line width=0.2pt] (9.5,6)-- (17,3.5);
\draw [line width=0.2pt] (10.5,6)-- (5,3.5);
\draw [line width=0.2pt] (10.5,6)-- (12,3.5);
\draw [line width=0.2pt] (10.5,6)-- (14,3.5);
\draw [line width=0.2pt] (11.5,6)-- (6,3.5);
\draw [line width=0.2pt] (11.5,6)-- (16,3.5);
\draw [line width=0.2pt] (11.5,6)-- (12,3.5);
\draw [line width=0.2pt] (12.5,6)-- (12,3.5);
\draw [line width=0.2pt] (12.5,6)-- (8,3.5);
\draw [line width=0.2pt] (12.5,6)-- (18,3.5);
\draw [line width=0.2pt] (13.5,6)-- (12,3.5);
\draw [line width=0.2pt] (13.5,6)-- (10,3.5);
\draw [line width=0.2pt] (13.5,6)-- (19,3.5);
\draw [line width=0.2pt] (14.5,6)-- (14,3.5);
\draw [line width=0.2pt] (14.5,6)-- (16,3.5);
\draw [line width=0.2pt] (14.5,6)-- (7,3.5);
\draw [line width=0.2pt] (21.5,6)-- (19,3.5);
\draw [line width=0.2pt] (21.5,6)-- (18,3.5);
\draw [line width=0.2pt] (21.5,6)-- (17,3.5);
\draw [line width=0.2pt] (20.5,6)-- (19,3.5);
\draw [line width=0.2pt] (20.5,6)-- (15,3.5);
\draw [line width=0.2pt] (20.5,6)-- (14,3.5);
\draw [line width=0.2pt] (19.5,6)-- (18,3.5);
\draw [line width=0.2pt] (19.5,6)-- (16,3.5);
\draw [line width=0.2pt] (19.5,6)-- (11,3.5);
\draw [line width=0.2pt] (18.5,6)-- (16,3.5);
\draw [line width=0.2pt] (18.5,6)-- (11,3.5);
\draw [line width=0.2pt] (18.5,6)-- (15,3.5);
\draw [line width=0.2pt] (17.5,6)-- (19,3.5);
\draw [line width=0.2pt] (17.5,6)-- (13,3.5);
\draw [line width=0.2pt] (17.5,6)-- (14,3.5);
\draw [line width=0.2pt] (16.5,6)-- (18,3.5);
\draw [line width=0.2pt] (16.5,6)-- (13,3.5);
\draw [line width=0.2pt] (16.5,6)-- (9,3.5);
\draw [line width=0.2pt] (15.5,6)-- (17,3.5);
\draw [line width=0.2pt] (15.5,6)-- (13,3.5);
\draw [line width=0.2pt] (15.5,6)-- (9,3.5);
\draw [line width=0.2pt] (2.49,6.01)-- (5.03,8.98);
\draw [line width=0.2pt] (2.49,6.01)-- (6.03,8.98);
\draw [line width=0.2pt] (2.49,6.01)-- (7.03,8.98);
\draw [line width=0.2pt] (3.49,6.01)-- (5.03,8.98);
\draw [line width=0.2pt] (3.49,6.01)-- (8.03,8.98);
\draw [line width=0.2pt] (3.49,6.01)-- (9.03,8.98);
\draw [line width=0.2pt] (4.49,6.01)-- (6.03,8.98);
\draw [line width=0.2pt] (4.49,6.01)-- (9.03,8.98);
\draw [line width=0.2pt] (4.49,6.01)-- (11.03,8.98);
\draw [line width=0.2pt] (5.49,6.01)-- (7.03,8.98);
\draw [line width=0.2pt] (5.49,6.01)-- (9.03,8.98);
\draw [line width=0.2pt] (5.49,6.01)-- (11.03,8.98);
\draw [line width=0.2pt] (6.49,6.01)-- (5.03,8.98);
\draw [line width=0.2pt] (6.49,6.01)-- (10.03,8.98);
\draw [line width=0.2pt] (6.49,6.01)-- (13.03,8.98);
\draw [line width=0.2pt] (7.49,6.01)-- (6.03,8.98);
\draw [line width=0.2pt] (7.49,6.01)-- (10.03,8.98);
\draw [line width=0.2pt] (7.49,6.01)-- (13.03,8.98);
\draw [line width=0.2pt] (8.49,6.01)-- (7.03,8.98);
\draw [line width=0.2pt] (8.49,6.01)-- (13.03,8.98);
\draw [line width=0.2pt] (8.49,6.01)-- (15.03,8.98);
\draw [line width=0.2pt] (9.49,6.01)-- (8.03,8.98);
\draw [line width=0.2pt] (9.49,6.01)-- (10.03,8.98);
\draw [line width=0.2pt] (9.49,6.01)-- (17.03,8.98);
\draw [line width=0.2pt] (10.49,6.01)-- (5.03,8.98);
\draw [line width=0.2pt] (10.49,6.01)-- (12.03,8.98);
\draw [line width=0.2pt] (10.49,6.01)-- (14.03,8.98);
\draw [line width=0.2pt] (11.49,6.01)-- (6.03,8.98);
\draw [line width=0.2pt] (11.49,6.01)-- (16.03,8.98);
\draw [line width=0.2pt] (11.49,6.01)-- (12.03,8.98);
\draw [line width=0.2pt] (12.49,6.01)-- (12.03,8.98);
\draw [line width=0.2pt] (12.49,6.01)-- (8.03,8.98);
\draw [line width=0.2pt] (12.49,6.01)-- (18.03,8.98);
\draw [line width=0.2pt] (13.49,6.01)-- (12.03,8.98);
\draw [line width=0.2pt] (13.49,6.01)-- (10.03,8.98);
\draw [line width=0.2pt] (13.49,6.01)-- (19.03,8.98);
\draw [line width=0.2pt] (14.49,6.01)-- (14.03,8.98);
\draw [line width=0.2pt] (14.49,6.01)-- (16.03,8.98);
\draw [line width=0.2pt] (14.49,6.01)-- (7.03,8.98);
\draw [line width=0.2pt] (21.49,6.01)-- (19.03,8.98);
\draw [line width=0.2pt] (21.49,6.01)-- (18.03,8.98);
\draw [line width=0.2pt] (21.49,6.01)-- (17.03,8.98);
\draw [line width=0.2pt] (20.49,6.01)-- (19.03,8.98);
\draw [line width=0.2pt] (20.49,6.01)-- (15.03,8.98);
\draw [line width=0.2pt] (20.49,6.01)-- (14.03,8.98);
\draw [line width=0.2pt] (19.49,6.01)-- (18.03,8.98);
\draw [line width=0.2pt] (19.49,6.01)-- (16.03,8.98);
\draw [line width=0.2pt] (19.49,6.01)-- (11.03,8.98);
\draw [line width=0.2pt] (18.49,6.01)-- (16.03,8.98);
\draw [line width=0.2pt] (18.49,6.01)-- (11.03,8.98);
\draw [line width=0.2pt] (18.49,6.01)-- (15.03,8.98);
\draw [line width=0.2pt] (17.49,6.01)-- (19.03,8.98);
\draw [line width=0.2pt] (17.49,6.01)-- (13.03,8.98);
\draw [line width=0.2pt] (17.49,6.01)-- (14.03,8.98);
\draw [line width=0.2pt] (16.49,6.01)-- (18.03,8.98);
\draw [line width=0.2pt] (16.49,6.01)-- (13.03,8.98);
\draw [line width=0.2pt] (16.49,6.01)-- (9.03,8.98);
\draw [line width=0.2pt] (15.49,6.01)-- (17.03,8.98);
\draw [line width=0.2pt] (15.49,6.01)-- (13.03,8.98);
\draw [line width=0.2pt] (15.49,6.01)-- (9.03,8.98);
\draw [line width=0.2pt] (9.34,11.13)-- (5.02,8.99);
\draw [line width=0.2pt] (5.02,8.99)-- (10.34,11.13);
\draw [line width=0.2pt] (9.34,11.13)-- (6.02,8.99);
\draw [line width=0.2pt] (9.34,11.13)-- (8.02,8.99);
\draw [line width=0.2pt] (9.34,11.13)-- (10.02,8.99);
\draw [line width=0.2pt] (9.34,11.13)-- (12.02,8.99);
\draw [line width=0.2pt] (10.34,11.13)-- (7.02,8.99);
\draw [line width=0.2pt] (10.34,11.13)-- (9.02,8.99);
\draw [line width=0.2pt] (10.34,11.13)-- (13.02,8.99);
\draw [line width=0.2pt] (10.34,11.13)-- (14.02,8.99);
\draw [line width=0.2pt] (11.34,11.13)-- (6.02,8.99);
\draw [line width=0.2pt] (11.34,11.13)-- (7.02,8.99);
\draw [line width=0.2pt] (11.34,11.13)-- (11.02,8.99);
\draw [line width=0.2pt] (11.34,11.13)-- (14.02,8.99);
\draw [line width=0.2pt] (11.34,11.13)-- (15.02,8.99);
\draw [line width=0.2pt] (12.34,11.13)-- (8.02,8.99);
\draw [line width=0.2pt] (12.34,11.13)-- (9.02,8.99);
\draw [line width=0.2pt] (12.34,11.13)-- (11.02,8.99);
\draw [line width=0.2pt] (12.34,11.13)-- (17.02,8.99);
\draw [line width=0.2pt] (12.34,11.13)-- (18.02,8.99);
\draw [line width=0.2pt] (13.34,11.13)-- (13.02,8.99);
\draw [line width=0.2pt] (13.34,11.13)-- (10.02,8.99);
\draw [line width=0.2pt] (13.34,11.13)-- (15.02,8.99);
\draw [line width=0.2pt] (13.34,11.13)-- (17.02,8.99);
\draw [line width=0.2pt] (13.34,11.13)-- (19.02,8.99);
\draw [line width=0.2pt] (14.34,11.13)-- (19.02,8.99);
\draw [line width=0.2pt] (14.34,11.13)-- (18.02,8.99);
\draw [line width=0.2pt] (14.34,11.13)-- (16.02,8.99);
\draw [line width=0.2pt] (11.34,11.13)-- (16.02,8.99);
\draw [line width=0.2pt] (14.34,11.13)-- (12.02,8.99);
\draw [line width=0.2pt] (9.34,11.11)-- (11.88,12.3);
\draw [line width=0.2pt] (10.34,11.11)-- (11.88,12.3);
\draw [line width=0.2pt] (11.34,11.11)-- (11.88,12.3);
\draw [line width=0.2pt] (12.34,11.11)-- (11.88,12.3);
\draw [line width=0.2pt] (13.34,11.11)-- (11.88,12.3);
\draw [line width=0.2pt] (14.34,11.11)-- (11.88,12.3);
 \coordinate[label=below: ${\bot_\mathfrak{C}}$] (a1) at (12,0.6);
 \coordinate[label=below: ${\omega_1}$] (a1) at (9.5,1.6);
 \coordinate[label=below: ${\omega_2}$] (a1) at (10.5,1.6);
 \coordinate[label=below: ${\omega_3}$] (a1) at (11.5,1.6);
 \coordinate[label=below: ${\omega_4}$] (a1) at (12.5,1.6);
 \coordinate[label=below: ${\omega_5}$] (a1) at (13.5,1.6);
 \coordinate[label=below: ${\omega_6}$] (a1) at (14.5,1.6);
   \coordinate[label=left: ${\mathbf{t}}$] (a1) at (6.5,6);
 \coordinate[label=left: $\mathbf{(\alpha_1\mid \top)}$] (a1) at (5,3.5);
  \coordinate[label=below: $\mathbf{(\alpha_2\mid \top)}$] (a1) at (11,3.5);
   \coordinate[label=right: $\mathbf{(\alpha_3\mid \top)}$] (a1) at (19,3.5);
   \coordinate[label=above: ${\top_\mathfrak{C}}$] (a1) at (11.88,12.3);

\draw [fill=black] (12,0.6) circle (4pt);
\draw [fill=lightgray] (9.5,1.6) circle (4pt);
\draw [fill=lightgray] (10.5,1.6) circle (4pt);
\draw [fill=lightgray] (11.5,1.6) circle (4pt);
\draw [fill=lightgray] (12.5,1.6) circle (4pt);
\draw [fill=lightgray] (13.5,1.6) circle (4pt);
\draw [fill=lightgray] (14.5,1.6) circle (4pt);
\draw [fill=black] (12,3.5) circle (1.5pt);
\draw [fill=black] (11,3.5) circle (4pt);
\draw [fill=black] (10,3.5) circle (1.5pt);
\draw [fill=black] (9,3.5) circle (1.5pt);
\draw [fill=black] (8,3.5) circle (1.5pt);
\draw [fill=black] (7,3.5) circle (1.5pt);
\draw [fill=black] (6,3.5) circle (1.5pt);
\draw [fill=black] (5,3.5) circle (4pt);
\draw [fill=black] (13,3.5) circle (1.5pt);
\draw [fill=black] (14,3.5) circle (1.5pt);
\draw [fill=black] (15,3.5) circle (1.5pt);
\draw [fill=black] (16,3.5) circle (1.5pt);
\draw [fill=black] (17,3.5) circle (1.5pt);
\draw [fill=black] (18,3.5) circle (1.5pt);
\draw [fill=black] (19,3.5) circle (4pt);
\draw [fill=black] (11.5,6) circle (1.5pt);
\draw [fill=black] (10.5,6) circle (1.5pt);
\draw [fill=black] (9.5,6) circle (1.5pt);
\draw [fill=black] (8.5,6) circle (1.5pt);
\draw [fill=black] (7.5,6) circle (1.5pt);
\draw [fill=black,xshift=-2.5pt,yshift=-1.5pt] (6.5,6) rectangle +(6pt, 6pt);
\draw [fill=black] (5.5,6) circle (1.5pt);
\draw [fill=black] (4.5,6) circle (1.5pt);
\draw [fill=black] (3.5,6) circle (1.5pt);
\draw [fill=black] (2.5,6) circle (1.5pt);
\draw [fill=black] (21.5,6) circle (1.5pt);
\draw [fill=black] (20.5,6) circle (1.5pt);
\draw [fill=black] (19.5,6) circle (1.5pt);
\draw [fill=black] (18.5,6) circle (1.5pt);
\draw [fill=black] (17.5,6) circle (1.5pt);
\draw [fill=black] (16.5,6) circle (1.5pt);
\draw [fill=black] (15.5,6) circle (1.5pt);
\draw [fill=black] (14.5,6) circle (1.5pt);
\draw [fill=black] (13.5,6) circle (1.5pt);
\draw [fill=black] (12.5,6) circle (1.5pt);
\draw [fill=black] (2.5,6) circle (1.5pt);
\draw [fill=black] (12.03,8.98) circle (1.5pt);
\draw [fill=black] (11.03,8.98) circle (1.5pt);
\draw [fill=black] (10.03,8.98) circle (1.5pt);
\draw [fill=black] (9.03,8.98) circle (1.5pt);
\draw [fill=black] (8.03,8.98) circle (1.5pt);
\draw [fill=black] (7.03,8.98) circle (1.5pt);
\draw [fill=black] (6.03,8.98) circle (1.5pt);
\draw [fill=black] (5.03,8.98) circle (1.5pt);
\draw [fill=black] (13.03,8.98) circle (1.5pt);
\draw [fill=black] (14.03,8.98) circle (1.5pt);
\draw [fill=black] (15.03,8.98) circle (1.5pt);
\draw [fill=black] (16.03,8.98) circle (1.5pt);
\draw [fill=black] (17.03,8.98) circle (1.5pt);
\draw [fill=black] (18.03,8.98) circle (1.5pt);
\draw [fill=black] (19.03,8.98) circle (1.5pt);
\draw [fill=black] (11.49,6.01) circle (1.5pt);
\draw [fill=black] (10.49,6.01) circle (1.5pt);
\draw [fill=black] (9.49,6.01) circle (1.5pt);
\draw [fill=black] (8.49,6.01) circle (1.5pt);
\draw [fill=black] (7.49,6.01) circle (1.5pt);
\draw [fill=black] (6.49,6.01) circle (1.5pt);
\draw [fill=black] (5.49,6.01) circle (1.5pt);
\draw [fill=black] (4.49,6.01) circle (1.5pt);
\draw [fill=black] (3.49,6.01) circle (1.5pt);
\draw [fill=black] (2.49,6.01) circle (1.5pt);
\draw [fill=black] (21.49,6.01) circle (1.5pt);
\draw [fill=black] (20.49,6.01) circle (1.5pt);
\draw [fill=black] (19.49,6.01) circle (1.5pt);
\draw [fill=black] (18.49,6.01) circle (1.5pt);
\draw [fill=black] (17.49,6.01) circle (1.5pt);
\draw [fill=black] (16.49,6.01) circle (1.5pt);
\draw [fill=black] (15.49,6.01) circle (1.5pt);
\draw [fill=black] (14.49,6.01) circle (1.5pt);
\draw [fill=black] (13.49,6.01) circle (1.5pt);
\draw [fill=black] (12.49,6.01) circle (1.5pt);
\draw [fill=black] (9.34,11.13) circle (1.5pt);
\draw [fill=black] (10.34,11.13) circle (1.5pt);
\draw [fill=black] (11.34,11.13) circle (1.5pt);
\draw [fill=black] (12.34,11.13) circle (1.5pt);
\draw [fill=black] (13.34,11.13) circle (1.5pt);
\draw [fill=black] (14.34,11.13) circle (1.5pt);
\draw [fill=black] (12.02,8.99) circle (1.5pt);
\draw [fill=black] (11.02,8.99) circle (1.5pt);
\draw [fill=black] (10.02,8.99) circle (1.5pt);
\draw [fill=black] (9.02,8.99) circle (4pt);
\draw [fill=black] (8.02,8.99) circle (1.5pt);
\draw [fill=black] (7.02,8.99) circle (1.5pt);
\draw [fill=black] (6.02,8.99) circle (1.5pt);
\draw [fill=black] (5.02,8.99) circle (1.5pt);
\draw [fill=black] (13.02,8.99) circle (1.5pt);
\draw [fill=black] (14.02,8.99) circle (4pt);
\draw [fill=black] (15.02,8.99) circle (4pt);
\draw [fill=black] (16.02,8.99) circle (1.5pt);
\draw [fill=black] (17.02,8.99) circle (1.5pt);
\draw [fill=black] (18.02,8.99) circle (1.5pt);
\draw [fill=black] (19.02,8.99) circle (1.5pt);
\draw [fill=black] (11.88,12.3) circle (4pt);
\draw [fill=black] (9.34,11.11) circle (1.5pt);
\draw [fill=black] (10.34,11.11) circle (1.5pt);
\draw [fill=black] (11.34,11.11) circle (1.5pt);
\draw [fill=black] (12.34,11.11) circle (1.5pt);
\draw [fill=black] (13.34,11.11) circle (1.5pt);
\draw [fill=black] (14.34,11.11) circle (1.5pt);
\end{tikzpicture}
\end{center}

\caption{\small{The algebra of conditionals $\mathcal{C}({\bf A})$ of Example \ref{Ex:atoms}, where $\atom({\bf A})= \{\alpha_1, \alpha_2, \alpha_3\}$  and  $\atom(\mathcal{C}({\bf A})) = \{\omega_1,\omega_2,\omega_3,\omega_4,\omega_5,\omega_6\}$. The element $t= (\alpha_1\mid \neg \alpha_3)$ (squared node) is  obtained as $\omega_1\sqcup  \omega_2\sqcup  \omega_5$. 
The atoms of $\mathcal{C}({\bf A})$ are marked by grey dots, while the elements $a$ of the original algebra ${\bf A}$, regarded as conditionals $(a \mid \top)$, correspond to the bigger black dots.}}
\label{fig:sixAtoms}
\end{figure}
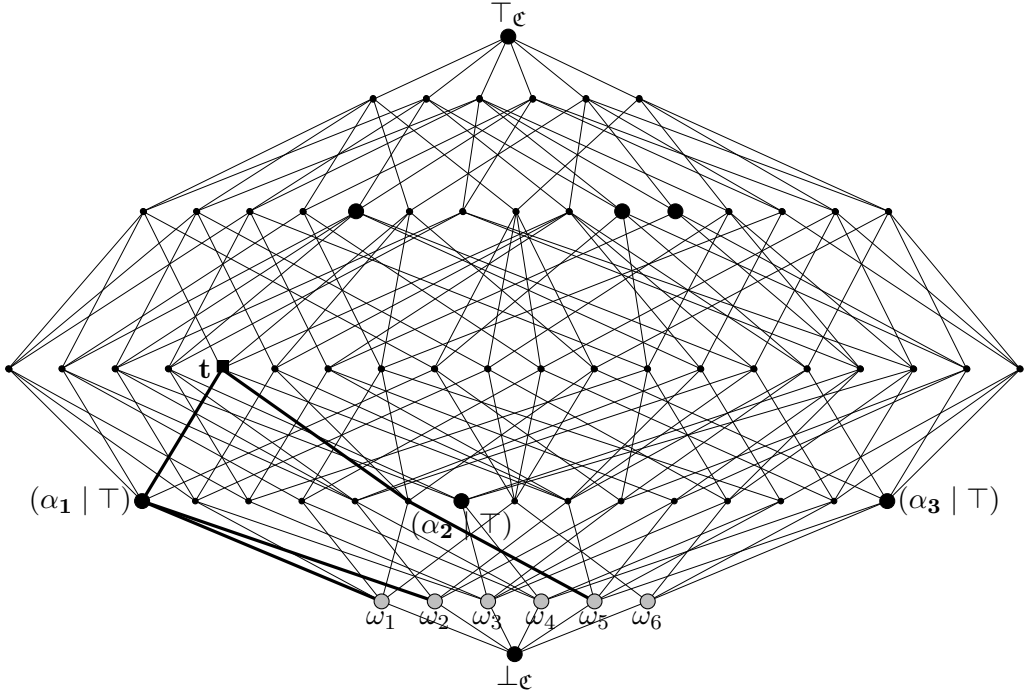 

Thanks to Theorem \ref{th:atoms} and without danger of confusion, for every sequence $\overline{\alpha}\in Seq({\bf A})$, we will henceforth {denote by $\omega_{\overline{\alpha}}$} the atom of $\mathcal{C}({\bf A})$ which is uniquely associated to the sequence $\overline{\alpha}$.

A simple argument about the cardinality of Boolean algebras shows that not every finite Boolean algebra is isomorphic to a finite Boolean algebra of conditionals. In fact, from Theorem \ref{th:atoms}, there is no Boolean algebras of conditionals of cardinality different from $2^{n!}$ for {every} $n$ whence, for instance, there is no Boolean algebra of conditionals with $2^8$ elements. Therefore, the class of all finite Boolean algebras of conditionals forms a proper subset of the class of all finite Boolean algebras. The following corollary shows that the cardinality of a finite Boolean algebra is enough to describe it as an isomorphic copy of some $\mathcal{C}({\bf A})$. 

\begin{corollary}\label{cor:SurjectiveC}
Let ${\bf B}$ be a Boolean algebra of cardinality $2^{n!}$. Then there exists a Boolean algebra ${\bf A}$ of cardinality {$2^n$} such that ${\bf B}\cong\mathcal{C}({\bf A})$. 
\end{corollary}

\begin{proof}
Let $\beta_1,\ldots, \beta_{n!}$ be the atoms of ${\bf B}$ and let $\mathcal{P}=\{P_1,\ldots, P_n\}$ be a partition of $\atom({\bf B})$ in subsets of  $(n-1)!$ elements. For every $P_i=\{\beta_{i_1},\ldots, \beta_{i_{(n-1)!}}\}$, let $a_i=\bigvee_{j=1}^{(n-1)!}\beta_{i_j}$. Since $\mathcal{P}$ is a partition, if $a_i\neq a_j$, then $a_i\wedge a_j=\bot_{\bf B}$ and $\bigvee_{i=1}^{n}a_i=\top_{\bf B}$. Thus, the subalgebra ${\bf A}_\mathcal{P}$ of ${\bf B}$ generated by $a_1,\ldots, a_{n}$ is such that $\atom({\bf A}_\mathcal{P})=\{a_1,\ldots, a_n\}$. Notice, that, if $\mathcal{P}$ and $\mathcal{P}'$ are two different partitions of $\atom({\bf B})$ in subsets of cardinality $(n-1)!$, then ${\bf A}_{\mathcal{P}}\cong{\bf A}_{\mathcal{P}'}$.
Thus, $\mathcal{C}({\bf A}_\mathcal{P})\cong{\bf B}$ since $|\atom(\mathcal{C}({\bf A}_\mathcal{P}))|=|\atom({\bf B})|$. 
\end{proof}

\subsection{Characterising the atoms below a basic conditional}\label{sec:belowbasic}
Now, we will be concerned with the description of the set of atoms of $\mathcal{C}({\bf A})$ below (according to the lattice ordering in  $\mathcal{C}({\bf A})$) a given basic conditional $(a\mid b)$.  In general, for every Boolean algebra ${\bf B}$ and for every $b\in B$, we will henceforth write $\atom_\leq(b)$ to denote the subset of $\atom({\bf B})$  below $b$. Thus, in particular, for every $(a\mid b)\in \mathcal{C}({\bf A})$,
$$
\atom_\leq(a\mid b)=\{\omega_{\overline{\alpha}}\in \atom(\mathcal{C}({\bf A}))\mid \omega_{\overline{\alpha}}\leq (a\mid b)\}.
$$

\begin{proposition}\label{prop:atomsbelow}
Let ${\bf A}$ be a Boolean algebra, let $\overline{\alpha}=\langle \alpha_1, \alpha_2, \ldots, \alpha_{n-1} \rangle \in Seq({\bf A})$ and let $(a\mid b)\in \mathcal{C}({\bf A})$ be such that $a \leq b$. 
Then, the following conditions are equivalent: 
\begin{enumerate}[(i)]
\item $\omega_{\overline{\alpha}}  \in \atom_\leq(a \mid b)$;
\item there is an index $i \leq n-1$ such that  $\alpha_j \leq \neg b$  for all $j < i$ and $\alpha_i \leq a$;
\item $\alpha_i \leq a$ for the smallest index $i \leq n-1$ such that $\alpha_i \leq b$.
\end{enumerate}
\end{proposition}
\begin{proof}
Let us start by showing that (ii) implies (i). 
Since $\omega_{\overline{\alpha}}\leq (a\mid b)=\bigvee_{\beta\leq a}(\beta\mid b)$ iff $\omega_{\overline{\alpha}}\leq (\beta\mid b)$ for some $\beta\leq a$, it is sufficient to prove the claim for the case in which $a=\beta$.  In such a case, if  $\alpha_i = \beta$ and $\neg \alpha_1 \land \ldots \land \neg \alpha_{i-1} \geq b$ then, since  $a\leq b$ and $\alpha_i=\beta=a$, one has $b\geq\alpha_i$  whence $\neg \alpha_1 \land \ldots \land \neg \alpha_{i-1} \geq b\geq \alpha_i$. Therefore, by Proposition \ref{prop:ordine} (iii), 
$$
(\alpha_i \mid \neg \alpha_1 \land \ldots \land \neg \alpha_{i-1}) \leq (\alpha_i\mid b)= (\beta \mid b),
$$ 
and hence $\omega_{\overline{\alpha}}   =  (\alpha_1 \mid \top) \sqcap  \ldots \sqcap  (\alpha_{i} \mid \neg\alpha_1 \land \ldots \land \neg \alpha_{i-1}) \sqcap \dots  \leq (\beta \mid b)$ as well. 

To prove the other direction, (i) implies (ii), we consider two cases:

\vspace{.2cm}

\noindent(a) There is $i \leq n-1$ such that $\alpha_i = \beta$. If $i = 1$, since $\top \geq b$, then the claim is fulfilled. Thus, assume $i > 1$ and $\omega_{\overline{\alpha}}  \sqcap (\beta \mid b) = \omega_{\overline{\alpha}}$, and let us prove $\neg \alpha_1 \land \ldots \land \neg \alpha_{i-1} \geq b$. Indeed, we have the following subcases:

\begin{itemize}

\item  If $\alpha_1 \leq b$, we would have  $(\alpha_1 \mid \top) \sqcap (\beta \mid b) \leq (\alpha_1 \mid b) \sqcap (\alpha_i \mid b) = \bot_\mathfrak{C}$, and hence $\omega_{\overline{\alpha}}  \sqcap (\beta \mid b)  = \bot_\mathfrak{C}$,  contradiction. Therefore $\alpha_1 \leq \neg b$. 

\item  If  $\alpha_2 \leq b$, since $\alpha_1 \leq \neg b$, we would have $(\alpha_2 \mid \neg \alpha_1) \sqcap (\beta \mid b) \leq (\alpha_2 \mid b) \sqcap (\beta \mid b) = \bot_\mathfrak{C}$, and hence $\omega_{\overline{\alpha}}  \sqcap (\beta \mid b)  = \bot_\mathfrak{C}$,  contradiction.  Therefore,  $\alpha_2 \leq \neg b$.

\item[\ldots]

\item If  $\alpha_{i-1} \leq b$, since $\alpha_1 \leq \neg b$, $\alpha_2 \leq \neg b$, \ldots, $\alpha_{i-2} \leq \neg b$, we would have  $(\alpha_{i-1} \mid \neg \alpha_1 \land \ldots \land \alpha_{i-2}) \sqcap (\beta \mid b) \leq (\alpha_2 \mid b) \sqcap (\beta \mid b) = \bot_\mathfrak{C}$, and hence $\omega_{\overline{\alpha}}  \sqcap (\beta \mid b)  = \bot$,  contradiction.  Therefore,  $\alpha_{i-1} \leq \neg b$.

\end{itemize}

As a consequence, $\alpha_1 \lor \ldots \lor \alpha_{i-1} \leq \neg b$ or, equivalently, $\neg \alpha_1 \land \ldots \land \neg \alpha_{i-1} \geq b$.  

\vspace{.2cm} 

\noindent(b) $\beta = \alpha_n$, where $\alpha_n$ is the remaining atom not appearing in $\overline{\alpha}$. In this case, one can show that $\omega_{\overline{\alpha}}  \sqcap (\beta \mid b) = \bot_\mathfrak{C}$, and hence $\omega_{\overline{\alpha}} \not\leq  (\beta \mid b)$. 
Indeed, if $\beta = \alpha_n < b$, it means that $b \geq \alpha_i \lor  \alpha_n$, with $i \leq n-1$. Then, the expression $(\alpha_i \mid \neg \alpha_1 \land \ldots \neg \alpha_{i-1}) = (\alpha_i \mid \alpha_i \lor \ldots \lor \alpha_n)$ appears as a conjunct in the atom  $\omega_{\overline{\alpha}}$. 
Thus,  $\omega_{\overline{\alpha}}  \sqcap (\beta \mid b) \leq (\alpha_i \mid \alpha_i \lor \ldots \lor \alpha_n) \sqcap (\beta \mid b) \leq (\alpha_i \mid \alpha_i \lor \alpha_n) \sqcap (\beta \mid \alpha_i \lor  \alpha_n) = \bot_\mathfrak{C}$. 

Finally, notice that (iii) is just an equivalent rewriting of (ii)  observing that if $\alpha_i \leq a$ then $\alpha_i \leq b$  as well, since we are assuming $a \leq b$. 
\end{proof}

As a consequence of the above characterisation, one can compute the number of atoms below a given conditional.
\begin{corollary}\label{lemma:counting}
Let ${\bf A}$ be a Boolean algebra with $|\atom({\bf A})|=n$.
For every basic conditional $(a\mid b)\in \mathcal{C}({\bf A})$ with $a\leq b$, $|\atom_\leq(a\mid b)|=n!\cdot \frac{|\atom_\leq(a)|}{|\atom_\leq(b)|}$. 
\end{corollary}
\begin{proof} See Appendix.
\end{proof}

\subsection{Equalities and inequalities among conditionals} \label{sub43}
 
The results proved in the previous subsections allow us to determine when two basic conditionals are equal. Further, in this subsection, we will investigate properties regarding the order among conditionals which improve those of Section \ref{sec:BAC}. Let us start with two preliminary lemmas.

\begin{lemma}\label{neq12} Let $ a, b, c \in A$ be such that $\bot < a < b$ and  $a < c$.  Then $(a \mid b) \geq (a \mid c)$ iff $b \leq c$. 
\end{lemma}

\begin{proof} One direction is easy. 
If $b \leq c$ then $b \lor c = c$, and by  Proposition 3.9 (ii), we have $(a \mid b) \sqcap (a \mid c) = (a \mid b \lor c) = (a \mid c)$, that is, $(a \mid b) \geq (a \mid c)$.

As for the other, let $\alpha$ be an atom of $\bf A$ such that $\alpha \leq a$. If $b \nleq c$,  there is an atom $\beta$ such that $\beta \leq b$ but $\beta\not\leq c$ (and thus $\beta \not\leq a$). Then let $\omega_{\overline{\beta}} \in \atom(\mathcal{C}({\bf A}))$ be of the form $\omega_{\overline{\beta}} = (\beta \mid \top) \sqcap (\alpha \mid \neg \beta) \sqcap \ldots $. Then it follows that $\omega_{\overline{\beta}} \leq (a \mid c)$, since $\beta \nleq c$ and both $\alpha \leq c$ and $\alpha \leq a$, but $\omega_{\overline{\beta}} \not\leq (a \mid b)$, since $\beta \leq b$ but $\beta \nleq a$. Hence $(a \mid b) \ngeq (a \mid c)$. 
\end{proof}

As a direct consequence of this lemma, we have the following property that strengthens  Proposition \ref{prop3} (i). 
\begin{corollary}\label{neq1} Let $ a, b, c \in A$ be such that $\bot < a < b$ and $a < c$.  Then $(a \mid b) = (a \mid c)$ iff $b = c$. 
\end{corollary}

\begin{lemma}\label{neq2} Let $a, b, c, d \in A$ be such that $\bot < a < b$ and $\bot < c < d$.  Then $(a \mid b) = (c \mid d)$ iff $a = c$ and $b = d$. 
\end{lemma}

 \begin{proof} One direction is trivial.  
As for the other, 
assume $a \neq c$, and hence there is an atom $\alpha$ of ${\bf A}$ such that $\alpha \leq a$ but $\alpha \not\leq c$. Further, since $c < d$ there exists an atom $\gamma$ of ${\bf A}$ such that $\gamma \leq d$ but $\gamma \nleq c$. Consider any atom $\omega_{\overline{\alpha}}$ of $\mathcal{C}({\bf A})$ of the form  $\omega_{\overline{\alpha}} = (\alpha\mid \top) \sqcap (\gamma \mid \neg \alpha) \sqcap \ldots$. Then, using Proposition \ref{prop:atomsbelow}, one can check that $\omega_{\overline{\alpha}} \leq (a \mid b)$ since $\alpha \leq a$, but $\omega_{\overline{\alpha}} \not\leq (c \mid d)$ since $\alpha \nleq c$ and both $\gamma \leq d$ and $\gamma \nleq c$.  
Therefore,  there are atoms below $(a \mid b)$ that are not below $(c \mid d)$, hence $(a \mid b) \neq (c \mid d)$, contradiction. Hence, it must be $a = c$. 

Finally, we  apply Corollary \ref{neq1} to get $b = d$ as well. 
\end{proof}

As a consequence of the above two lemmas, we can {provide necessary and sufficient conditions for the equality between two basic conditionals.} 

\begin{theorem} \label{equality} For any pair of conditionals $(a \mid b), (c \mid d) \in \mathcal{C}({\bf A})$, we have $(a \mid b) = (c \mid d)$ iff either  $(a \mid b) = (c \mid d) = \top$, or $(a \mid b) = (c \mid d) = \bot$, or $a \land b = c \land d$ and $b = d$. 
\end{theorem}
Thanks to the characterisation Theorem \ref{equality}, one can easily compute the number of (distinct) basic conditionals in a given algebra of conditionals $\mathcal{C}({\bf A})$. Indeed, to compute the number of basic conditionals different from $\top$ and $\bot$ amounts to counting all pairs $(M, N)$ with $\emptyset \neq M \subset N \subseteq \atom({\bf A})$. 

\begin{corollary} Let $\bf A$ be a Boolean algebra with $n$ atoms, i.e. with $| \atom({\bf A}) | = n$. Then the number of basic conditionals in $\mathcal{C}({\bf A})$ is  $\mathsf{bc}(n) = 2 + \sum_{r=2}^{n} \binom{n}{r} \cdot (2^r-2) $.  
\end{corollary}

\begin{proof}  The counting of pairs $(M, N)$ with $\emptyset \neq M \subset N \subseteq \atom({\bf A})$ results from the following observations: 

\begin{itemize}
\item[(1)]  For each $N$ with $r$ elements, there are $2^r-2$ non-empty subsets $M$ of $N$ with less than $r$ elements (note that necessarily $| N | \geq 2$ since $\emptyset \neq M \subset N$); 

\item[(2)]  Hence, there are $\binom{n}{r} \cdot (2^r-2)$ such pairs $(M, N)$ such that $| N | = r$; 

\item[(3)]  Thus, the total number of such pairs will be  $\displaystyle{\sum_{r=2}^{n} \binom{n}{r} \cdot (2^r-2)}$. 
\end{itemize}
Finally, to get the total number of basic conditionals we only need to add 2 (for the $\bot$ and $\top$ conditionals) to the above quantity. 
\end{proof}

For example, in the algebra  $\mathcal{C}({\bf A})$ of Fig. \ref{fig:sixAtoms}, built from the algebra $\bf A$ with 3 atoms ($n = 3$), we have  $\mathsf{bc}(3) = 14$ basic conditionals, out of 64 elements in total. Hence it contains 50 proper compound conditionals.  

Thanks to the results about the atomic structure of the algebras of conditionas $\mathcal{C}({\bf A})$, we can now also provide similar, but partial, characterisation results of when an inequality between two basic conditionals holds, stronger than those  shown  e.g. in  Proposition  \ref{prop:ordine}. 

\begin{lemma}\label{neq22} Let $a, b, c, d \in A$ be such that $\bot < a < b$, $\bot < c < d$. Then:
\begin{enumerate}[(i)]
\item if  $a \leq c$ and $b \geq d$ then $(a \mid b) \leq (c \mid d)$; 
\item if $(a \mid b) \leq (c \mid d)$ then $a \leq c$;
\item  if $c \leq b$ and $(a \mid b) \leq (c \mid d)$ then $b \geq d$.
\end{enumerate}
\end{lemma}

\begin{proof}
\begin{itemize}
\item[(i)] Assume  $a \leq c$ and $b \geq d$. In such a case we have $\bot < c \land b < b$. Indeed, if it were $c \land b = b$,  it would mean $b \leq c$, and then we would have $a < b \leq c < d$, hence $b < d$ that is in contradiction with the hypothesis $b \geq d$. 
Therefore we have the following chain of inequalities:  
 $$
 (a \mid b) \leq (c \mid b) = (c \land b \mid b) \leq (c \land b \mid d) \leq (c \mid d).
 $$
 
Observe that the first and third inequalities are clear from (iii) of Proposition \ref{prop:ordine}, while the second one follows from Lemma \ref{neq12} due to the fact that $c \land b < b$ and $c \land d < d$. 

\item[(ii)] Assume $a \nleq c$, and hence assume there is an atom $\alpha$ of ${\bf A}$ such that $\alpha \leq a$ but $\alpha \not\leq c$. Further, since $c < d$ there exists an atom $\gamma$ of ${\bf A}$ such that $\gamma \leq d$ but $\gamma \nleq c$. Consider any atom $\omega_{\overline{\alpha}}$ of $\mathcal{C}({\bf A})$ of the form  $\omega_{\overline{\alpha}} = (\alpha\mid \top) \sqcap (\gamma \mid \neg \alpha) \sqcap \ldots$. Then, using Proposition \ref{prop:atomsbelow}, one can check that $\omega_{\overline{\alpha}} \leq (a \mid b)$ since $\alpha \leq a$, but $\omega_{\overline{\alpha}} \not\leq (c \mid d)$ since $\alpha \nleq c$ and both $\gamma \leq d$ and $\gamma \nleq c$. Therefore, $(a \mid b) \nleq (c \mid d)$. 

\item[(iii)] Assume $b \ngeq d$. In this case, we can further assume $a \leq c$, otherwise the previous item can be applied. 
Let $\alpha$ be an atom of $\bf A$ such that $\alpha \leq a$, and hence $\alpha \leq b$ and $\alpha \leq c \leq d$ as well. 

If $b \ngeq d$,  there is an atom $\beta$ such that $\beta \leq d$ but $\beta\not\leq b$ (and thus $\beta \not\leq a$ and $\beta \not\leq c$ as well since $a < b$ and $c \leq b$). Then let $\omega_{\overline{\beta}} \in \atom(\mathcal{C}({\bf A}))$ be of the form $\omega_{\overline{\beta}} = (\beta \mid \top) \sqcap (\alpha \mid \neg \beta) \sqcap \ldots $. Then it follows that $\omega_{\overline{\beta}} \leq (a \mid b)$, since $\beta \nleq b$ and both $\alpha \leq a$ and $\alpha \leq b$, but $\omega_{\overline{\beta}} \not\leq (c \mid d)$, since $\beta \leq d$ but $\beta \nleq c$. Therefore, $(a \mid b) \nleq (c \mid d)$. 
\end{itemize}
\end{proof}

Note that the property (i) in the above lemma is  stronger  than both (ii) and (iii) of Proposition \ref{prop:ordine}. 

From these properties, we can express the following characterisation result, although with an additional assumption (the one in (iii) above) that restricts the scope of its application. It remains as an open problem whether this extra condition could be eventually removed. 

\begin{corollary} \label{inequality} For any pair of conditionals $(a \mid b), (c \mid d) \in \mathcal{C}({\bf A})$ such that $c \land d \leq b$, we have $(a \mid b) \leq (c \mid d)$ iff either  $(c \mid d) = \top$, $(a \mid b) = \bot$, or $a \land b \leq c \land d$ and $b \geq d$. 
\end{corollary}

\section{{Two tree-like representations of $\atom({\cal C}({\bf A}))$
  }}\label{sec:tree-like}
In this section we provide two representations of the set $\atom({\cal
  C}({\bf A}))$ of atoms of a conditional Boolean algebra  as trees
that will turn very helpful in the proof of the main result of next
Section \ref{sebsec:measure2}. We advise readers who are not
interested in the full detail of the proof to skim through this
section to get acquainted with the notation, which will be useful
later on.

\subsection{A first representation} \label{firsttree}

Let ${\bf A}$ be a Boolean algebra with $n$ atoms, say $\atom({\bf A}) = \{\alpha_1,\ldots, \alpha_n\}$, and let us inductively define the following rooted trees $\mathbb{T}(m)$ of increasing depth, for $m= 0, 1, \ldots, n$: 
\begin{itemize}
\item[($0$)]  $\mathbb{T}(0)$ consists of only one node: $(\top\mid \top)$. 
\item[($1$)]  $\mathbb{T}(1)$ is obtained from $\mathbb{T}(0)$ by attaching the following $n$ child nodes to $(\top\mid \top)$: $(\alpha_1\mid \top), \ldots, (\alpha_n\mid \top)$. 

\item[($k$)] For $k = 1, \ldots, n-2$, $\mathbb{T}(k+1)$ is obtained by spanning $\mathbb{T}(k)$ in the following way.  For each leaf $(\alpha_i \mid b)$ of $\mathbb{T}(k)$, let $Atup(\alpha_i,b)$ be the set of atoms that appear in the consequents of conditionals along the path from $(\alpha_i \mid b)$ to the root, including $\alpha_i$ itself. Then attach $(\alpha_i \mid b)$ with the $n-k$ child nodes of the form $(\alpha_j \mid b \land \neg \alpha_i)$, for each $\alpha_j \in \atom({\bf A}) \setminus Atup(\alpha_i,b)$. 

\item[($n-1$)]  Put $\mathbb{T} = \mathbb{T}(n-1)$.\footnote{In the above construction,  if we would proceed to the stage $n$, the leaves of $\mathbb{T}(n)$  would be of the form $(\alpha_j\mid\bigwedge_{\alpha_i\neq\alpha_j}\neg\alpha_i) =(\alpha_j\mid \alpha_j) =(\top\mid\top)$. For this reason, and in order not to trivialize the construction, we  stop the definition at level $n-1$. }
\end{itemize}

Figure \ref{figTreeAtoms1} clarifies the above construction in the case of an algebra $\bf A$ with four atoms.
The following hold:
\begin{itemize}
\item[Fact 1:] all the conditionals appearing in $\mathbb{T}$, but the root, are of the form $(\alpha_i\mid b)$ for $\alpha_i\in \atom({\bf A})$ and $b\in  A$;
\item[Fact 2:] at each level of the tree, all conditionals which are children of the same node 
share the same antecedent. {For instance the leafs of $\mathbb{T}(k+1)$ are all of the form $(\alpha_j \mid b \land \neg \alpha_i)$, for each $\alpha_j \in \atom({\bf A}) \setminus Atup(\alpha_i,b)$ (see the construction above).}
\item[Fact 3:] there is a bijective correspondence between the atoms of $\mathcal{C}({\bf A})$ and the paths from the root to the leaves in $\mathbb{T}$.  
\end{itemize}

\begin{figure}[!htbp]\label{figTreeAtoms1}
\definecolor{ududff}{rgb}{0.30196078431372547,0.30196078431372547,1.}
\definecolor{uuuuuu}{rgb}{0.26666666666666666,0.26666666666666666,0.26666666666666666}
\begin{center}
\begin{tikzpicture}[scale=.9]
\draw [dashed,line width=0.8pt] (8.25,-18.83)-- (5.25,-17.83);
\draw [line width=0.4pt] (8.25,-18.83)-- (7.25,-17.83);
\draw [line width=0.4pt] (8.25,-18.83)-- (9.25,-17.83);
\draw [line width=0.4pt] (8.25,-18.83)-- (11.25,-17.83);
\draw [dashed,line width=0.8pt] (5.25,-17.83)-- (2.69,-16.83);
\draw [line width=0.4pt] (5.25,-17.83)-- (3.75,-16.83);
\draw [line width=0.4pt] (5.25,-17.83)-- (4.63,-16.79);
\draw [line width=0.4pt] (7.25,-17.83)-- (5.71,-16.81);
\draw [line width=0.4pt] (7.25,-17.83)-- (6.63,-16.85);
\draw [line width=0.4pt] (7.25,-17.83)-- (7.81,-16.81);
\draw [line width=0.4pt] (9.25,-17.83)-- (8.75,-16.85);
\draw [line width=0.4pt] (9.25,-17.83)-- (9.77,-16.83);
\draw [line width=0.4pt] (9.25,-17.83)-- (10.77,-16.81);
\draw [line width=0.4pt] (11.25,-17.83)-- (11.77,-16.83);
\draw [line width=0.4pt] (11.25,-17.83)-- (12.75,-16.85);
\draw [line width=0.4pt] (11.25,-17.83)-- (13.85,-16.83);
\draw [dashed,line width=0.8pt] (2.69,-16.83)-- (1.25,-15.83);
\draw [line width=0.4pt] (2.69,-16.83)-- (1.92,-15.83);
\draw [line width=0.4pt] (3.75,-16.83)-- (2.62,-15.85);
\draw [line width=0.4pt] (3.75,-16.83)-- (3.25,-15.83);
\draw [line width=0.4pt] (4.63,-16.79)-- (3.89,-15.85);
\draw [line width=0.4pt] (4.63,-16.79)-- (4.6,-15.83);
\draw [line width=0.4pt] (5.71,-16.81)-- (5.11,-15.85);
\draw [line width=0.4pt] (5.71,-16.81)-- (5.71,-15.82);
\draw [line width=0.4pt] (6.63,-16.85)-- (6.25,-15.83);
\draw [line width=0.4pt] (6.63,-16.85)-- (6.84,-15.78);
\draw [line width=0.4pt] (7.81,-16.81)-- (7.25,-15.83);
\draw [line width=0.4pt] (7.81,-16.81)-- (7.84,-15.82);
\draw [line width=0.4pt] (8.75,-16.85)-- (8.67,-15.82);
\draw [line width=0.4pt] (8.75,-16.85)-- (9.25,-15.83);
\draw [line width=0.4pt] (9.77,-16.83)-- (9.67,-15.78);
\draw [line width=0.4pt] (9.77,-16.83)-- (10.25,-15.83);
\draw [line width=0.4pt] (10.77,-16.81)-- (10.8,-15.82);
\draw [line width=0.4pt] (10.77,-16.81)-- (11.4,-15.85);
\draw [line width=0.4pt] (11.77,-16.83)-- (11.91,-15.83);
\draw [line width=0.4pt] (11.77,-16.83)-- (12.62,-15.85);
\draw [line width=0.4pt] (12.75,-16.85)-- (13.25,-15.83);
\draw [line width=0.4pt] (12.75,-16.85)-- (13.89,-15.85);
\draw [line width=0.4pt] (13.85,-16.83)-- (14.59,-15.83);
\draw [line width=0.4pt] (13.85,-16.83)-- (15.25,-15.83);
\begin{scriptsize}
\draw [fill=black] (5.25,-17.83) circle (2.0pt);
\coordinate[label=left: $\alpha_1\mid\top$] (a1) at (5.25,-18.00);
\coordinate[label=below: $\alpha_2\mid\neg\alpha_1$] (a1) at (2.25,-16.80);
\coordinate[label=below: $\alpha_1\mid\neg\alpha_2$] (a1) at (5.44,-16.80);
\draw [fill=black] (8.25,-18.83) circle (2.0pt);
\draw [fill=black] (7.25,-17.83) circle (2.0pt);
\coordinate[label=left: $\alpha_2\mid\top$] (a1) at (7.25,-18.00);
\coordinate[label=right: $\alpha_3\mid\top$] (a1) at (9.15,-18.00);
\coordinate[label=right: $\alpha_4\mid\top$] (a1) at (11.15,-18.00);
\draw [fill=black] (8.25,-18.83) circle (2.0pt);
\draw [fill=black] (9.25,-17.83) circle (2.0pt);
\draw [fill=black] (8.25,-18.83) circle (2.0pt);
\coordinate[label=below: $\top\mid\top$] (a0) at (8.25,-18.83);
\draw [fill=black] (11.25,-17.83) circle (2.0pt);
\draw [fill=black] (5.25,-17.83) circle (2.0pt);
\draw [fill=black] (2.69,-16.83) circle (2.0pt);
\coordinate[label=below: $\alpha_3\mid\neg\alpha_1\wedge\neg\alpha_2$] (a1) at (1.10,-15.80); %(1.80,-15.33);
\draw [fill=black] (5.25,-17.83) circle (2.0pt);
\draw [fill=black] (3.75,-16.83) circle (2.0pt);
\draw [fill=black] (5.25,-17.83) circle (2.0pt);
\draw [fill=black] (4.63,-16.79) circle (2.0pt);
\draw [fill=black] (7.25,-17.83) circle (2.0pt);
\draw [fill=black] (5.71,-16.81) circle (2.0pt);
\draw [fill=black] (7.25,-17.83) circle (2.0pt);
\draw [fill=black] (6.63,-16.85) circle (2.0pt);
\draw [fill=black] (7.25,-17.83) circle (2.0pt);
\draw [fill=black] (7.81,-16.81) circle (2.0pt);
\draw [fill=black] (9.25,-17.83) circle (2.0pt);
\draw [fill=black] (8.75,-16.85) circle (2.0pt);
\draw [fill=black] (9.25,-17.83) circle (2.0pt);
\draw [fill=black] (9.77,-16.83) circle (2.0pt);
\draw [fill=black] (9.25,-17.83) circle (2.0pt);
\draw [fill=black] (10.77,-16.81) circle (2.0pt);
\draw [fill=black] (11.25,-17.83) circle (2.0pt);
\draw [fill=black] (11.77,-16.83) circle (2.0pt);
\draw [fill=black] (11.25,-17.83) circle (2.0pt);
\draw [fill=black] (12.75,-16.85) circle (2.0pt);
\draw [fill=black] (11.25,-17.83) circle (2.0pt);
\draw [fill=black] (13.85,-16.83) circle (2.0pt);
\draw [fill=black] (2.69,-16.83) circle (2.0pt);
\draw [fill=black] (1.25,-15.83) circle (2.0pt);
\draw [fill=black] (2.69,-16.83) circle (2.0pt);
\draw [fill=black] (1.92,-15.83) circle (2.0pt);
\draw [fill=black] (3.75,-16.83) circle (2.0pt);
\draw [fill=black] (2.62,-15.85) circle (2.0pt);
\draw [fill=black] (3.75,-16.83) circle (2.0pt);
\draw [fill=black] (3.25,-15.83) circle (2.0pt);
\draw [fill=black] (4.63,-16.79) circle (2.0pt);
\draw [fill=black] (3.89,-15.85) circle (2.0pt);
\draw [fill=black] (4.63,-16.79) circle (2.0pt);
\draw [fill=black] (4.6,-15.83) circle (2.0pt);
\draw [fill=black] (5.71,-16.81) circle (2.0pt);
\draw [fill=black] (5.11,-15.85) circle (2.0pt);
\draw [fill=black] (5.71,-16.81) circle (2.0pt);
\draw [fill=black] (5.71,-15.82) circle (2.0pt);
\draw [fill=black] (6.63,-16.85) circle (2.0pt);
\draw [fill=black] (6.25,-15.83) circle (2.0pt);
\draw [fill=black] (6.63,-16.85) circle (2.0pt);
\draw [fill=black] (6.84,-15.78) circle (2.0pt);
\draw [fill=black] (7.81,-16.81) circle (2.0pt);
\draw [fill=black] (7.25,-15.83) circle (2.0pt);
\draw [fill=black] (7.81,-16.81) circle (2.0pt);
\draw [fill=black] (7.84,-15.82) circle (2.0pt);
\draw [fill=black] (8.75,-16.85) circle (2.0pt);
\draw [fill=black] (8.67,-15.82) circle (2.0pt);
\draw [fill=black] (8.75,-16.85) circle (2.0pt);
\draw [fill=black] (9.25,-15.83) circle (2.0pt);
\draw [fill=black] (9.77,-16.83) circle (2.0pt);
\draw [fill=black] (9.67,-15.78) circle (2.0pt);
\draw [fill=black] (9.77,-16.83) circle (2.0pt);
\draw [fill=black] (10.25,-15.83) circle (2.0pt);
\draw [fill=black] (10.77,-16.81) circle (2.0pt);
\draw [fill=black] (10.8,-15.82) circle (2.0pt);
\draw [fill=black] (10.77,-16.81) circle (2.0pt);
\draw [fill=black] (11.4,-15.85) circle (2.0pt);
\draw [fill=black] (11.77,-16.83) circle (2.0pt);
\draw [fill=black] (11.91,-15.83) circle (2.0pt);
\draw [fill=black] (11.77,-16.83) circle (2.0pt);
\draw [fill=black] (12.62,-15.85) circle (2.0pt);
\draw [fill=black] (12.75,-16.85) circle (2.0pt);
\draw [fill=black] (13.25,-15.83) circle (2.0pt);
\draw [fill=black] (12.75,-16.85) circle (2.0pt);
\draw [fill=black] (13.89,-15.85) circle (2.0pt);
\draw [fill=black] (13.85,-16.83) circle (2.0pt);
\draw [fill=black] (14.59,-15.83) circle (2.0pt);
\draw [fill=black] (13.85,-16.83) circle (2.0pt);
\draw [fill=black] (15.25,-15.83) circle (2.0pt);
\end{scriptsize}
\end{tikzpicture}
\end{center}
\caption{\small{The tree $\mathbb{T}=\mathbb{T}(3)$ whose $24$ paths describe the atoms of $\mathcal{C}({\bf A})$ with $|\atom({\bf A})|=4$. In particular, notice that the left-most path (dashed in the figure) provides a description for the atom $\omega_{\overline{\alpha}}=(\alpha_1\mid \top)\sqcap(\alpha_2\mid\neg\alpha_1)\sqcap(\alpha_3\mid\neg\alpha_1\wedge\neg\alpha_2)$ where $\overline{\alpha}=\langle \alpha_1,\alpha_2,\alpha_3\rangle$.}}
\end{figure}
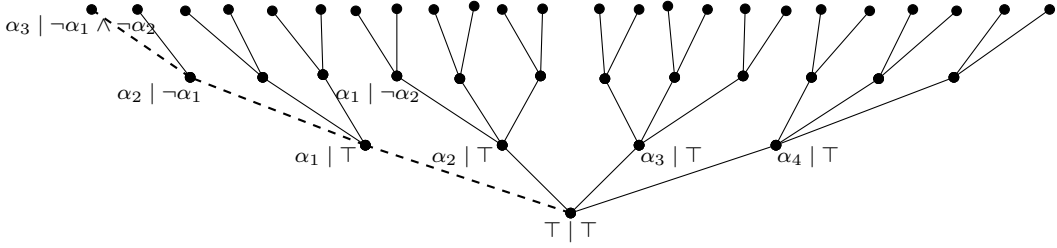

\subsection{A second representation}\label{subsec:second}

Let $\bf A$ be such that $\atom({\bf A}) = \{\alpha_1, \ldots, \alpha_n\}$. Consider a basic conditional of the form $(\alpha\mid b)$, with $\alpha\in \atom({\bf A})$. Without loss of generality, we will henceforth assume $\alpha=\alpha_1$. In what follows, we provide another description of  $\atom_\leq(\alpha_1\mid b)$, alternative to the one given in Proposition \ref{prop:atomsbelow}. To this end, let us start introducing the following notation: 
 for every $j=1,\ldots, n$, let $\mathbb{S}_j$ be the subset of atoms of $\mathcal{C}({\bf A})$ below $(\alpha_1\mid b)$ whose first conjunct is $(\alpha_j\mid \top)$, that is, equivalently,
\begin{equation}\label{eqSj}
\mathbb{S}_j=\{\omega_{\overline{\gamma}}\mid \overline{\gamma}=\langle \gamma_1,\ldots,\gamma_{n-1} \rangle \in Seq({\bf A}), \;   \omega_{\overline{\alpha}}\leq (\alpha_1\mid b) \mbox{ and }\gamma_1=\alpha_j\}.
\end{equation}
 Obviously $\atom_\leq(\alpha_1\mid b)=\bigcup_{j=1}^n \mathbb{S}_j$. Moreover, by construction, the sets $\mathbb{S}_j$ satisfy the following properties.
\begin{lemma}\label{lemma:property1Sj}
For every basic conditional of the form $(\alpha_1\mid b)$, the family $\{\mathbb{S}_j\}_{j=1,\ldots,n}$ is a partition of $\atom_{\leq}(\alpha_1\mid b)$. 
Further, the following properties hold:
\begin{enumerate}[(i)]
\item $\mathbb{S}_1=\{\omega_{\overline{\gamma}}\mid \langle \gamma_1,\ldots, \gamma_{n-1} \rangle\in Seq({\bf A}),\; \gamma_1=\alpha_1\}$.
\item For every $j\geq 2$, if $\alpha_j\leq b$, then $\mathbb{S}_j=\emptyset$.
\end{enumerate}
\end{lemma}
\begin{proof}
We already noticed that $\atom_\leq(\alpha_1\mid b)=\bigcup_{j=1}^n \mathbb{S}_j$ and moreover, it is immediate to check that, for $j_1\neq j_2$, $\mathbb{S}_{j_1}\cap\mathbb{S}_{j_2}=\emptyset$. Thus, it is left to prove (i) and (ii).
\vspace{.2cm}

\noindent
(i). It is clear that any atom with first coordinate $\alpha_1$ is below $(\alpha_1 \mid b)$ since $(\alpha_1 \mid \top) \sqcap (\beta_2 \mid \neg \alpha_1) \sqcap \ldots \leq (\alpha_1 \mid \top) \leq (\alpha_1 \mid b)$. 
\vspace{.2cm}

\noindent(ii). It follows from Proposition \ref{prop:atomsbelow}.
\end{proof}

For every pair of sequences of symbols $\sigma_1,\sigma_2$, we will write $\sigma_1 \ll \sigma_2$ to denote that $\sigma_1$ is an {\em initial segment} of $\sigma_2$.  
\begin{definition}\label{def:blocks}
Let ${\bf A}$ be a Boolean algebra with $n$ atoms and let $1\leq i\leq n-1$. Then, for any sequence $\overline{\alpha}=\langle\alpha_{1},\ldots,\alpha_{i}\rangle\in Seq_i({\bf A})$, we define:
$$
\llb \alpha_{1},\ldots, \alpha_{i}\rrb=\{\omega_{\overline{\gamma}}\in \atom(\mathcal{C}({\bf A}))\mid \langle\alpha_{1},\ldots,\alpha_{i}\rangle\ll \overline{\gamma} \}.
$$
\end{definition}
In other words, $\llb \alpha_{1},\ldots, \alpha_{i}\rrb$ denotes the subset of atoms of $\mathcal{C}({\bf A})$ of the form
$\omega_{\overline{\gamma}}$, with $\overline{\gamma}=\langle \gamma_1,\ldots, \gamma_{n-1} \rangle$ where $\gamma_1=\alpha_{1}, \ldots, \gamma_i=\alpha_{i}$.

\begin{proposition}\label{prop:Sj}
Let ${\bf A}$ be a Boolean algebra with atoms $\alpha_1,\ldots, \alpha_n$, and let $b\in A$ be such that $\neg b=\beta_1\vee\ldots\vee\beta_k$ with $\beta_k=\alpha_n$. For every $t=2,\ldots, k-1$, let $\Pi_t$ denote the set of all permutations $\pi:\{1,\ldots, t\}\to\{1,\ldots, t\}$. Then,
$$
\mathbb{S}_n=\llb\alpha_n,\alpha_1\rrb\cup\bigcup_{i=1}^{k-1}\llb\alpha_n,\beta_i,\alpha_1\rrb\cup\bigcup_{t=2}^{k-1}\bigcup_{\pi\in \Pi_t}\llb\alpha_n, \beta_{\pi(1)},\ldots, \beta_{\pi(t)},\alpha_1\rrb.
$$
\end{proposition}
\begin{proof}
Immediate consequence of the definition of $\llb\alpha_n,\alpha_1\rrb$, $\llb\alpha_n,\beta_i,\alpha_1\rrb$ and $\llb\alpha_n, \beta_{\pi(1)},\ldots, \beta_{\pi(t)},\alpha_1\rrb$.
\end{proof}
Obviously, renaming the indexes of the atoms of ${\bf A}$, the above proposition provides a description for all $\mathbb{S}_j$'s. However, we preferred to show the case of $\mathbb{S}_n$ since that will be the case we shall need in Subsection \ref{sebsec:measure2}.

Let us start defining a tree whose nodes are sets of atoms in $\mathbb{S}_n$ as in Proposition \ref{prop:Sj} above. 
Further, keeping the same notation as in the statement of Proposition \ref{prop:Sj} above, we denote by $\beta_1,\ldots,\beta_k$ the elements of $\atom_\leq(\neg b)$ 
and $\beta_k=\alpha_n$. 

\begin{definition}\label{def:treeT}
With the above premises, we define the tree $\mathbb{B}$ in the following inductive way.
\begin{itemize}
\item[0.] The root of $\mathbb{B}$ is $\llb\alpha_n,\alpha_1\rrb$.
\item[1.] The child nodes of $\llb\alpha_n,\alpha_1\rrb$ are $\llb\alpha_n,\beta_1,\alpha_1\rrb,\ldots, \llb\alpha_n,\beta_{k-1},\alpha_1\rrb$.
\item[\ldots]
\item[t.] The child nodes of $\llb\alpha_n,\beta_{\pi(1)},\ldots, \beta_{\pi(t)},\alpha_1\rrb$ are $\llb\alpha_n,\beta_{\pi(1)},\ldots, \beta_{\pi(t)},\beta_{l_1},\alpha_1\rrb,\\ \ldots, \llb\alpha_n,\beta_{\pi(1)},\ldots, \beta_{\pi(t)},\beta_{l_m},\alpha_1\rrb$, where $\{\beta_{l_1},\ldots, \beta_{l_m}\}=\{\beta_1,\ldots, \beta_{k-1}\}\setminus\{\beta_{\pi(1)},\ldots, \beta_{\pi(t)}\}$. 
\end{itemize}
Given any  node $\llb\alpha_n,\beta_{\pi(1)},\ldots, \beta_{\pi(t)},\alpha_1\rrb$ of $\mathbb{B}$, we will henceforth denote by $\mathbb{B}\llb\alpha_n,\beta_{\pi(1)},\ldots, \beta_{\pi(t)},\alpha_1\rrb$ the subtree of $\mathbb{B}$ generated by $\llb\alpha_n,\beta_{\pi(1)},\ldots, \beta_{\pi(t)},\alpha_1\rrb$. 

\end{definition}
We clarify the above construction with the following example.
\begin{example}
In order  not to trivialize the example, let us consider the algebra $\mathcal{C}({\bf A})$  with $\atom({\bf A})=\{\alpha_1,\ldots,\alpha_5\}$. Let $b=\alpha_1\vee\alpha_2$, whence $\neg b=\alpha_3\vee\alpha_4\vee\alpha_5$, and $\mathbb{S}_5=\{\omega_{\overline{\gamma}}\in \atom(\mathcal{C}(\alpha))\mid \gamma_1=\alpha_5,\; \omega_{\overline{\gamma}}\leq (\alpha_1\mid b)\}$. Then, according to the above definition, we first need to consider the following subsets of $\mathbb{S}_5$: 
\begin{itemize}
\item[] $\llb\alpha_5,\alpha_1\rrb=\{\omega_{\overline{\gamma}}\mid \langle 5,1\rangle\ll\overline{\gamma}\} =$ \\ $ \{\omega_{\overline{\gamma}}\mid \overline{\gamma}\in \{ \langle 5, 1, 2,3\rangle, \langle 5, 1, 2, 4\rangle, \langle 5, 1, 3,2 \rangle, \langle 5, 1, 4,2 \rangle, \langle 5, 1, 3, 4\rangle, \langle 5, 1, 4, 3\rangle \}\}$; 
\item[] $\llb\alpha_5,\alpha_3,\alpha_1\rrb=\{\omega_{\overline{\gamma}}\mid \langle 5,3,1\rangle\ll\overline{\gamma}\}=\{\omega_{\overline{\gamma}}\mid \overline{\gamma}\in \{\langle 5,3,1, 2\rangle, \langle 5,3,1, 4\rangle\}\}$;
\item[] $\llb\alpha_5,\alpha_4,\alpha_1\rrb=\{\omega_{\overline{\gamma}}\mid \langle 5,4,1\rangle\ll\overline{\gamma}\}=\{\omega_{\overline{\gamma}}\mid \overline{\gamma}\in \{\langle 5,4,1, 2\rangle, \langle 5,4,1, 3\rangle\}\}$;
\item[] $\llb\alpha_5,\alpha_3,\alpha_4,\alpha_1\rrb=\{\omega_{\overline{\gamma}}\mid \overline{\gamma}=\langle 5,3,4,1\rangle\}$;
\item[] $\llb\alpha_5,\alpha_4,\alpha_3,\alpha_1\rrb=\{\omega_{\overline{\gamma}}\mid \overline{\gamma}=\langle 5,4,3,1\rangle\}$.
\end{itemize}
It is now clear that $\mathbb{S}_5=\llb\alpha_5,\alpha_1\rrb\cup\llb\alpha_5,\alpha_3,\alpha_1\rrb\cup \llb\alpha_5,\alpha_4,\alpha_1\rrb\cup\llb\alpha_5,\alpha_3,\alpha_4,\alpha_1\rrb\cup \llb\alpha_5,\alpha_4,\alpha_3,\alpha_1\rrb$. The resulting tree $\mathbb{B}$ is hence depicted as in Figure \ref{fig:treeB}.
\qed
\end{example}

\begin{figure}[!h]\label{fig:treeB}
\definecolor{ududff}{rgb}{0.30196078431372547,0.30196078431372547,1.}
\definecolor{uuuuuu}{rgb}{0.26666666666666666,0.26666666666666666,0.26666666666666666}
\begin{center}
\begin{tikzpicture}[scale=1.5]
\draw [line width=0.4pt] (0,0)-- (-1,1);
\draw [line width=0.4pt] (0,0)-- (1,1);
\draw [line width=0.4pt] (-1,1)-- (-1,2);
\draw [line width=0.4pt] (1,1)-- (1,2);
\begin{scriptsize}
\draw [fill=black] (0,0) circle (2.0pt);
\coordinate[label=below:{$\llb\alpha_5,\alpha_1\rrb$}] (a1) at (0,0);
\draw [fill=black] (-1,1) circle (2.0pt);
\coordinate[label=left:{$\llb\alpha_5,\alpha_3,\alpha_1\rrb$}] (a1) at (-1,1);
\draw [fill=black] (1,1) circle (2.0pt);
\coordinate[label=right:{$\llb\alpha_5,\alpha_4,\alpha_1\rrb$}] (a1) at (1,1);
\draw [fill=black] (-1,2) circle (2.0pt);
\coordinate[label=left:{$\llb\alpha_5,\alpha_3,\alpha_4,\alpha_1\rrb$}] (a1) at (-1,2);
\draw [fill=black] (1,2) circle (2.0pt);
\coordinate[label=right:{$\llb\alpha_5,\alpha_4,\alpha_3,\alpha_1\rrb$}] (a1) at (1,2);
\end{scriptsize}
\end{tikzpicture}
\end{center}
\caption{\small{The tree $\mathbb{B}$ for $\mathbb{S}_5$ when ${\bf A}$ has 5 atoms and $b=\alpha_1\vee\alpha_2$.}}
\end{figure}
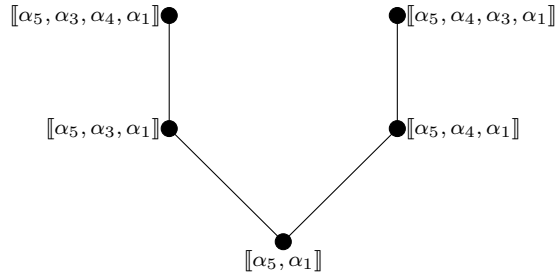

\section{Probabilities on Boolean algebras of conditionals}\label{sec:measures0}
With the desired Boolean algebraic structure for conditionals in
place, we are now in a position to tackle the main issue of this paper
which is the relation between  conditional probability functions on  a
Boolean algebra ${\bf A}$ and simple (i.e. ``unconditional'')
probabilities  on the conditional algebra $\mathcal{C}({\bf A})$. In
the following subsections we shall address two questions that go in
the direction of clarifying when, and under which conditions, a
probability function on a conditional algebra can be regarded as a conditional probability. In particular, in Subsection \ref{sebsec:measure1} we will determine under which conditions a simple measure on $\mathcal{C}({\bf A})$ satisfies the axioms of a conditional probability function on ${\bf A}$, while in Subsection \ref{sebsec:measure2} we will prove our main result, namely, a canonical way to define a simple measure on $\mathcal{C}({\bf A})$ which agrees on each basic conditional with a given conditional probability on $\bf A$.  
\subsection{Preliminary observations about probabilities on $\mathcal{C}({\bf A})$}\label{sebsec:measure1}
We assume the reader to be familiar with the usual and well-known notion of {\em (simple or
  unconditional) finitely-additive probability}. 
Let us recall the notion of a {\em conditional probability}
map which we take, with inessential variations,  from \cite[Definition 3.2.3]{Halpern2003}. See also \cite{Popper} and \cite{DeFinetti1935} where conditional probability was firstly considered as a primitive notion.

\begin{definition}
  For a Boolean algebra ${\bf A}$, a two-place function
  $CP:A\times A'\to[0,1]$, is a \emph{conditional probability} if it
  satisfies the following conditions, where we write, as usual, $(x \mid y)$ instead of $(x, y)$:
\begin{itemize}
\item[(CP1)] for all $b\in A'$, $CP(b\mid b)=1$;
\item[(CP2)] if $a_1, a_2\in A$, $a_1\wedge a_2=0$ and $b\in A'$,
  $CP(a_1\vee a_2\mid b)=CP(a_1\mid b)+CP(a_2\mid b)$;
\item[(CP3)] if $a\in A$ and $b\in A'$,
  $CP(a\mid b)=CP(a\wedge b\mid b)$;
\item[(CP4)] if $a \in A$ and $b, c\in A'$ with $a\leq b \leq c$,
  then $CP(a\mid c)= CP(a\mid b)\cdot  CP(b\mid c)$.
\end{itemize} 
\end{definition}

In what follows, for the sake of a simpler notation, if  $\mu:\mathcal{C}({\bf A})\to[0,1]$ is a (unconditional) probability we will write $\mu(a \mid b)$ instead of $\mu((a \mid b))$.

\begin{remark}\label{rem:subProb1}
Take any probability $\mu$ on $\mathcal{C}({\bf A})$ and fix an element $b \in A'$. Then, ${\bf A}\mid b$ is a Boolean subalgebra of $\mathcal{C}({\bf A})$ from Corollary \ref{subalgebra}, and the restriction $\mu_b$ of $\mu$ to ${\bf A}\mid b $ is a probability measure. Notice that $\mu_b$ also satisfies axiom (CP3). \qed
\end{remark}

Therefore, a simple  probability on $\mathcal{C}({\bf A})$, once restricted to basic conditionals, always satisfies properties  (CP1), (CP2) and (CP3) above. However, axiom (CP4), also known as {\em chain rule}, does not always hold as we show in the next example. 

\begin{example}\label{ex:nonCond}
In the conditional algebra $\mathcal{C}({\bf A})$ of Example \ref{Ex:atoms}, consider the elements $(\alpha_1\mid \top)$, $(\alpha_1\mid \alpha_1\vee \alpha_2)$ and $(\alpha_1\vee \alpha_2\mid \top)$. Clearly, in the Boolean algebra ${\bf A}$, we have $\alpha_1\leq \alpha_1\vee \alpha_2\leq \top$ and hence $(\alpha_1\mid \top)=(\alpha_1\mid \alpha_1\vee \alpha_2)\sqcap (\alpha_1\vee \alpha_2\mid\top)$ by Proposition \ref{prop2} \eqref{e8}. As usual, for every element $t\in \mathcal{C}({\bf A})$, let us write $\atom_\leq(t)$ to denote the set of atoms of $\mathcal{C}({\bf A})$ below $t$. Then, in particular, 
\begin{itemize}
\item[1.] $\atom_\leq(\alpha_1\mid \top)=\{\omega_1,\omega_2\}$,
\item[2.] $\atom_\leq(\alpha_1\vee \alpha_2\mid \top)=\atom_\leq(\alpha_1\mid \top)\sqcup \atom_\leq(\alpha_2\mid \top)=\{\omega_1, \omega_2, \omega_3, \omega_4\}$,
\item[3.] $\atom_\leq(\alpha_1\mid \alpha_1\vee \alpha_2)=\{\omega_1,\omega_2,\omega_5\}$. 
\end{itemize}
Notice that $\atom_\leq(\alpha_1\mid \top)\cap \atom_\leq(\alpha_1\vee \alpha_2\mid \top) = \atom_\leq(\alpha_1\mid \top) \neq \emptyset$ and in particular 
$$
\atom_\leq(\alpha_1\vee \alpha_2\mid \top)\setminus \atom_\leq(\alpha_1\mid \top) = \atom_\leq(\alpha_2\mid \top) \neq \emptyset.
$$ 
Let  $p:\atom(\mathcal{C}({\bf A}))\to[0,1]$ be the probability distribution defined by the following stipulation: 
$$
p(x)=\left\{
\begin{array}{ll}
0 & \mbox{ if }x\in \{\omega_1, \omega_2\},\\
1/4 &\mbox{ if }x\in \{\omega_3, \omega_4, \omega_5, \omega_6\}.
\end{array}
\right.
$$
Thus, let $\mu_p:\mathcal{C}({\bf A})\to[0,1]$ be the probability measure on $\mathcal{C}({\bf A})$ induced by $p$: for all $t\in \mathcal{C}({\bf A})$, $\mu_p(t)=\sum_{\omega\in \atom_{\leq}(t)}p(\omega)$. In particular we have: $\mu_p(\alpha_1\mid \top)=p(\omega_1)+p(\omega_2)=0$, while $\mu_p(\alpha_1\vee \alpha_2\mid \top)=p(\omega_1)+p(\omega_2)+p(\omega_3)+p(\omega_4)=1/2$ and $\mu_p(\alpha_1\mid \alpha_1\vee \alpha_2)=p(\omega_1)+p(\omega_2)+p(\omega_5)=1/4$.  Hence 
$$
0=\mu_p(\alpha_1\mid \top)\neq \mu_p(\alpha_1\mid \alpha_1\vee \alpha_2)\cdot \mu_p(\alpha_1\vee \alpha_2\mid \top)=1/8
$$ and thus $\mu_p$ does not satisfy (CP4) and is not a conditional probability.

One might wonder if the failure of (CP4) is a consequence of the fact that the distribution $p$ assigns $0$ to some atoms, and hence the measure $\mu_p$ is not positive. This is not the case. Indeed, consider the following distribution parametrized by $\epsilon$:
$$
p_\epsilon(x)=\left\{
\begin{array}{ll}
\epsilon & \mbox{ if }x\in \{\omega_1, \omega_2\},\\
1/4 - \epsilon/2 &\mbox{ if }x\in \{\omega_3, \omega_4, \omega_5, \omega_6\}.
\end{array}
\right.
$$
The equation $\mu_{p_\epsilon}(\alpha_1\mid \top) =  \mu_{p_\epsilon}(\alpha_1\mid \alpha_1\vee \alpha_2)\cdot \mu_{p_\epsilon}(\alpha_1\vee \alpha_2\mid \top)$ has solution only for  $\epsilon = 1/2$ and $\epsilon= 1/6$. In particular, for every $1/6<\epsilon<1/2$, the probability $\mu_{p_\epsilon}$ is positive but does not satisfy (CP4).
\qed
\end{example}

The above example can be easily generalised to algebras $\mathcal{C}({\bf A})$ with $\bf A$ having more than 3 atoms. In other words, discarding trivial cases, only a proper subclass of probabilities on a conditional algebra $\mathcal{C}({\bf A})$  gives rise to conditional probabilities on ${\bf A}$, and obviously, these are those satisfying the condition %Precisely, 
$$
\mu((a\mid b)\sqcap (b\mid c))=\mu(a\mid b)\cdot \mu(b\mid c), \mbox{ for all } a, b, c \in A' \mbox{ s.t. } a\leq b\leq c, 
$$
which is equivalent to the chain rule in condition (CP4) above as in $\mathcal{C}({\bf A})$, $(a\mid c)=(a\mid b)\sqcap (b\mid c)$ under the hypothesis that $a\leq b\leq c$.
Note that it amounts in turn to require that, in $\mathcal{C}({\bf
  A})$, any pair of conditional events $(a\mid b)$ and $(b\mid c)$,
with $a \leq b \leq c$, to be stochastically independent with respect to %the simple probability 
$\mu$. This motivates our terminology in the following definition.%\footnote{

\begin{definition}[Separable probabilities]
A probability $\mu:\mathcal{C}({\bf A})\to[0,1]$ is said to be {\em separable} if $\mu$ satisfies the chain rule (CP4).
\end{definition}

The following are two significant examples of separable probabilities on $\mathcal{C}({\bf A})$. 

\begin{example} 
\mbox{} 
\begin{itemize}
\item[(i)] First notice that a map $P$ from a finite Boolean algebra ${\bf A}$ to the  Boolean algebra of classical truth-values ${\bf 2}=\{0,1\}$ is a probability iff it is a homomorphism and hence a truth-valuation. Now, every probability $P$ of $\mathcal{C}({\bf A})$ into  ${\bf 2}$ is separable according to the previous definition.  Indeed, $P: \mathcal{C}({\bf A}) \to \{0,1\}$ satisfies (CP4) since, due to condition (C5) in Proposition \ref{prop2}, if $a \leq b \leq c$ then 
$$
P(a \mid c) = P( (a\mid b)\sqcap  (b\mid c)) = \min(P(a \mid b), P(b \mid c)) = P(a \mid b) \cdot P(b \mid c).
$$ 
Therefore, every $\{0,1\}$-valued probability on $\mathcal{C}({\bf A})$ is a conditional probability on ${\bf A}$.
\item[(ii)] For every finite Boolean algebra ${\bf A}$, the probability measure $\mu_u: \mathcal{C}({\bf A})\to[0,1]$ induced by the uniform distribution on the atoms of $\mathcal{C}({\bf A})$, i.e. the one defined by $\mu_u(\omega) = 1/ |\atom(\mathcal{C}({\bf A}))|$ for every $\omega \in \atom(\mathcal{C}({\bf A}))$,  is separable. 
Indeed, by Corollary \ref{lemma:counting}, one has  $\mu_u(a \mid b) = |\atom_\leq(a)| / |\atom_\leq(b)|$  if $a \leq b$.
\end{itemize}
\end{example}

We finish this section with an observation about convexity. It is well-known that every probability  on a finite  Boolean algebra ${\bf B}$ is a convex combination of homomorphisms of ${\bf B}$ into ${\bf 2}=\{0,1\}$. So in particular this is the case for $\mathcal{C}({\bf A})$. As shown in the example above, all these homomorphisms are separable. Therefore, since  not every probability on $\mathcal{C}({\bf A})$ with $|\atom(A)|\geq 3$ is separable 
the following corollary immediately follows.  

\begin{corollary}\label{cor1}
 Let ${\bf A}$ be a Boolean algebra such that  $|\atom(A)|\geq 3$. Then, the set of separable probabilities on $\mathcal{C}({\bf A})$ is not convex.
\end{corollary}

\subsection{Conditional probabilities on $\bf A$ as simple probabilities on  $\mathcal{C}({\bf A})$}\label{sebsec:measure2}

In this subsection we finally address the fundamental question that has motivated our investigation, namely: 

\begin{itemize}

\item[] ``Given a (positive) probability on an algebra of events $\bf A$, $P:{\bf A}\to[0,1]$, determine whether it can always be extended to a probability on the algebra of conditionals events $\mathcal{C}({\bf A})$, $\mu_P: \mathcal{C}({\bf A})\to [0, 1]$, agreeing on all basic conditionals,  that is, 
\begin{equation}\label{mainQuestion} \mu_P(a \mid b) = \frac{P(a\wedge b)}{P(b)} 
\end{equation}
for any basic conditional $(a \mid b) \in \mathcal{C}({\bf A})$.''
\end{itemize}

This is, in a slightly different setting, what Goodman and Nguyen call
in \cite{GN94} the {\em strong conditional event problem}. They solve
it positively by defining conditional events as countable unions of
{\em special  cylinders} in the infinite algebra ${\bf A}^\infty$ whose set
of atoms is $$\atom({\bf A}^\infty) = \{ (\alpha_1, \alpha_2, \ldots ):
\alpha_i \in \atom({\bf A}), i= 1, 2, \ldots \} = (\atom({\bf
  A}))^{\mathbb{N}}$$ i.e.\ infinite sequences of atoms of ${\bf A}$,
and by defining a probability $\hat{P}$ on ${\bf A}^\infty$ as the
product probability measure with identical marginals $P$ on each
factor space.

 In this section we show we can also solve the problem in a finitary setting by defining a suitable probability $\mu_P: \mathcal{C}({\bf A})\to [0, 1]$  on the {\em finite} Boolean algebra  $\mathcal{C}({\bf A})$.

To this end, given a positive probability on ${\bf A}$, we start defining  
 the following map on $\atom(\mathcal{C}({\bf A}))$. 

\begin{definition} Let $P: {\bf A} \to [0, 1]$ be a positive probability. Then we define the map 
$\mu_P:\atom(\mathcal{C}({\bf A}))\to[0,1]$ by the following stipulation: 
for every $\overline{\alpha}=\langle \alpha_1,\ldots, \alpha_{n-1}\rangle\in Seq({\bf A})$,
$$
\mu_P(\omega_{\overline{\alpha}}) = P(\alpha_1) \cdot P(\alpha_2 \mid \neg \alpha_1) \cdot \ldots \cdot P(\alpha_{n-1} \mid \neg\alpha_1 \land \ldots \land \neg \alpha_{n-2})
$$
or, equivalently,
$$
\mu_P(\omega_{\overline{\alpha}}) = P(\alpha_1) \cdot P(\alpha_2 \mid \alpha_2 \lor \ldots \lor \alpha_{n-1} \lor \alpha_n) \cdot \ldots \cdot P(\alpha_{n-1} \mid \alpha_{n-1} \lor \alpha_n).
$$
\end{definition}
Now, we show that the map $\mu_P$ is indeed a probability distribution. 

\begin{lemma}\label{MuPProb} The map $\mu_P$ is  a probability distribution on $\atom(\mathcal{C}({\bf A}))$, that is, 
$$\sum_{\overline{\alpha}  \in Seq({\bf A})} \mu_P(\omega_{\overline{\alpha}}) = 1.$$
\end{lemma}

\begin{proof} See Appendix.
\end{proof}

Having  a distribution $\mu_P$ on the atoms of $\mathcal{C}({\bf A})$, we can define the corresponding probability measure (that we shall keep denoting by $\mu_P$) on $\mathcal{C}({\bf A})$ in the obvious way. 

\begin{definition} Given a positive probability $P$ on $\bf A$, the probability measure on $\mathcal{C}({\bf A})$ induced by the distribution $\mu_P$ on $\atom(\mathcal{C}({\bf A}))$, i.e. for $
t \in \mathcal{C}({\bf A})$, 
\begin{equation}\label{MuP}
 \mu_P(t) = \sum_{\omega \in  \atom(\mathcal{C}({\bf A})), \omega \leq t}\mu_P(\omega)
\end{equation}
will be called the {\em canonical extension} of $P$ to $\mathcal{C}({\bf A})$. 
\end{definition}

The next three lemmas provide a necessary technical  preparation for the main result of this section, namely Theorem \ref{thm:main1}. For them, we invite the reader to recall the main definitions and constructions of Subsection \ref{sec:belowbasic} and, in particular, Definitions \ref{def:blocks} and \ref{def:treeT}.
With an abuse of notation, for every subset $X$ of $\atom(\mathcal{C}({\bf A}))$ we write $\mu_P(X)$ for $\mu_P(\bigvee_{\omega\in X}\omega)=\sum_{\omega\in X}\mu_P(\omega)$. 

\begin{lemma}\label{lemma:prep1}
Let $\alpha_{i_1},\ldots,\alpha_{i_t}\in \atom({\bf A})$. Then
\begin{enumerate}[(i)]
\item
$\mu_P(\llb \alpha_{i_1},\ldots,\alpha_{i_t} \rrb)=P(\alpha_{i_1})\cdot \frac{P(\alpha_{i_2})}{P(\neg \alpha_{i_1})}\cdot\ldots\cdot\frac{P(\alpha_{i_t})}{P(\neg\alpha_{i_1}\wedge\neg\alpha_{i_2}\wedge\ldots\wedge\neg\alpha_{i_{t-1}})}.%\cdot\frac{P(\alpha_1)}{P(\neg\alpha_n\wedge\neg\beta_1\wedge\ldots\wedge\neg\beta_{t})}.
$
\item $\mu_P(\llb \alpha_{i_1},\ldots,\alpha_{i_t} \rrb)=\mu_P(\llb \alpha_{i_1},\ldots,\alpha_{i_{j-1}},\alpha_{i_{j+1}},\ldots,\alpha_{i_t} \rrb)\cdot \frac{P(\alpha_{i_j})}{P(\neg\alpha_{i_1}\wedge\neg\alpha_{i_2}\wedge\ldots\wedge\neg\alpha_{i_{t-1}})}.$
\end{enumerate}
\end{lemma}
\begin{proof} See Appendix.

\end{proof}
In the next two lemmas, we will present results which allow us to compute the measure $\mu_P$ of a basic conditional $(\alpha\mid b)$, for $\alpha\in \atom({\bf A})$. In particular, adopting the same conventions used in Subsection \ref{subsec:second} we will henceforth fix, without loss of generality, $\alpha=\alpha_1$.
Also, we invite the reader to remind the definition of $\llb\alpha_1,\ldots, \alpha_i\rrb$ as the set of atoms of $\mathcal{C}({\bf A})$ whose initial segment is $\langle \alpha_1,\ldots, \alpha_i\rangle$ (Definition \ref{def:blocks}) and the definition of the tree $\mathbb{B}$ (Definition \ref{def:treeT}) together with the definition of subtree of $\mathbb{B}$ generated by a $\llb\alpha_1,\ldots, \alpha_i\rrb$.

\begin{lemma}\label{lemma:cases0}
Let $(\alpha_1\mid b)$ be a basic conditional such that 
$\neg b=\beta_1\vee\ldots\vee \beta_k$ and $\beta_k=\alpha_n$. Then the following holds: 
for all $t\in \{1,\ldots, k-1\}$, 
$$
\mu_P(\mathbb{B}\llb\alpha_n,\beta_1,\ldots, \beta_t,\alpha_1\rrb)=\mu_P(\llb\alpha_n,\beta_1,\ldots, \beta_{t-1},\alpha_1\rrb)\cdot \frac{P(\beta_t)}{P(b)},
$$
where $\mathbb{B}\llb\alpha_n,\beta_1,\ldots, \beta_t,\alpha_1\rrb$ is understood as the set of all the atoms belonging to its nodes.
\end{lemma}
\begin{proof} See Appendix.

\end{proof}
For the next one, recall the definition of the sets $\mathbb{S}_j$ as we did in Subsection \ref{subsec:second} and, in particular, Lemma \ref{lemma:property1Sj}.
\begin{lemma}\label{lemma:cases}
Let $(\alpha_1\mid b)$ be a basic conditional and let $b\geq \alpha_1$. 
Then, 
\begin{enumerate}[(i)]
\item $\mu_P(\mathbb{S}_1)=P(\alpha_1)$;
\item for any $2\leq j\leq n$, $\mu_P(\mathbb{S}_j)=\displaystyle{\frac{P(\alpha_1)\cdot P(\alpha_j)}{P(b)}}$.
\end{enumerate}
\end{lemma}
\begin{proof} See Appendix.
\end{proof}

We can now prove the following.
\begin{theorem}\label{thm:main1}
For every positive probability measure $P$ on ${\bf A}$, the canonical extension $\mu_P$ on $\mathcal{C}({\bf A})$ is such that, for every basic conditional $(a\mid b)$, 
$$
\mu_P(a\mid b)=\frac{P(a\wedge b)}{P(b)}.
$$
\end{theorem}
\begin{proof}
Let $P$ be given as in the hypothesis, and let $\mu_P$ be defined on $\mathcal{C}({\bf A})$ as in (\ref{MuP}). By definition and Lemma \ref{MuPProb}, $\mu_P$ is a positive probability function. Thus, it is left to prove that, for every conditional $(a\mid b)$, $\mu_P(a\mid b)=P(a\wedge b)/P(b)$.

To this end, recall that for each basic conditional we have $(a\mid b) =(\bigvee_{\alpha_i\leq a} \alpha_i\mid b)=\bigsqcup_{\alpha_i\leq a}(\alpha_i\mid b)$. Thus, since $\mu_P$ is additive, it is sufficient to prove the claim for those conditionals of the form $(\alpha\mid b)$ where $\alpha\in \atom({\bf A})$. Without loss of generality we will assume $\alpha=\alpha_1$. 

Notice first that, if $\alpha_1\not\leq b$, then $(\alpha_1\mid b)=\bot_{\mathfrak{C}}$, whence the claim is trivial. Thus, we shall henceforth assume that $\alpha_1\leq b$.
Now, from Lemma \ref{lemma:property1Sj}, $\{\mathbb{S}_j\}_{j=1,\ldots, n}$ is a partition of $\atom_{\leq}(\alpha_1\mid b)$. Thus,  
\begin{equation}\label{eq:ThmMain}
\mu_P(\alpha_1\mid b)=\sum_{j=1}^n\mu_P(\mathbb{S}_j)=\mu_P(\mathbb{S}_1)+\sum_{j=2}^n\mu_P(\mathbb{S}_j).
\end{equation} 
From Lemma \ref{lemma:property1Sj} (ii), $\mu_P(\mathbb{S}_j)=0$ for all $j$ such that $\alpha_j\leq b$,  therefore,
$$
\sum_{j=2}^n \mu_P(\mathbb{S}_j)=\sum_{j: \alpha_j\leq \neg b}\mu_P(\mathbb{S}_j). 
$$
By Lemma \ref{lemma:cases} (ii) above, we also have: 
$$
\sum_{j: \alpha_j\leq \neg b}\mu_P(\mathbb{S}_j)=\sum_{j: \alpha_j\leq \neg b}\frac{P(\alpha_1)\cdot P(\alpha_j)}{P(b)}=\frac{P(\alpha_1)}{P(b)}\cdot \sum_{j: \alpha_j\leq \neg b}P(\alpha_j)=\frac{P(\alpha_1)}{P(b)}\cdot P(\neg b).
$$
Finally, using this and Lemma  \ref{lemma:cases} (i), from (\ref{eq:ThmMain})  we get: 

$$
\begin{array}{lll}
\mu_P(\alpha_1\mid b)&=&P(\alpha_1)+\frac{P(\alpha_1)}{P(b)}\cdot (1-P(b))\\
&=&P(\alpha_1)+\frac{P(\alpha_1)}{P(b)}-P(\alpha_1)\\
& = &P(\alpha_1\mid b).
 \end{array}
$$
\end{proof}

An immediate consequence of the above theorem is that, for every positive probability $P$ on ${\bf A}$ and for every  $\overline{\alpha}=\langle \alpha_1,\ldots, \alpha_{n-1}\rangle\in Seq({\bf A})$, one has
\begin{equation}\label{eqindependence}
\begin{array}{lll}
\mu_P(\omega_{\overline{\alpha}})&=&\mu_P((\alpha_1\mid \top)\sqcap(\alpha_2\mid \neg \alpha_1)\sqcap\ldots\sqcap(\alpha_{n-1}\mid\neg\alpha_1\wedge\neg \alpha_2\wedge\ldots\wedge \neg\alpha_{n-2}))\\
&=&P(\alpha_1\mid \top)\cdot P(\alpha_2\mid \neg \alpha_1)\cdot\ldots\cdot P(\alpha_{n-1}\mid\neg\alpha_1\wedge\neg \alpha_2\wedge\ldots\wedge \neg\alpha_{n-2})\\
&=& \mu_P(\alpha_1\mid \top)\cdot\mu_P(\alpha_2\mid \neg \alpha_1)\cdot\ldots\cdot\mu_P(\alpha_{n-1}\mid\neg\alpha_1\wedge\neg \alpha_2\wedge\ldots\wedge \neg\alpha_{n-2}).
\end{array}
\end{equation}
That is, the basic conditionals which conjunctively define an atom of $\mathcal{C}({\bf A})$ are jointly independent with respect to the probability $\mu_P$.

Furthermore, the following result provides a characterization for the positive and separable probabilities on $\mathcal{C}({\bf A})$.

\begin{corollary}\label{cor:main1}
For every finite Boolean algebra ${\bf A}$ and for every positive probability $\mu$ on $\mathcal{C}({\bf A})$, the following are equivalent:
\begin{enumerate}[(i)]
\item $\mu$ is separable;
\item there is a positive probability $P_\mu$ on ${\bf A}$ such that 
for every basic conditional $(a\mid b)$, 
$$
\mu(a\mid b)=\frac{P_\mu(a\wedge b)}{P_\mu(b)}=\mu_{P_{\mu}}(a\mid b).
$$
\end{enumerate}
\end{corollary}

\begin{proof}
$(i)\Rightarrow(ii)$. 
Let $P_\mu$ be the restriction of $\mu$ to the Boolean subalgebra ${\bf A}\mid \top$ of $\mathcal{C}({\bf A})$. Thus, as ${\bf A}\mid \top$ is isomorphic to ${\bf A}$ by Corollary \ref{subalgebra},  $P_\mu$ is a positive probability of ${\bf A}$ (Remark \ref{rem:subProb1} plus the trivial observation that the positivity of $\mu$ induces the positivity of $P_\mu$). Furthermore, for every basic conditional $(a\mid b)$ of $\mathcal{C}({\bf A})$, from Theorem \ref{thm:main1} and the definition of $P_\mu$,
$$
\mu_{P_\mu}(a\mid b)=\frac{P_\mu(a\wedge b)}{P_\mu(b)}=\frac{\mu(a\wedge b\mid \top)}{\mu(b\mid \top)}. 
$$
In order to conclude the proof notice that, since $\mu$ is separable,  $\mu(a\mid b)\cdot \mu(b\mid \top)=\mu(a\wedge b\mid b)\cdot \mu(b\mid \top)=\mu(a\wedge b\mid \top)$. Therefore, $\mu_{P_\mu}(a\mid b)=\displaystyle{\frac{\mu(a\wedge b \mid \top)}{\mu(b \mid \top)}=\mu(a\mid b)}$.
\vspace{.2cm}

\noindent $(2)\Rightarrow(1)$. For each basic conditional $(a\mid b)\in \mathcal{C}({\bf A})$, $\mu(a\mid b)=\mu_P(a\mid b)$, and since $\mu_P$ is separable, for all $a\leq b\leq c\in A$, one has $\mu(a\mid c)=\mu_P(a\mid c)=\mu_P(a\mid b)\cdot\mu_P(b\mid c)=\mu(a\mid b)\cdot\mu(b\mid c)$ which settles the claim.
\end{proof}

Corollary \ref{cor:main1} implies that, if $\mu$ is separable, the values on basic conditionals $(a \mid b)$ only depend on the values on basic conditionals of the form $(\alpha \mid \top)$, with $\alpha$ an atom of $\bf A$. Indeed, if $\mu$  is separable, then 
$$\mu(a \mid b) = \frac{\mu(a \land b \mid \top)}{\mu(b \mid \top)}.$$
Therefore, any other measure $\mu'$ that coincides with $\mu$ on the basic conditionals of the form $(\alpha \mid \top)$, will coincide as well on each basic conditional. 

However, although positive separable measures on $\mathcal{C}({\bf A})$ are characterized through Corollary \ref{cor:main1} above as those measures which coincide with $\mu_P$ on basic conditionals, the measures of the form $\mu_P$ are not the unique positive and separable probability functions on $\mathcal{C}({\bf A})$. Indeed, let $\mu_P$ be the canonical extension to $\mathcal{C}({\bf A})$ of a positive probability $P:{\bf A}\to[0,1]$. Let $\atom({\bf A})={\alpha_1,\ldots, \alpha_n}$ (with $n\geq 3$) and fix two atoms $\omega_1,\omega_2\in \atom_\leq(\alpha_1\mid\top)$. Assume, without loss of generality, that $\mu_P(\omega_1)\leq \mu_P(\omega_2)$ and take any $0<\varepsilon<\frac{\min(\mu_P(\omega_1), 1-\mu_P(\omega_2))}{2}$ so that  $\mu_P(\omega_1)-\varepsilon>0$ and $\mu_P(\omega_2)+\varepsilon<1$. Now,  for every $\omega\in \atom(\mathcal{C}({\bf A}))$, define
$$
\mu'(\omega)=\left\{
\begin{array}{ll}
\mu_P(\omega)&\mbox{ if }\omega\neq \omega_1, \omega\neq \omega_2,\\
\mu_P(\omega_1)-\varepsilon&\mbox{ if }\omega=\omega_1,\\
\mu_P(\omega_2)+\varepsilon&\mbox{ if }\omega=\omega_2.
\end{array}
\right.
$$
Clearly, $\sum_{\omega\in \atom(\mathcal{C}({\bf A}))}\mu'(\omega)=1$ and $\mu'(\omega)>0$ for all $\omega\in \atom(\mathcal{C}({\bf A}))$. Thus, $\mu'$ extends to a positive probability, that we still denote by $\mu'$, on the whole $\mathcal{C}({\bf A})$. Further, by definition, $\mu'(\alpha_1\mid\top)=\mu_P(\alpha_1\mid \top)$, whence for all $\alpha_i\in \atom({\bf A})$, $\mu'(\alpha_i\mid \top)=\mu_P(\alpha_i\mid \top)$ and therefore, by the above considerations, 
\begin{equation}\label{eqSeparableCounter}
\mu'(a\mid b)=\mu_P(a\mid b)\mbox{ for every basic conditional }(a\mid b)\in \mathcal{C}({\bf A}). 
\end{equation}
Moreover, the above property and the separability of $\mu_P$ imply the separability of $\mu'$. Indeed, for every $a\leq b\leq c\in A'$,
$$
\mu'((a\mid b)\sqcap(b\mid c))=\mu'(a\mid c)=\mu_P(a\mid c)=\mu_P(a\mid b)\cdot\mu_P(b\mid c)=\mu'(a\mid b)\cdot\mu'(b\mid c). 
$$
However, $\mu'$ is not expressible as the canonical extension $\mu_{P'}$ for some positive $P'$ on ${\bf A}$. For otherwise, assuming that $\omega_1=(\alpha_1\mid \top)\sqcap(\alpha_2\mid \neg\alpha_1)\sqcap\ldots\sqcap(\alpha_{n-1}\mid \bigwedge_{i=1}^{n-1}\neg\alpha_i)$, from (\ref{eqindependence}) and (\ref{eqSeparableCounter}), one would have
$$
\mu'(\omega_1)=\mu'(\alpha_1\mid \top)\cdot\ldots\cdot\mu'(\alpha_{n-1}\mid \bigwedge_{i=1}^{n-1}\neg\alpha_i)=\mu_P(\alpha_1\mid \top)\cdot\ldots\cdot\mu_P(\alpha_{n-1}\mid \bigwedge_{i=1}^{n-1}\neg\alpha_i)=\mu_P(\omega_1)
$$
while $\mu'(\omega_1)=\mu_P(\omega_1)-\varepsilon$ and $\varepsilon>0$.

Summing up, besides the whole class of (simple) probabilities on a Boolean algebra of conditionals $\mathcal{C}({\bf A})$, we can identify three  relevant subclasses:
\begin{enumerate}
\item $\Sigma=\{\mu:\mathcal{C}({\bf A})\to[0,1] \mid \mu \mbox{ is a  positive probability function}\}$
\item $\Gamma=\{\mu:\mathcal{C}({\bf A})\to[0,1] \mid \mu \mbox{ is positive and separable}\}$
\item $\Pi=\{\mu:\mathcal{C}({\bf A})\to[0,1] \mid \mu=\mu_P \mbox{ for some positive probability } P \mbox{ on } {\bf A} \}$
\end{enumerate}
Obviously $\Pi\subseteq\Gamma\subseteq\Sigma$. Furthermore, 
for  $| \atom({\bf A})| \geq 3$   all inclusions are proper as witnessed by the above and Example \ref{ex:nonCond}.

\section{A logical reading of Boolean conditionals}\label{sec:logic}
In this section we undertake a natural step towards defining a logic to reason with
conditionals whose semantics is in accordance with the notion of
the Boolean algebra of conditionals as described above. This logic, that we call the Logic of Boolean Conditionals (LBC) is axiomatized in the following subsection, while its semantics is introduced in Subsection \ref{subsec:sem} where we also prove a completeness theorem. In Subsection \ref{sec:NMR}, we investigate the relation between LBC and non-monotonic reasoning.

\subsection{Syntax and axiomatics}
Let us start by considering the classical propositional logic language $\mathsf{L}$, built from a finite
set of propositional variables $p_1, p_2, \ldots p_m$. Let
  $\vdash_{CPL}$ denote derivability in  Classical Propositional Logic.
Based on $\mathsf{L}$, we define the language $\mathsf{CL}$ of
conditionals 
by the following stipulations: 

\begin{itemize}

\item[-] Basic (or atomic) conditional formulas, expressions of the form $(\varphi \mid \psi)$ where $\varphi, \psi \in \mathsf{ L}$ and such that $ \not\vdash_{CPL} \neg\psi$, are in $\mathsf{CL}$.

\item[-] 
Further, if $\Phi, \Psi  \in \mathsf{ CL}$, then $\neg \Phi, \Phi \land \Psi \in \mathsf{ CL}$.\footnote{We use the same symbols for connectives in $\mathsf{ L}$ and in $\mathsf{ CL}$ without danger of confusion.} Other connectives like $\lor$, $\to$ and $\leftrightarrow$ are defined as usual. 

\end{itemize}
Note that we do not allow the nesting of conditionals, as usually done
in the vast literature on the modal approaches to Conditional Logics
(see, for a general survey, \cite{Arlo-Costa}.) 

Actually, we could also allow purely propositional formulas from $\mathsf{L}$ to be part of $\mathsf{ CL}$ but, as a matter of fact, any proposition $\varphi$ will  be always identifiable with the conditional $(\varphi \mid \top)$. Next an axiomatic system for a logic of Boolean conditionals over the language $\mathsf{ CL}$ is presented. 

\begin{definition} The {\em Logic of Boolean conditionals} (LBC for
  short) has the following axioms\footnote{Actually, Axiom (A4) is not
    independent since it follows from (A1) and (A3), but we leave it
    for the sake of coherence with the algebraic part. In the same vein, we could replace the equivalence connective $\leftrightarrow$ by the implication connective $\to$ in Axioms (A2) and (A3), since the other direction can be deduced by classical reasoning and using axiom (A4) and the rule (R1) respectively.} and rules:

\begin{itemize}

\item[(CPL)] For any tautology of CPL, the formula resulting from a uniform replacement of the variables by basic conditionals.

\item[(A1)] $(\psi \mid \psi)$

\item[(A2)] $\neg (\varphi \mid \psi) \leftrightarrow  (\neg \varphi \mid \psi)$

\item[(A3)] $(\varphi \mid \psi) \land (\delta \mid \psi) \leftrightarrow (\varphi \land \delta \mid \psi)$

\item[(A4)] $(\varphi \mid \psi) \leftrightarrow (\varphi \land \psi \mid \psi)$

\item[(A5)] $(\varphi \mid \psi) \leftrightarrow (\varphi \mid \chi) \land (\chi \mid \psi)$,  if $\vdash_{CPL} \varphi \to \chi$ and  $\vdash_{CPL} \chi \to \psi$

\item[(R1)]  from $\vdash_{CPL} \varphi \to \psi$ derive $(\varphi \mid \chi) \to (\psi \mid \chi)$
\item[(R2)]  from $\vdash_{CPL} \chi \leftrightarrow \psi$ derive $(\varphi \mid \chi) \leftrightarrow (\varphi \mid \psi)$
\item[(MP)] Modus Ponens: from $\Phi$ and $\Phi \to \Psi$ derive $\Psi$
\end{itemize} 

\noindent The notion of proof 
  in LBC,
$\vdash_{LBC}$, is defined as usual from the above axioms and rules. 
\end{definition}

As we pointed out at the beginning of this section, in the language $\mathsf{ CL}$, basic conditional formulas $(\varphi \mid \psi)$ can be regarded, in logical terms, as new (complex) propositional variables where the compound conditional formulas are built upon using the classical CPL connectives $\land$ and $\neg$. Therefore, deductions in LBC employ axioms (A1)-(A5) and rules (R1), (R2) to reason about the internal structure of basic conditional formulas, and the axioms and rules of CPL (captured by the axiom schema (CPL) and the Modus Ponens rule) for compound conditional formulas. This means that, if we denote by $\mathsf{AX}$ the set of all instantiations of the axioms (A1)-(A5) once closed by rules (R1) and (R2), deductions in LBC can in fact be translated to deductions from $\mathsf{AX}$ (taken as a theory) only using the rules and axioms of CPL. This is formally expressed in the first item of following lemma. As a consequence of this translation, among other facts, it follows that the logic LBC satisfies the deduction theorem, also stated in the same lemma. 

\begin{lemma}\label{lemma:DedThm} For any set of $\mathsf{ CL}$-formulas $\Gamma \cup \{\Phi, \Psi\}$, the following properties hold, where $\mathsf{AX}$ is as above:
\begin{itemize}
\item $\Gamma \vdash_{LBC} \Phi$ iff $\Gamma \cup \mathsf{AX} \vdash_{CPL} \Phi$.
 \item Deduction theorem: $\Gamma \cup \{\Psi\} \vdash_{LBC} \Phi$ iff   $\Gamma\vdash_{LBC} \Psi \to \Phi$.
\end{itemize}
\end{lemma}

The above axiomatic system is clearly inspired by the key
properties of Boolean algebras of conditionals we discussed in Section \ref{sec:BAC}. 
Indeed, as expected, we can prove that the Lindenbaum algebra of provably equivalent formulas of LBC can be seen, in fact, as a Boolean algebra of conditionals. 
In slightly more detail, for any $\Phi, \Psi\in \mathsf{CL}$, let us write $\Phi\equiv\Psi$ if $\vdash_{LBC}\Phi\leftrightarrow\Psi$. Notice that rules (R1) and (R2) ensure that $\equiv$ preserves propositional logical equivalence in the sense that if $\vdash_{CPL} \varphi \leftrightarrow \varphi'$ and $\vdash_{CPL} \psi \leftrightarrow \psi'$, then $(\varphi \mid \psi) \equiv (\varphi' \mid \psi')$. Moreover, $\equiv$ is compatible, in the sense of Section \ref{sec:algpre}, with Boolean operations. Thus, the quotient set $\mathsf{CL}/_\equiv$ endowed with the following operations is a Boolean algebra:
\begin{itemize}
\item[] $\top_\equiv =[(\top\mid\top)]_\equiv$ and $\bot_\equiv=[(\bot\mid \top)]_{\equiv}$; 
\item[] $[\Phi]_\equiv\sqcap[\Psi]_\equiv=[\Phi\wedge\Psi]_{\equiv}$;
\item[] $[\Phi]_\equiv\sqcup[\Psi]_\equiv=[\Phi\vee\Psi]_{\equiv}$;
\item[] $\msim[\Phi]_\equiv=[\neg\Phi]_\equiv$.
\end{itemize}
We claim that  ${\bf CL}=(\mathsf{CL}/_\equiv, \sqcap, \sqcup, \sim, \top_\equiv, \bot_\equiv) $ is isomorphic to $\mathcal{C}({\bf L})$, the Boolean algebra of conditionals of the Lindenbaum algebra ${\bf L}$ of CPL over the language $\mathsf{L}$. Note that, by definition,  
$\vdash_{LBC} \Phi$ iff $[\Phi]_\equiv = \top_\equiv$.

\begin{theorem}\label{thm:LogIso}
${\bf CL} \cong\mathcal{C}({\bf L})$.
\end{theorem}

The proof is not difficult but the details are a bit tedious and can be found in the Appendix. However, it is convenient to point out that the isomorphism $\iota$ between ${\bf CL}$ and $\mathcal{C}({\bf L})$ acts on basic conditionals in the expected way: for each pair of formulas $\varphi, \psi$ such that $\not\vdash_{CPL}\neg \psi$, the element 
$[(\varphi\mid \psi)]_\equiv$ of ${\bf CL}$ is mapped by $\iota$ to $([\varphi]\mid [\psi])\in \mathcal{C}({\bf L})$. As for operations it is enough to ask that $\iota$ commutes with $\sqcap$ and $\sim$. Then, if for instance we have two basic conditionals of ${\bf CL}$ of the form $[(\varphi\mid \psi)]_\equiv$ and $[(\delta\mid \psi)]_\equiv$, by definition of $\equiv$, $\iota([(\varphi\mid \psi)]_\equiv\sqcap[(\delta\mid \psi)]_\equiv)=\iota([(\varphi\wedge\delta\mid\psi)]_\equiv)=([\varphi\wedge\psi]\mid [\delta])$ and, in turn by the property of the conditional algebra $\mathcal{C}({\bf L})$, $([\varphi\wedge\delta]\mid [\psi])=([\varphi]\mid[\psi])\wedge([\delta]\mid[\psi])=\iota([(\varphi\mid\psi)]_\equiv)\wedge\iota([(\delta\mid\psi)]_\equiv)$.

\subsection{Semantics and completeness}\label{subsec:sem}

The intuitive idea in defining a semantics for the logic LBC is that, as in classical logic (recall Proposition \ref{prop:atomsAlgebra} (iv)) the evaluations of $\mathsf{CL}$-formulas
should be in 1-1 correspondence with the atoms of the algebra $\mathcal{C}({\bf L})$, 
 which is, by Theorem \ref{thm:LogIso}, isomorphic to the Lindenbaum
 algebra  ${\bf CL}$ of LBC. As we have seen in
 Section \ref{sec:atoms}, atoms of  $\mathcal{C}({\bf L})$ can be
 described by sequences of atoms of $\bf L$.

In the following let $\Omega$ be the set of (classical) evaluations  $v: \mathsf{L} \to \{0, 1\}$ for the propositional language $\mathsf{L}$. 
Recall from Section \ref{sec:algpre} that if $\mathsf{L}$ is built from $m$ propositional variables, then $| \Omega | = 2^m$. Moreover, we can identify every evaluation $v \in \Omega$ with its corresponding  minterm and hence with an atom of Lindenbaum algebra $\bf L$. 
Therefore, it follows  from Section \ref{sec:atoms}  that  the atoms of $\mathcal{C}({\bf L})$ are of the form  
$$
(v_1\mid \top)\sqcap(v_2\mid \neg v_1)\sqcap\ldots\sqcap(v_{2^m-1}\mid \neg v_1\wedge\ldots\wedge\neg v_{2^m-2}),
$$
where $\langle v_1,\ldots, v_{2^m-1}\rangle\in Seq({\bf L})$.
The idea is then to define {\em $\mathsf{CL}$-interpretations} as sequences  $e = \langle v_1,\ldots, v_{2^m}\rangle$ of pair-wise distinct  $2^m$ $\mathsf{L}$-interpretations of $\Omega$ and to stipulate that such a $\mathsf{CL}$-interpretation $e$ makes true a conditional $(\varphi \mid \psi)$ when the atom in $\mathcal{C}({\bf L})$ determined by $e$ 
\begin{equation}\label{eq:atomsCL}
\omega_{e}=(v_1 \mid \top) \sqcap  (v_2 \mid \neg v_1) \sqcap \ldots \sqcap
 (v_{2^m-1} \mid \neg v_1 \land \ldots \land \neg v_{2^m-2}),
 \end{equation}
 is below $(\varphi \mid \psi)$, the latter thought of as an element of $\mathcal{C}({\bf L})$ as well. This condition has an easier expression according to (iii) of Proposition \ref{prop:atomsbelow}, that leads to the following natural definition. In the following we will assume $|\Omega|=m$ and we fix $n=2^m$ to simplify the reading. 

\begin{definition}
A  {\em $\mathsf{CL}$-interpretation} is a sequence $e = \langle v_1, v_2, \ldots, v_{n}\rangle$ of $n$ pairwise distinct $v_1, \ldots, v_{n} \in \Omega$. We denote also by $e$ the corresponding {\em evaluation} of $\mathsf{CL}$-formulas, i.e. the mapping $e: \mathsf{CL} \to \{0, 1\}$ defined as follows: 

\begin{itemize}

\item[-] for basic $\mathsf{CL}$-formulas: $e(\varphi \mid \psi) = 1$ if $v_i \models \varphi$ for the lowest index $i$ such that $v_i \models \psi$,  and $e(\varphi \mid \psi) = 0$ otherwise.

\item[-]  for compound $\mathsf{CL}$-formulas: $e$ is extended using Boolean truth-functions. 
\end{itemize}

\noindent The corresponding notion of consequence is as expected: for any set of $\mathsf{CL}$-formulas $\Gamma \cup \{\Phi\}$, $\Gamma \models_{LBC} \Phi$ if for every $\mathsf{CL}$-interpretation $e$ such that $e(\Psi)=1$  for all $\Psi \in \Gamma$, it also holds $e(\Phi) = 1$. 
\end{definition}

By Proposition \ref{prop:atomsAlgebra}, given any homomorphism $h:\mathcal{C}({\bf L})\to\{0,1\}$, let $\omega_h$ be its corresponding atom, i.e.,  $\lambda(\omega_h)=h$ in the terminology of Proposition \ref{prop:atomsAlgebra} (iii). We further denote by $\Lambda(h)$ the sequence $\langle v_1,\ldots, v_{n-1}, v_n\rangle$ such that $\langle v_1,\ldots, v_{n-1}\rangle$  univocally determines $\omega_h$ via (\ref{eq:atomsCL}) and $v_n$ is the only evaluation left in $\Omega\setminus\{v_1,\ldots, v_{n-1}\}$.

\begin{lemma}\label{lemma:homoCL}
For every homomorphism $h:\mathcal{C}({\bf L})\to\{0,1\}$, 
$\Lambda(h)$ is a $\mathsf{CL}$-interpretation. Further,  for each $\mathsf{CL}$-formula $\Phi$, $h(\Phi)=1$ iff $\Lambda(h)(\Phi) = 1$.
\end{lemma}

\begin{proof} 
If $h:\mathcal{C}({\bf L})\to\{0,1\}$ is a homomorphism, $\Lambda(h)$ is a $\mathsf{CL}$-interpretation by definition. Now we prove, by structural induction on the complexity of $\Phi$, that $\Lambda(h)(\Phi) = 1$ iff $h(\Phi)=1$.  The interesting case is when $\Phi$ is a basic conditional $(\varphi\mid \psi)$ such that $(\varphi\mid \psi) \not\equiv \top$. 
In that case, $h(\Phi)=1$ iff $\omega_h\leq (\varphi\mid\psi)$ iff  (by Proposition \ref{prop:atomsbelow}) there exists an index  $i\leq n-1$ such that $v_i(\varphi)=v_i(\psi)=1$ and $v_j(\psi)=0$ for all $j<i$, and thus, iff  $\Lambda(h)(\Phi) = 1$.
\end{proof}

Now the soundness and completeness of LBC easily follow from the above.  

\begin{theorem}[Soundness and completeness]\label{ThmCompletness} LBC  is sound and complete w.r.t. $\mathsf{CL}$-evaluations, i.e. $ \vdash_{LBC} \; =  \; \models_{LBC} $. 
\end{theorem}

\begin{proof}  Soundness is easy. As for completeness, assume that $\Gamma \not\vdash_{LBC} \Phi$. Thus, in particular $\Phi$ is false in $\mathcal{C}({\bf L})$, meaning that there exists a homomorphism $h:\mathcal{C}({\bf L})\to\{0,1\}$ such that $h(\gamma)=1$ for all $\gamma\in \Gamma$, and $h(\Phi)=0$. Thus, by Lemma \ref{lemma:homoCL}, $\Lambda(h)$ is a $\mathsf{CL}$-interpretation such that $\Lambda(h)(\gamma) = 1$ for every $\gamma \in \Gamma$ and $\Lambda(h)(\Phi) = 0$, i.e. $\Gamma \not\models_{LBC} \Phi$ . 
 \end{proof}

 \begin{remark} One could also define an alternative semantics for LBC in terms of a class of Kripke models with order relations, more in the style of conditional logics (see \cite{Halpern2003}). This Kripke-style semantics can also be shown to be adequate for LBC (and hence equivalent to the previous more algebraic semantics), although it is not fully exploited as the language of LBC does not contain nested applications of the conditioning operator nor pure propositional formulas.  Indeed, one can define  {\em LBC-Kripke models}  as structures $M = (W, \{ \succ_w\}_{w \in W}, e)$, where (i) $W$ is a non-empty set of worlds, (ii) for each world $w \in W$, $\succ_w$ is a (dual) well-order\footnote{A dual well-order in $W$ is a total order $\succ$ such that every subset $B \subseteq W$ has a greatest element.} 
such that $w \succ_w w'$ for every $ w \neq w' \in W$, and  (iii) $e: W\times V \to \{0, 1\}$ is a valuation function for propositional variables, that naturally extends to any Boolean combination of variables. 
For every $w \in W$ and $B \subseteq W$, let $best_w(B)$ denote the
greatest element  in $B$ according to $\succ_w$. Moreover, for any propositional (non-conditional) $\varphi$, we will use $\llb \varphi \rrb$ to denote the set of  worlds in $M$ evaluating $\varphi$ to true, i.e.\ $\llb \varphi \rrb = \{w \in W \mid  e(w, \varphi) = 1\}$. 

Truth in a pointed model is defined by induction as follows. For each $w \in W$ we define: %$(M, w) \models \Phi$ : 

\begin{itemize}
\item  $(M, w)  \models_K (\varphi \mid \psi)$ if either $\llb \psi \rrb = \emptyset$ or 
$best_w(\llb \psi \rrb) \in \llb \varphi \rrb$. Equivalently,  in case $\llb \psi \rrb \neq \emptyset$, then $(M, w) \models_K (\varphi \mid \psi)$ if  $(M, best_w(\llb \psi \rrb)) \models_K \varphi$.
\item Satisfaction for Boolean combinations of conditionals is defined in the usual way. 
\end{itemize}
It is not difficult to see that LBC is also sound w.r.t. to the class of LBC-Kripke models. As for completeness, assume as usual that  $\Gamma \not\vdash_{LBC} \Phi$, with $\Gamma$ a finite set, and let us prove that  $\Gamma \not\models_{K} \Phi$. From Lemma \ref{lemma:DedThm}, the condition $\Gamma \not\vdash_{LBC} \Phi$ is equivalent to 
$\not\vdash_{LBC} \Phi'$, where $\Phi' = \bigwedge_{\Psi \in \Gamma}\Psi \to \Phi$. This means that $[\Phi']_\equiv \neq [\top]_\equiv$ in the (finite) Lindenbaum algebra ${\bf CL}$. By Theorem \ref{thm:LogIso}, for practical purposes we can identify ${\bf CL}$ with the Boolean algebra of conditionals $\mathcal{C}({\bf L})$. Therefore, %that means in turn 
there is an atom $\omega \in \mathcal{C}({\bf L})$ such that $\omega \not \leq [\Phi']_\equiv$. As we have seen in Section \ref{sec:atoms}, atoms of $ \mathcal{C}({\bf L})$ can be described by sequences of pair-wise different atoms $\omega_i$ of ${\bf L}$. But each atom $\omega_i$ of ${\bf L}$ can in turn be identified with (the minterm associated to) a Boolean evaluation $v_i$ of the propositional variables $V$. Therefore if we let $\Omega$
 be the set of evaluations of the propositional variables $V$, we can assume
$\omega =  \langle v_1, \ldots, v_{n-1} \rangle$, with $n = 2^{|V|}$ and where $v_1, \ldots, v_{n-1} \in \Omega$.  Finally, we consider an LBC-Kripke model $M^* = (\Omega, \{ \succ^*_v\}_{v \in \Omega}, e^*)$ where: 
\begin{itemize}
\item $\succ^*_{v_1} \subset \Omega \times \Omega$ is such that  $v_i \succ^*_{v_1} v_{i+1}$ for $i=1, \ldots, n-1$\footnote{Where $v_n = \Omega \setminus \{v_1, \ldots, v_{n-1}\}$.} and for all $i\neq 1$, $\succ^*_{v_i}$ is an arbitrary (dual) well-order on $\Omega$; 
\item $e^*: \Omega \times V \to \{0, 1\}$ is such that $e^*(v, p) = v(p)$ for every $p \in V$. 
\end{itemize}
Then, it is clear  that the condition $\omega \not \leq [\Phi']_\equiv$ guarantees that $(M^*, v_1) \not\models_{K} \Phi'$.  \qed
 \end{remark}
An immediate consequence of Theorem \ref{ThmCompletness} and the fact that there are only finitely many $\mathsf{CL}$-evaluations is the decidability of the calculus LBC.
Furthermore, a basic conditional $(\varphi\mid \psi)$ is satisfiable in LBC iff there exists a  {\em classical} valuation $v$ such that $v(\varphi\wedge \psi)=1$. Indeed, if  a {\em $\mathsf{CL}$-interpretation} $e = \langle v_1, v_2, \ldots, v_{n}\rangle$  is such that $e(\varphi\mid \psi) = 1$, there is $i$ such that $v_i(\varphi\wedge \psi)=1$. Conversely, if $v(\varphi\wedge \psi)=1$, then any  {\em $\mathsf{CL}$-interpretation} $e = \langle v, v_2, \ldots, v_{n}\rangle$, by definition, is a model of the conditional $(\varphi\mid \psi)$. 
Thus,  the satisfiability of a basic conditional $(\varphi\mid\psi)$ reduces to the classical satisfiability.  
 However, this direct reduction does not generally apply to the cases in which the conditional formula is a (non-trivial) Boolean combination of basic conditionals.  
 
 Moreover, since  basic conditionals of the form $(\varphi\mid \top)$ are equivalent to  $\varphi$, it is also  easy to see that the classical satisfiability  is a subproblem of the satisfiability  of LBC which shows the latter to be NP-hard.

Determining a possible NP-containment for the LBC-satisfiability is out of the scope of the present paper and it will be addressed in our future work.

\subsection{Relation to non-monotonic reasoning models}\label{sec:NMR}

Conditionals possess  an implicit non-monotonic behaviour. Given a
conditional $(\varphi \mid \psi)$, it does not follow in general that
we can freely strengthen its antecedent, i.e. in general, $(\varphi
\mid \psi) \not \vdash_{LBC} (\varphi \mid \psi\land \chi)$. For
instance, $\varphi, \psi,\chi$ can be such that $\varphi \land \psi
\not \models \bot$ while $\varphi \land \psi \land \chi \models
\bot$. Actually, and not very surprisingly, the logic $\vdash_{LBC}$
satisfies the analogues of the KLM-properties which characterize the well-known system P of preferential entailment \cite{LM92,Makinson2005}.

\begin{lemma}\label{Lemma:Logic}
$\vdash_{LBC}$ satisfies the following properties: 
\begin{itemize}

\item[] Reflexivity: $\vdash_{LBC} (\varphi \mid \varphi)$ 

\item[]  Left logical equivalence: if $\models_{CPL} \varphi \leftrightarrow \psi$ then $(\chi  \mid \varphi) \vdash_{LBC} (\chi  \mid \psi)$

\item[]  Right weakening: if $\models_{CPL} \varphi \to \psi$ then $(\varphi \mid \chi) \vdash_{LBC} (\psi  \mid \chi)$

\item[] Cut: $(\varphi \mid \psi)  \land (\chi \mid \varphi \land \psi) \vdash_{LBC} (\chi \mid \psi)$

\item[] OR: $(\varphi \mid \psi) \land (\varphi \mid \chi) \vdash_{LBC} (\varphi \mid \psi \lor \chi)$

\item[] AND: $(\varphi \mid \psi) \land (\delta \mid \psi) \vdash_{LBC} (\varphi \land \delta \mid \psi)$

\item[] Cautious Monotony: $(\varphi \mid \psi) \land (\chi \mid \psi) \vdash_{LBC} (\chi \mid \varphi \land \psi )$

\end{itemize}
\end{lemma}   

\begin{proof}  
 {\em Reflexivity}, \emph{Left Logical Equivalence}, 
  \emph{Right Weakening} and \emph{AND} correspond to (A1), (R2),
  (R1), and (A3) of LBC, respectively. The other cases are
  proved as follows. 
  
  \vspace{.1cm}

\noindent {\em Cut}:  
by (A4), $(\chi \mid \varphi \land \psi) \land (\varphi \mid \psi)$ is equivalent to $(\chi \land  \varphi \land \psi \mid \varphi \land \psi) \land (\varphi \land \psi \mid \psi)$, and by (A5), it is equivalent to $(\chi \land  \varphi \land \psi \mid \psi) $, and by (R1) this clearly implies $(\chi  \mid \psi)$. 

\vspace{.1cm}

\noindent {\em Cautious Monotony}: 
by (A3), $(\varphi \mid \psi) \land (\chi \mid \psi)$ is equivalent to $(\varphi \land \chi \mid \psi)$, which by (A4) is in turn equivalent   $(\varphi \land \chi \land \psi \mid \psi)$, and by (A5) implies  $(\varphi \land \chi \land \psi \mid  \varphi \land \psi)$, which by (A3) is equivalent to $( \chi  \mid  \varphi \land \psi)$. 

\vspace{.1cm}

\noindent {\em OR}: 
 $(\varphi \mid \psi) \land (\varphi \mid \chi)$ is equivalent to $[(\varphi \mid \psi) \land (\varphi \mid \chi) \land (\psi \mid \psi \lor \chi)] \lor [(\varphi \mid \psi) \land (\varphi \mid \chi) \land (\chi \mid \psi \lor \chi)]$, and this implies $[(\varphi \land \psi \mid \psi) \land (\psi\mid \psi \lor \chi)] \lor   [ (\varphi \land \chi \mid \chi) \land (\chi \mid \psi \lor \chi)]$, that is equivalent to 
 $(\varphi \land \psi \mid \psi \lor \chi) \lor (\varphi \land \chi \mid \psi \lor \chi)$, which finally implies $(\varphi \mid  \psi \lor \chi)$.  
\end{proof}

Now, let us fix a set of (atomic) conditional statements $K$, and let
us define the consequence relation associated to $K$:  $\varphi\;\!
\nmdash_K \!\;\psi$ if $K \vdash_{LBC} (\psi \mid \varphi)$. 
Our last claim is easily derived from the previous lemma.

\begin{theorem}\label{prop:preferential}  $\nmdash_K$ is a preferential consequence relation. 
\end{theorem}

However, the following also well-known rule %of Rational Monotonicity, 
\begin{itemize}
 \item[] {\em Rational Monotonicity}:  if $\psi \; \nmdash\; \varphi$ and $\psi  \;|\!\!\!\not\sim \neg \chi$ then $ \psi \land \chi \;\nmdash\; \varphi$
 \end{itemize}
  does not hold in general for  $\nmdash = \nmdash_K$
  %in LBC 
  as the following example shows, and thus $\nmdash_K$ is not a rational consequence relation in the sense of Lehmann and Magidor \cite{LM92}. 
  
  \begin{example}
  Fix a propositional language $\mathsf{L}$ with 2 propositional variables, say $p, q$, so that it has four minterms ($\alpha_1 = p \land q, \alpha_2 = p \land \neg q,  \alpha_3 =\neg p \land q,  \alpha_4 = \neg p \land \neg q$), that correspond to the 4 atoms of its Lindenbaum algebra ${\bf L}$. 
The Lindenbaum algebra ${\bf CL}$ of the language of conditionals  $\mathsf{CL}$ is isomorphic to the algebra $\mathcal{C}({\bf L})$ with 24 atoms and $2^{24}$ elements. 
  
Consider  the following propositions: 
$\varphi = \alpha_1 \lor \alpha_4$, $\psi = \top$, $\chi = \alpha_1 \lor \alpha_3$ and the following set of only one conditional statement $K = \{(\varphi \mid \psi)\} = \{(\alpha_1 \lor \alpha_4 \mid \top)\}$. Note that  $(\neg \chi \mid \psi) = (\neg (\alpha_1 \lor \alpha_3) \mid \top) \equiv (\alpha_2 \lor \alpha_4 \mid \top)$ and $(\varphi \mid \chi \land \psi ) =  (\alpha_1 \lor \alpha_4 \mid \alpha_1 \lor \alpha_3)  \equiv (\alpha_1\mid \alpha_1 \lor \alpha_3)$. Then we have:
 \begin{enumerate}
 \item[(1)] $K \vdash (\varphi \mid \psi)$
 \item[(2)]  $K \not \vdash_{LBC} (\neg \chi \mid \psi)$, i.e.\ $(\alpha_1 \lor \alpha_4 \mid \top)  \not \vdash_{LBC} (\alpha_2 \lor \alpha_4 \mid \top)$, but 
 \item[(3)] $K \not \vdash_{LBC} (\varphi \mid \chi \land \psi )$, i.e.\ $(\alpha_1 \lor \alpha_4 \mid \top)  \not \vdash_{LBC}  (\alpha_1\mid \alpha_1 \lor \alpha_3)$ \end{enumerate}
  
(1) and (2) are clear. Hence, let us prove (3). Via Theorem \ref{ThmCompletness}, it is enough to show that there is an atom from $\mathcal{C}({\bf L})$ below $(\alpha_1 \lor \alpha_4 \mid \top)$ but not below $(\alpha_1\mid \alpha_1 \lor \alpha_3)$ (recall to this end Proposition \ref{prop:atomsbelow}). Consider the atom from $\mathcal{C}({\bf L})$ $\omega_{\overline{\beta}}$, with $\overline{\beta} = \langle\alpha_4, \alpha_2, \alpha_3\rangle$, that is, $\omega_{\overline{\beta}} = (\alpha_4 \mid \top) \sqcap (\alpha_2 \mid \neg \alpha_4) \sqcap (\alpha_3 \mid \neg \alpha_4 \land \neg \alpha_2 )$. Then: 
 
 - Clearly  $\omega_{\overline{\beta}} \leq (\alpha_1 \lor \alpha_4 \mid \top)$. 
 
 - However,  $\omega_{\overline{\beta}} \not\leq (\alpha_1\mid \alpha_1 \lor \alpha_3)$. Indeed, the least index $i$ for which $\beta_i \leq \alpha_1 \lor \alpha_3$ is $i=3$: $\beta_3 = \alpha_3 \leq \alpha_1 \lor \alpha_3$, but $\alpha_3 \not \leq \alpha_1$. Then $\mathcal{C}({\bf L})\not\models (\alpha_1 \lor \alpha_4 \mid \top)  \to  (\alpha_1\mid \alpha_1 \lor \alpha_3)$ and hence, by Theorem \ref{thm:LogIso} and Lemma \ref{lemma:DedThm},  $(\alpha_1 \lor \alpha_4 \mid \top)  \not \vdash_{LBC}  (\alpha_1\mid \alpha_1 \lor \alpha_3)$. \qed
 \end{example}

Let us notice that Rational Monotonicity does hold whenever we further require the relation $\nmdash_K$ to be defined by a {\em complete} theory $K$. Indeed, if $K$ is complete,  for each $\varphi$ and $\psi$, one has $\varphi\;\nmdash_K\;  \psi$ iff $\varphi \;|\!\!\!\not\sim_K \neg\psi$. Then, if we replace $\varphi \;|\!\!\!\not\sim_K \neg\psi$ by $\varphi\;\nmdash_K\;  \psi$ in  the previous expression of the Rational Monotony property, what we get is:
\begin{itemize}
\item[]  If $\psi \; \nmdash_K \; \varphi$ and $\psi  \;\nmdash_K\; \chi$ then $ \psi \land \chi \;\nmdash_K\; \varphi$,
\end{itemize}
that is, we are back to the Cautious Monotony property  (see \cite[\S3.4]{LM92}).

We finally notice that axiom (A2) of LBC implies 
$$  \vdash_{LBC} (\varphi\mid \psi) \vee (\neg\varphi\mid \psi).$$
In general, this does not imply that $\nmdash_K$ satisfies the so-called
{\em Conditional Excluded Middle} property \cite{Bezz}, namely
$$  \text{ either }  \psi\;\nmdash_K\;\varphi \text{ or }  \psi\;\nmdash_K\;\neg\varphi.$$
However if $K$ is complete,
then $\psi\; \nmdash_K\;\neg \varphi$ is logically equivalent to
$\psi \not\nmdash_K \;\varphi$, yielding the property of Conditional
Excluded Middle.

\section{Related work on conditional algebras}\label{sec:related}

In this section we discuss related algebraic approaches to conditional
events, which we group in (three-valued) Measure-free conditionals and Conditional
Event Algebras. The former lead necessarily to non-Boolean structures
whereas the latter to Boolean structures. Our
approach combines elements of both. Boolean Algebras of Conditionals
have a similar language as Measure-free conditionals and share the underlying Boolean structure with
Conditional Event Algebras.

\subsection{Measure-free conditionals}\label{sec:measure-free}

One of the most relevant
 approaches to study conditionals independently from conditional
 probability and outside modal logic has been to consider measure-free
 conditionals as three-valued objects: given  two propositions $a, b$
 of a classical propositional language $L$, a conditional $(a \mid b)$
 is true when $a$ and $b$ are true, is false when $a$ is false and $b$
 is true, and {\em inapplicable} when $b$ is false.  This three-valued
 approach actually goes back to de Finetti in the Thirties
 \cite{DeFinetti1937} and Schay in the Sixties \cite{Sh68}, and was
 later further developed among others by Calabrese \cite{Cal},
 Goodman, Nguyen and Walker \cite{GN94,GoodNgu,GNW,Wal94}, Gilio and Sanfilippo
 \cite{GS2}, and Dubois and Prade
 \cite{DP91,DP94}. In particular the last authors have also formally related measure-free conditionals and non-monotonic reasoning. The rest of this section is mainly from \cite{DP94}.

The idea is to extend any Boolean interpretation $v : L  \to \{0, 1\}$ for the language $L$ to a three-valued interpretation $v: (L \mid  L)  \to \{0, 1, ?\}$ on the set of conditionals $(L \mid  L)$ built from $L$ and defined as follows:
$$v (a \mid b) = \left \{
\begin{array}{ll}
v(a), & \mbox{if } v(b) = 1 \\
?, & \mbox{otherwise.}
\end{array}
\right .
$$
Although the third truth-value ``$?$'' denotes {inapplicable}, it is usually be understood as ``any truth-value'', i.e. one can take $?$ as the set $\{0, 1\}$. Indeed, this interpretation is compatible with taking $v(a \mid b)$  as the set of solutions $x \in \{0,1\}$ of the equation
$$v(a \land b) =x \land^* v(b),$$
where $\land^*$ is the Boolean truth-function for conjunction, 
once values for $v(a)$ and $v(b)$ have been fixed, and hence $v(a \land b)$ as well. In particular, when $b = \top$, there is a unique solution $v(a \mid \top) = \{v(a)\}$, for every $v$ and $a$.

From an algebraic point of view, considering plain events belonging to a Boolean algebra $\bf A$ (f.i.\ the Lindenbaum algebra for the above propositional language $L$), the above three-valued semantics led Goodman and Nguyen's to the following definition of a conditional event \cite{GoodNgu}, not as another event, but as a set of events in $\bf A$: 
$$ (a \mid b) = \{x \in A \mid x \land b = a \land b \}, $$
that turns out to actually be an interval in the algebra $\bf A$, indeed, it can be checked that 
$$ (a \mid b) =  \{x \in A \mid a \land b \leq x \leq b \to  a \} = [a \land b,  b \to a] .$$
It is worth noticing that conditionals of the form $(a \mid \top)$, where $\top$ is the top element of $\bf A$, can be safely identified with the plain event $a \in A$, since $[a \land \top,  \top \to a] = \{a\}$. Moreover, this definition leads, in turn, to usually accepted equalities among conditionals like  
$$
(a \mid b) = (a \land b \mid b) =  (a \leftrightarrow b \mid b) = (b \to a \mid b), 
$$ 
since all of them define the same interval. By Proposition \ref{prop2} and Corollary \ref{cor:props}, the above equalities  hold in every Boolean algebras of conditionals. However, observe that, contrarily to Boolean conditionals, if $a < \top$, then $(a \mid a) \neq (\top \mid \top)$.

From the above definition, it directly follows that the condition for the equality between two conditionals is the same as we determined in Theorem \ref{equality} (except for the case where both are $\top_{\mathfrak{C}}$ or $\bot_{\mathfrak{C}}$): 
$$(a \mid  b) = (c \mid d) \;\mbox{   iff   }\;  a\land b = c \land d \; \mbox{ and } \; b = d, $$    
while the commonly adopted and compatible notion of ordering is the one defined by the interval order: 
\begin{equation}\label{eqOrderFinal}
(a \mid b) \leq (c \mid d)  \;\mbox{   iff   }\;   a \land b \leq c \land d  \; \mbox{ and  } \; b \to a \leq d \to c, 
\end{equation}
 that is of course a partial order, with  ${\sf T} = (\top \mid \top)$ and ${\sf F} = (\bot \mid \top)$ being the top and bottom conditionals respectively. It is interesting to notice that, if $\bot<a<b$ and $\bot<c<d$, the left-to-right direction of (\ref{eqOrderFinal}) holds in every $\mathcal{C}({\bf A})$. Indeed, under those further conditions, $(a\mid b)\leq (c\mid d)$ gives, by Lemma \ref{neq22} (ii), that $a=a\wedge b\leq c=c\wedge d$. Moreover, $a>\bot$ implies $\bot<b\to a$, and $b\to a\leq a<b$. Similarly, $\bot< d\to c< d$. Therefore, since $(a\mid b)=(b\to a\mid b)$ and $(c\mid d)=(d\to c\mid d)$, one has $b\to a\leq d\to c$ as well.

It is clear  that, in the above setting of three-valued conditionals, any conditional can always be equated to another one of the form $(a \mid b)$ where $a \leq b$. Thus the set of conditionals built from events from a Boolean algebra $\bf A$ can be identified with the set $A \mid A = \{(a \mid b) : a, b \in A, a \leq b \}$, that is in 1-1 correspondence with the set $Int({\bf A})$ of intervals in $\bf A$. 
Therefore, in this setting, conditioning can be viewed in fact as an external operation ${\bf A} \times {\bf A} \to Int({\bf A})$.

Several attempts have been made to define reasonable operations on measure-free conditionals as internal operations on $Int({\bf A})$, in particular operations of ``conjunction'' $\wedge$, ``disjunction'' $\lor$ and ``negation'' $\neg$.
There is a widespread consensus on the definition of the negation $\neg$ by stipulating $\neg(a\mid b)=(\neg a\mid b)$, which coincides with the one on Boolean conditionals. 
As to conjunction, there have been some reasonable proposals in the literature, corresponding to different possibilities of defining the truth-table of a three-valued conjunction on  $\{0, 1, ?\}$ (see \cite{DP89,DP94}), but two of them emerge as the major competing proposals, as they are the only ones satisfying the following qualitative counterpart of Bayes rule: $(a \mid b) \mbox{ AND } (b \mid \top) = (a \land b \mid \top)$.\footnote{Note that this a particular case of our Requirement R4 for Boolean algebras of conditionals, see (v) of Proposition \ref{prop2} and (ii) of Corollary \ref{cor:props}.}
\begin{description}
\item[Goodman and Nguyen] (cf. \cite{GoodNgu}): $(a\mid b)\wedge_I (c\mid d)= 
( a\wedge c \mid (\neg a\wedge b)\vee(\neg c\wedge d)\vee(b\land d))$ 
\item[Schay, Calabrese] (cf. \cite{Sh68,Cal}):  $(a\mid b) \wedge_Q\; (c\mid d)=((b\to a)\wedge (d\to c)\mid b\vee d)$
\end{description}

The operation $\land_I$ is in fact a genuine {\em interval-based conjunction}, in the sense that the interval interpreting $(a\mid b)\wedge_I (c\mid d)$ is the result of computing all the conjunctions of the elements of the intervals interpreting the conditionals $(a\mid b)$ and $(c\mid d)$, that is, if $ (a \mid b) \land_I (c \mid d) = (u \mid v)$ then $[u \land v, v \to u] = \{ x \land y  \mid x \in [a\land b, b \to a], y \in [c\land d, d \to c] \}$. 

On the other hand, the operation $\wedge_Q$, also known as {\em quasi-conjunction} \cite{Adams1975}, is tightly related to the interval order on conditionals defined above, in the sense that 
$(a \mid  b) \leq (c \mid d)$ iff $(a \mid b) \land_Q (c \mid d) = (a \mid b)$. 
Dubois and Prade \cite{DP94} consider that only the use of the quasi-conjunction $\land_Q$ is appropriate in the context of non-monotonic reasoning, where conditionals $(a \mid b)$  are viewed as non-monotonic conditional assertions or default rules ``if $b$ then generally $a$'', usually denoted as $b \;|\!\!\sim a$. 
In fact, consider an entailment relation between finite sets of conditionals and conditionals as follows: a finite set of conditionals $K = \{(a_i \mid b_i)\}_{i \in I}$ entails another conditional $(c \mid d)$, written $K \models (c \mid d)$, if either $(c \mid d) \in K$  or there is a non-empty subset $S \subseteq K$ s.t. $C(S) \leq (c \mid d)$, where $C(S) = \land_Q \{ (a \mid b) \mid (a \mid b) \in S \}$ is the quasi-conjunction of all the conditionals in $S$. Then they show that this notion of entailment basically coincides with the nonmonotonic consequence relation of System P. More explicitly, they show the following characterisation result: \\

\begin{tabular}{rll}
$K \models (c \mid d)$ & iff & $(c \mid d)$ can be derived from $K$ using the rules of system P \\
& & and the axiom schema $(x \mid x)$. \\
\end{tabular}
\newline

On the other hand, corresponding disjunction operations $\vee_I$ and $\vee_Q$ on conditionals are defined by De Morgan's laws from $\wedge_I,\wedge_Q$ respectively.

All these conjunctions and disjunctions are commutative, associative and idempotent, and moreover they coincide over conditionals with a common antedecent. More precisely, $(a \land c \mid b) = (a \mid b) \land_I (c \mid b) = (a \mid b) \land_Q (c \mid b)$. 
However, none of the algebras $({A\mid A}, \wedge_i, \lor_i, \neg, {\sf T}, {\sf F})$, for $i \in \{I,Q\}$, is in fact a Boolean algebra. 
For instance, Schay \cite{Sh68} and Calabrese \cite{Cal} show that $\wedge_Q$ and $\vee_Q$ %are not even distributive 
 do not distribute
with respect to each other.

It is worth noticing that the above  conjunctions are defined in order to make the class $A\mid A$ of conditional objects closed under $\wedge_i$, and hence an algebra. Therefore, for every $a_1,b_1,a_2,b_2\in A$, and for every $i \in \{I, Q\}$, there exists $c,d\in A$ such that, $(a_1\mid b_1)\wedge_i (a_2\mid b_2)=(c\mid d)$.
On the other hand,  our construction of conditional algebra defines a structure whose domain strictly contains all the elements $(a\mid b)$ for $a$ in $A$, and $b \in A'$. 
This allows us to relax this condition of closure as stated above. Indeed,  for every pair of conditionals of the form $(a_1\mid b_1)$ and $(a_2\mid b_2)$ belonging to the conditional algebra, their conjunction is always an element of the algebra (i.e. the conjunction is a total, and not a partial, operation), but in general it will be not in the form $(c\mid d)$. 
Moreover our definition of conjunction behaves as $\wedge_I$ and $\wedge_Q$ whenever restricted to those conditionals with a common antedecent or context.

\subsection{Conditional Event Algebras}\label{Sub:Comparison2}

In his paper \cite{vanFraassen76}, van Fraassen devises a minimal
logic CE of conditionals whose language is obtained from classical
propositional language by adding a ``conditional symbol''
$\Rightarrow$.  The algebraic counterpart of CE are the so-called {\em
  proposition algebras}. A  proposition algebra is a pair $\langle {\bf F}, \Rightarrow \rangle$ where $\bf F$ is a Boolean algebra (of events) and  $\Rightarrow$ is a partial binary operation on $\bf F$ such that (where defined) satisfying the following requirements:
\begin{itemize}
\item[(I)] $(A \Rightarrow B) \land  (A \Rightarrow C) =  A \Rightarrow (B \land C)$
\item[(II)]  $(A \Rightarrow B) \lor  (A \Rightarrow C) =  A \Rightarrow (B \lor C)$
\item[(III)] $A \land (A \Rightarrow B) = A \land B$
\item[(IV)] $A \Rightarrow A = \top$
\end{itemize}

Propositional algebras lead to probabilistic models by adding a
probability measure $P$ to the algebras and requiring that the following
is satisfied

\begin{itemize}
\item[(m2)] If $P(A)\not= 0$ then $P(A\Rightarrow B)=P(A\wedge B)/P(A)$.
\end{itemize}

In the same paper van Fraassen uses the usual product space
construction to provide models of the above. In particular, starting
from a Boolean algebra of ordinary events $\bf A$ and a probability
measure $P$ on it,  he finds a way of representing conditional events
such that: (1) conditional events $(b \mid a)$ live in a bigger
Boolean algebra $\bf A^*$; (2) ordinary events $a$ are a special kind
of conditional events, i.e. $a$ should be identified with $(a \mid
\top)$;  (3) the rule of modus ponens $a \land (b \mid a) \leq b$
holds true in $\bf A^*$; and (4) the probability measure $P$ on $\bf
A$ can be  extended to probability measure $P^*$ on $\bf A^*$
satisfying (m2) above.

Goodman and Nguyen \cite{GN94} build on this to define a Conditional
Event Algebras (known as {\em Goodman-Nguyen-van Fraassen}
algebra). To compare this to our setting, let us recall this construction in the case where the
initial algebra $\bf A$ is finite. 
Consider the countable Cartesian product $\Omega^\infty =  \atom(\bf A) \times  \atom(\bf A) \times\ldots$, i.e. the set of countably infinite sequences of atoms of $\bf A$. Then  ${\bf A}^*$ is defined as the $\sigma$-algebra of subsets of  $\Omega^\infty$ 
generated by all the cylinder sets of the form 
$$\langle a_1, \ldots,  a_n,  \top, \top,  \ldots \rangle = \{(w_1, w_2, \ldots, w_n, w_{n+1}, \ldots) \in \Omega^\infty \mid w_i \leq a_i, \mbox{ for } i=1,\ldots,n\},$$ 
for $n \in \mathbb{N}$ and  $a_1, \ldots, a_n \in  A$. 
Then, for every pair of events $a,b\in A$, Goodman and Nguyen define the {\em conditional} event $(b\mid a)$ as the element of $A^*$ that is the following countable union of pairwise disjoint cylinder sets: 
 \begin{equation}\label{eq:goodman}
(b \mid a) :=  \bigcup_{k\geq 0} \langle \neg a, \stackrel{k}{\ldots}, \neg a, a\wedge b, \top,  \top\ldots \rangle \; . 
 \end{equation}
Note that the initial algebra of events $\bf A$ is isomorphic to the subalgebra of $\bf A^*$ of conditionals of the form $(a \mid \top) = \langle a, \top, \top,\ldots \rangle$, with $a \in A$. 
Now, any probability measure $P$ on $\bf A$ extends to a suitable probability $P^*$ on $\bf A^*$. Namely, one can define $P^*$ on the cylinder sets of $\bf A^*$ as 
$$P^*(\langle a_1, \ldots,  a_n,  \top, \top,  \ldots \rangle) = P(a_1) \cdot \ldots \cdot P(a_n),$$ 
and using Kolmogorov extension theorem (cf. \cite[\S2.4]{Tao}), $P^*$ can be extended to the whole $\sigma$-algebra $A^*$.

Finally, they show that the probability $P^*$ on conditionals actually coincides with the conditional probability:  if $P(a) > 0$, then one has: 

\begin{equation}\label{eq:goodmanProb}
 \begin{array}{lll} 
P^*((b\mid a)) & = & P^*\left( \displaystyle\bigcup_{k\geq 0} \langle \neg a, \stackrel{k}{\ldots}, \neg a, a\wedge b, \top,  \top\ldots \rangle\right) \\
& = & \displaystyle\sum_{k \geq 0} P^*(\langle \neg a, \stackrel{k}{\ldots}, \neg a, a\wedge b, \top,  \top\ldots \rangle) \\
& = & P(a \land b) \cdot \displaystyle\sum_{k\geq 0} ( 1-P (a))^k\\ 
& = & \displaystyle\frac{P(a\wedge b)}{P(a)},  
 \end{array}
\end{equation}
where the last equality is obtained by using the well-known formula for the sum of a geometric series, in this case the ratio being $1-P(a)$, and  getting $\sum_{k\geq 0}  (1-P (a))^k = 1/P(a)$. 

 Summing up, both our approach and the Goodman-Nguyen-van Frassen  one
 aim at the introduction of a formal structure to study conditional probability as  measures of conditionals. To this end we both require the underlying algebraic structure to be Boolean.   On the other hand, and this marks a first  technical difference, whenever the starting Boolean algebra $\bf A$ is finite, $\mathcal{C}({\bf A})$ is finite as well while ${\bf A}^*$ is always infinite, and moreover each conditional in ${\bf A}^*$ is defined itself as an infinitary joint, see (\ref{eq:goodman}) above.  Note that, for equating $P^*((b \mid a))$ with $P(a \land b) / P(a)$, the derivation \eqref{eq:goodmanProb} necessarily requires $P^*((b \mid a))$ to decompose in infinitely-many summands.  

 A further important difference concerns the definition of conditional
 events. 
 In fact, although in our approach the algebra $\mathcal{C}({\bf A})$ itself defines the conditionals and characterizes their properties at a formal algebraic level,  
 in $\bf A^*$ conditionals are identified with a particular subset of its elements, namely those elements that can be expressed via (\ref{eq:goodman}) above. This difference owes  mainly  to the fact that the two approaches really capture distinct intuitions about conditional probability.

\section{Conclusions and future work}\label{sec:conclusion}
In this paper we have introduced a construction which defines, for every Boolean algebra of events ${\bf A}$, its corresponding Boolean algebra of conditional events $\mathcal{C}(\bf A)$. Our construction preserves the finiteness of ${\bf A}$, in the sense that $\mathcal{C}({\bf A})$ is finite whenever so is ${\bf A}$. 
Actually, 
in the case of finite algebras, we have provided a full characterization of the atomic structure of $\mathcal{C}({\bf A})$ and proved that any positive probability measure $P$ on ${\bf A}$ canonically extends to a (plain) positive probability $\mu_P$ on $\mathcal{C}({\bf A})$ that coincides with the conditional probability induced by $P$, that is, it satisfies the condition: $$\mu_P(a\mid b)=\frac{P(a\wedge b)}{P(b)}, $$ for any events $a, b \in A$. 
Finally we have introduced the logic LBC to reason about conditionals which is sound and complete with respect to a semantics defined in accordance with the notion of Boolean algebra of conditionals. 
Moreover, we have pointed out the tight connections of this logic with preferential nonmonotonic consequence relations.

Several interesting questions remain open and we aim to address them in future work. 
In this work we have not considered algebras with iterated conditionals. It would be interesting to study whether it is possible  to introduce in the algebras of conditionals $\mathcal{C}(\bf A)$ a proper binary operator capturing the notion of iterated conditioning, for instance in the sense of whether an object of the form $((a \mid b) \mid (c \mid d))$ can be defined as another element of $\mathcal{C}(\bf A)$ in a meaningful way, i.e.\ without having to resort to a meta-structure of the kind  $\mathcal{C}(\mathcal{C}(\bf A))$.

As we have shown, given a positive probability $P$ on an algebra of events $\bf A$, we can extend it to a probability $\mu_P$ on the whole algebra of conditionals $\mathcal{C}(\bf A)$, and hence it is in principle possible to compute the probability of any compound conditional. However it would be interesting to investigate whether there are more operational rules for computing the probability of conjunctions and disjunctions of basic conditionals. In particular, one could check whether rules appearing in e.g. \cite{GN94,Kau} apply also in our framework.

 On the measure-theoretical side we  leave open the problem of generalizing Theorem \ref{thm:main1} to the case when $P$ is a not necessarily positive probability on $A$. Preliminary investigations in this direction shows that a result of this kind needs a deeper algebraic analysis of the Boolean algebras of conditionals $\mathcal{C}(\bf A)$, in particular on the relation between the congruences of ${\bf A}$ and those of $\mathcal{C}({\bf A})$.

At the foundational level, a pressing question will be to
 investigate Boolean algebras of conditionals in light of de Finetti's
 coherence criterion for conditional assignments. In particular we
 will investigate if, or up to which extent, the coherence of a
 ``book'' on conditional events $(a_1\mid b_1),\ldots, (a_k\mid b_k)$
 can be characterized in terms of the coherence of an ``unconditional
 book'' on   the $(a_i\mid b_i)$'s viewed as elements of a Boolean
 algebra of conditionals. A satisfactory solution to this problem
 would then motivate the extension of this coherence-based analysis to
 non-probabilisitic measures of uncertainty, along the lines of \cite{Flaminio2015}. 

More generally, a natural question to be addressed in future research
has to do with defining non-probabilistic analogues of the relation
established in the present paper between Boolean algebras of
conditionals and conditional probabilities. Most interesting targets include possibility and
necessity measures, ranking functions, belief and plausibility
functions, and imprecise probabilities.

As to those latter recall that, as a consequence of Corollary
6.6, convexity fails, in general, for sets of separable
probabilities. This  raises the question as to how
separable probabilities relate to the long standing problem (see e.g.
\cite{Cozman2012}) of reconciling various forms of qualified stochastic independence with
convex sets of probabilities.

\subsection*{Acknowledgments} The authors are thankful to the anonymous reviewers for their helpful remarks and suggestions. They would also like to thank Didier Dubois, Marcelo Finger, Henri Prade  and Giuseppe Sanfilippo for fruitful discussions on the arguments of the present paper. 
Flaminio and Godo acknowledge partial support by the Spanish FEDER/MINECO project TIN2015-71799-C2-1-P. Flaminio also acknowledges partial support by the Spanish Ram\'on y Cajal research program RYC-2016-19799.
Hosni's research was funded by the Department of Philosophy ``Piero
Martinetti" of the University of Milan under the Project ``Departments
of Excellence 2018-2022" awarded by the Ministry of Education,
University and Research (MIUR). He also acknowledges funding from the Deutsche Forschungsgemeinschaft (DFG, grant LA 4093/3-1).

%\section*{References}

\appendix

\section{Proofs}

\subsection*{Proofs from Section \ref{sec:BAC}}
\vspace{.2cm}

\noindent{\bf Proposition \ref{prop:ordine}.} 
{\em 
In every  algebra $\mathcal{C}({\bf A})$ the following properties hold for  every $a,c\in A$ and $b\in A'$:
\begin{enumerate}[(i)]
\item\label{p1} $(a\mid b)\geq (b\mid b)$ iff $a\geq b$;
\item\label{p2}  if $a\leq c$, then $(a\mid b)\leq (c\mid b)$; in particular $a\leq c$ iff $(a\mid \top)\leq (c\mid\top)$; 
\item \label{p3} if $a \leq b \leq d$, then $ (a \mid b) \geq (a \mid d)$; in particular $(a \mid b) \geq (a \mid a \lor b)$; 
\item\label{p5}  if $(a\mid b)\neq(c\mid b)$, then $a\land b \neq c \land b$;
\item\label{p6}  $(a\wedge b\mid \top)\leq(a \mid b) \leq (b \to a \mid \top) $;
\item\label{p7} if  $a \land d = \bot$ and $ \bot < a \leq b$, then $(a \mid \top) \sqcap  (d \mid b) = \bot_\mathfrak{C}$; 
\item\label{p8}  $(b \mid \top) \sqcap (a \mid b) \leq (a \mid \top) $.
\end{enumerate}
}

 \begin{proof}
\noindent\eqref{p1}. By definition, $(a\mid b)\geq (b\mid b)$ iff $(a\mid b)\sqcap(b\mid b)=(b\mid b)$ iff, by Proposition \ref{prop2} \eqref{e4},  $(b\mid b)=(a\wedge b\mid b)=(a\mid b)$. Finally, let us prove that if $a\not\geq b$, then $(a\mid b)<\top_\mathfrak{C}$. If $a\not\geq b$, then $a\wedge b<b$ and hence there exists $c\in A$, different from $\bot$,  such that $b=(a\wedge b)\vee c$. Therefore, from Proposition \ref{prop2} (i) and Proposition \ref{prop3} (iii), $\top_\mathfrak{C}=(b\mid b)=((a\wedge b)\vee c\mid b)=(a\wedge b\mid b)\vee (c\mid b)$ and $(c\mid b)\neq \bot_{\mathfrak{C}}$. Thus, $(a\mid b)=(a\wedge b\mid b)<\top_\mathfrak{C}$.
\vspace{.2cm}

\noindent\eqref{p2}. If $a\leq c$, then $a\wedge c=a$, and hence   $(a\wedge c\mid b)=(a\mid b)$. Therefore by Proposition \ref{prop2} \eqref{e2}, $(a\mid c)\sqcap (c\mid b)=(a\mid b)$, and $(a\mid b)\leq (c\mid b)$. Moreover, if $(a\mid \top)\leq(c\mid \top)$, then easily $(a\to c\mid\top)=\top_\mathfrak{C} =(\top\mid\top)$. From Proposition \ref{prop3} \eqref{e5}, $a\to c=\top$, and hence 
 $a\leq c$ and the first part of \eqref{p2} holds.
 
\noindent\eqref{p3}.
 If  $a \leq b \leq d$, then Proposition \ref{prop2} \eqref{e8} implies that $(a \mid d)  = (a \mid b)  \sqcap  (b \mid d)$, and hence  $(a \mid d)  \leq  (a \mid b)$.
 \vspace{.2cm}

\noindent\eqref{p5}. It directly follows from Proposition \ref{prop2} \eqref{e4}.  \vspace{.2cm}

\noindent\eqref{p6}.
First of all, $(a \mid b)=(a\wedge b\mid b)\geq (a\wedge b\mid \top)$ from Proposition \ref{prop2} \eqref{e4} and the last claim of \eqref{p2}. In order to show the second inequality, assume first $a \leq b$. Since $b \lor \neg b = \top$ we have $(a \mid b) = (a \mid b) \sqcap (b \lor \neg b \mid b) = ((a \mid b) \sqcap (b \mid \top)) \sqcup ((a \mid b) \sqcap (\neg b \mid \top))$. Now, from Proposition \ref{prop2} \eqref{e8}, $(a \mid b) \sqcap (b \mid \top) = (a \mid \top)$, and clearly $(a \mid b) \sqcap (\neg b \mid \top) \leq (\neg b \mid \top)$, therefore $(a \mid b) \leq (a \mid \top) \sqcup (\neg b \mid \top) = (\neg b \lor a \mid \top)$. Now, in the general case, $(a \mid b) = (a \land b \mid b) \leq (\neg b \lor (a \land b)  \mid \top) = (\neg b \lor a \mid \top)=(b\to a\mid \top)$. 
\vspace{.2cm}

\noindent \eqref{p7}. We have  $(a \mid \top) \leq (a \mid b)$ from \eqref{p2}. Then, 
 $(a \mid \top) \sqcap  (d \mid b) \leq  (a \mid b) \sqcap  (d \mid b) = (a \land d \mid b) = (\bot \mid b) = \bot_\mathfrak{C}$. 
\vspace{.2cm}

\noindent  \eqref{p8}. From Proposition \ref{prop2} \eqref{e4} and \eqref{p2}, it follows that %and $({\bf o}2)$, 
$(b\mid \top)\sqcap(a\mid b)=(b\mid \top)\sqcap(a\wedge b\mid b)\leq (b\mid \top)\sqcap(a\wedge b\mid \top)=(a\wedge b\mid \top)\leq (a\mid \top)$.
\end{proof}

\noindent{\bf Proposition \ref{prop5}.}
{\em 
In every  algebra $\mathcal{C}({\bf A})$ the following properties hold for all $a, a'\in A$ and $b, b'\in A'$:
\begin{enumerate}[(i)]
\item\label{p12} $(a \mid b) \sqcap (a \mid b') \leq (a \mid b \lor b')$; 
in particular, $(a \mid b) \sqcap (a \mid \neg b) \leq (a \mid \top)$;  

\item \label{p11}  if $a \leq b \land b'$, then $(a \mid b) \sqcap (a \mid b') = (a \mid b \lor b')$; 
\item\label{p13} $(a\mid b)\leq (b\to a\mid b\vee b')$;
\item\label{p14} $(a\mid b)\sqcap (a'\mid b')\leq ((b\to a) \wedge (b'\to a') \mid b\vee b')$.

\end{enumerate}
}

\begin{proof}
\noindent\eqref{p12}. For all $a\in A$ and $b$ and $c$ in $A'$, since $(b \vee c\mid b\vee c)=\top_\mathfrak{C}$, $(a\mid b)\sqcap(a\mid c)=(a\mid b)\sqcap(a\mid c)\sqcap(b\vee c\mid b\vee c)=(a\mid b)\sqcap(a\mid c)\sqcap((b\mid b\vee c)\sqcup(c\mid b\vee c)))=$
\begin{equation}\label{eqProp11}
=((a\mid b)\sqcap(a\mid c)\sqcap(b\mid b\vee c))\sqcup((a\mid b)\sqcap(a\mid c)\sqcap(c\mid b\vee c)).
\end{equation}
Now, $(a\mid b)\sqcap(a\mid c)\sqcap(b\mid b\vee c)\leq (a\mid b)\sqcap(b\mid b\vee c)= (a \wedge b \mid b)\sqcap(b\mid b\vee c)=(a\wedge b\mid b\vee c)$, where the first equality is due to Proposition \ref{prop2} \eqref{e4}  and the second one to Proposition \ref{prop2} \eqref{e8}. Analogously, we have $(a\mid b)\sqcap(a\mid c)\sqcap(c\mid b\vee c)\leq (a\mid c)\sqcap(c\mid b\vee c)=(a\wedge c\mid b\vee c)$. Thus, by (\ref{eqProp11}), we get that $(a\mid b)\sqcap(a\mid c)\leq (a\wedge b\mid b\vee c)\sqcup (a\wedge c\mid b\vee c)=(a\wedge(b\vee c)\mid b\vee c)=(a\mid b\vee c)$, from Proposition \ref{prop2} \eqref{e4}. 

Obviously the particular case follows by taking $b'=\neg b$. 

\noindent\eqref{p11}. Now, assume $a\leq b\wedge b'$. Then, from (\ref{eqProp11}) the claim follows provided that, under this further hypothesis, $(a\mid b)\sqcap(a\mid b')\sqcap(b\mid b\vee b')= (a\mid b)\sqcap(b\mid b\vee b')$ and $(a\mid b)\sqcap(a\mid b')\sqcap(b'\mid b\vee b')= (a\mid b')\sqcap(b'\mid b\vee b')$. Let us prove the first equality, the second being completely analogous. 

First of all $(a\mid b)=(a\wedge b\mid b)$ (recall Proposition \ref{prop2} \eqref{e4}), whence $(a\mid b)\sqcap(a\mid b')\sqcap(b\mid b\vee b'))=(a\wedge b\mid b)\sqcap(a\mid b')\sqcap(b\mid b\vee b')=(a\wedge b\mid b\vee b')\sqcap(a\mid b')$ (the last equality follows from Proposition \ref{prop2} \eqref{e8}). Now, $(a\wedge b\mid b\vee b')\leq (a\mid b\vee b')\leq (a\mid b')$ because by hypothesis $a\leq b'\leq b\vee b'$. Thus, $(a\wedge b\mid b\vee b')\sqcap(a\mid b')=(a\wedge b\mid b\vee b')=(a\mid b\vee b')$ since $a\leq b$. Finally, since again $a\leq b\leq b\vee b'$, from Proposition \ref{prop2} \eqref{e8}, one has  $(a\mid b\vee b')=(a\mid b)\wedge (b\mid b\vee b')$ as desired.

 \vspace{.2cm}
 
\noindent\eqref{p13}. In every Boolean algebra it holds that $b\to a = \neg b \lor a \geq \neg b$. Thus $b\to a\geq \neg b\geq c\wedge \neg b$. Thus, from Proposition \ref{prop:ordine} \eqref{p1}, $(b\to a \mid c\land \neg b)=\top_{\mathfrak{C}}$. Hence, $(a\mid b)\leq (b\to a\mid b)$ and $(a\mid b)\leq (b\to a \mid c\land \neg b)$, whence $(a\mid b)\leq (b\to a\mid b)\sqcap (b\to a \mid c\land \neg b)$. The latter, from \eqref{p12}, is less or equal to $(b\to a\mid (c\wedge \neg b)\vee b)=(b\to a \mid c\vee b)$. This settles the claim.
 \vspace{.2cm}
 
\noindent\eqref{p14}. From \eqref{p13}, $(a\mid b)\leq (b\to a\mid b\vee b')$ and $(a'\mid b')\leq (b' \to a'\mid b\vee b')$. Thus, $(a\mid b)\sqcap (a'\mid b')\leq (b\to a\mid b\vee b')\sqcap(b'\to a'\mid b\vee b')=((b\to a)\wedge (b'\to a')\mid b\vee b')$.
 \end{proof}

 %%%%
 \subsection*{Proofs from Section \ref{sec:atoms}}
 %%%%%
 \vspace{.2cm}

\noindent{\bf Proposition \ref{useful}.} 
{\em 
$Part_i(\mathcal{C}({\bf A}))$ is a  partition of $\mathcal{C}({\bf A})$. }
\vspace{.1cm}

\noindent\begin{proof}  We have to prove the two following conditions: 
\vspace{.2cm}

\noindent(a) $\bigsqcup Part_i(\mathcal{C}({\bf A})) = \top_\mathfrak{C}$, 
\vspace{.2cm}

\noindent(b) for  distinct $t_1, t_2 \in Part_i(\mathcal{C}({\bf A}))$, $t_1 \sqcap t_2 = \bot_\mathfrak{C}$. 
\vspace{.2cm}

\noindent
(a) We shall prove this claim by induction on $i$.  The case $i= 1$ is easy as $Seq_1({\bf A}) = \{ \langle \alpha \rangle \mid \alpha \in \atom({\bf A})\}$ and clearly $\bigsqcup_{\alpha\in \atom({\bf A})} (\alpha \mid \top)  = (\bigvee_\alpha \alpha \mid \top) = \top_\mathfrak{C}$. 

If $1 < i$, let $ 1 < j \leq i$ and suppose the claim is true for $j-1$, that is,  $\bigsqcup Part_{j-1}(\mathcal{C}(A)) = \top_\mathfrak{C}$. We have to prove that $\bigsqcup Part_{j}(\mathcal{C}(A)) = \top_\mathfrak{C}$ as well. 
 For each sequence $\overline{\beta} = \langle \beta_1, \ldots, \beta_{j-1}\rangle \ \in Seq_{j-1}({\bf A})$, consider its corresponding  compound conditional $\omega_{\overline{\beta}} = (\beta_1 \mid \top)  \sqcap \ldots \sqcap (\beta_{j-1} \mid \neg \beta_1 \land \ldots \land \neg \beta_{j-2})$. By inductive hypothesis, we have 
 $$
 \bigsqcup_{\overline{\beta} \in Seq_{j-1}({\bf A})} \omega_{\overline{\beta}} = \bigsqcup Part_{j-1}(\mathcal{C}({\bf A})) = \top_\mathfrak{C}.
 $$

Further, let $D(\overline{\beta}) = \atom({\bf A}) \setminus \{\beta_1, \ldots \beta_{j-1}\}$ be the set of $n-j+1$ atoms of ${\bf A}$ disjoint from  $\{\beta_1, \ldots \beta_{j-1}\}$.  
Then it is clear that 
$$
\bigsqcup_{\alpha \in D(\overline{\beta})} (\alpha \mid \neg\beta_1 \land \ldots \land \neg\beta_{j-1}) = \top_\mathfrak{C},
$$ 
and thus $\omega_{\overline{\beta}} = \bigsqcup_{\alpha \in D(\overline{\beta})}  \omega_{\overline{\beta}} \sqcap (\alpha \mid \neg\beta_1 \land \ldots \land \neg\beta_{j-1})$. 

Therefore, since this holds for every sequence $\overline{\beta} \in Seq_{j-1}$, we finally get: 
\begin{align*}
 \top_\mathfrak{C} & =  \bigsqcup_{\overline{\beta} \in Seq_{j-1}({\bf A})} \omega_{\overline{\beta}}  
=  \bigsqcup_{\overline{\beta} \in Seq_{j-1}({\bf A})} 
\left (  \bigsqcup_{\beta \in D(\overline{\beta})}  \omega_{\overline{\beta}} \sqcap (\beta \mid \neg\beta_1 \land \ldots \land \neg\beta_{j-1}) 
\right ) \\
&= \bigsqcup_{\overline{\delta} \in Seq_{j}({\bf A})} \omega_{\overline{\delta}}  = \bigsqcup Part_{j}(\mathcal{C}({\bf A})). 
 \end{align*}

\noindent (b) Let $ \omega_{\overline{\alpha}},  \omega_{\overline{\beta}} \in Part_{i}(\mathcal{C}({\bf A}))$ with $\overline{\alpha} \neq \overline{\beta}$. To be more precise, if $\overline{\alpha} = \langle \alpha_1, \ldots, \alpha_i \rangle$ and $\overline{\beta} = \langle \beta_1, \ldots, \beta_i \rangle$, let $1\leq k \leq i$ the minimum index such that $\alpha_k \neq \beta_k$. Then it holds that $(\alpha_k \mid \neg \alpha_1 \land \ldots \neg \alpha_{k-1}) \sqcap (\beta_k \mid \neg \beta_1 \land \ldots \neg \beta_{k-1}) = \bot_\mathfrak{C}$ since $\neg \alpha_1 \land \ldots \neg \alpha_{k-1} = \neg \beta_1 \land \ldots \neg \beta_{k-1}$ and $\alpha_k \land \beta_k = \bot$. Then, the claim follows from observing that $\overline{\alpha} \sqcap \overline{\beta} \leq (\alpha_k \mid \neg \alpha_1 \land \ldots \neg \alpha_{k-1}) \sqcap (\beta_k \mid \neg \beta_1 \land \ldots \neg \beta_{k-1})$. 
\end{proof}

\noindent{\bf Corollary \ref{lemma:counting}.}
{\em Let ${\bf A}$ be a Boolean algebra with $|\atom({\bf A})|=n$.
For every basic conditional $(a\mid b)\in \mathcal{C}({\bf A})$ with $a\leq b$, $|\atom_\leq(a\mid b)|=n!\cdot \frac{|\atom_\leq(a)|}{|\atom_\leq(b)|}$. }
\vspace{.1cm}

\noindent\begin{proof}
We start proving the following
\begin{claim}\label{claim:counting}
 If $a=\alpha$ is an atom of ${\bf A}$, then $|\atom_\leq(\alpha\mid b)|=\frac{n!}{|\atom_\leq(b)|}$.
 \end{claim}
 \begin{proof} (of Claim \ref{claim:counting})
 Obviously, if $b$ is an atom, and since we are assuming $a\leq b$, it must be $a=b=\alpha$. Then $|\atom_\leq(a\mid b)|=|\atom_\leq(\alpha\mid \alpha)|=|\atom(\top_\mathfrak{C})|=n!$. Conversely, if $b$ is not an atom, let $\alpha_1,\ldots, \alpha_k\in \atom({\bf A})$ (with $k>1$) such that $b=\bigvee_{i=1}^k \alpha_i$. Further, since $\alpha\leq b$, there is $i_0\leq k$ such that $\alpha_{i_0}=\alpha$. For simplicity, and without any loss of generality, assume $i_0=1$. Hence, 
 $$
 n!=|\atom_\leq(\alpha_1\vee \alpha_2\vee\ldots\vee\alpha_k\mid b)|=\sum_{i=1}^k |\atom_\leq(\alpha_i\mid b)| 
 $$
 and hence, by a symmetric argument, $|\atom_\leq(\alpha\mid b)|=\frac{n!}{k}=\frac{n!}{|\atom_\leq(b)|}$ and Claim \ref{claim:counting} is settled. 
 \end{proof}
 Coming back to the proof of Corollary \ref{lemma:counting}, and remembering from Proposition \ref{lemma:repres} that $a\mid b=\bigvee_{\alpha\leq x}(\alpha\mid y)$, we have that, thanks to Claim \ref{claim:counting}, $|\atom_\leq(a\mid b)|=|\atom_\leq(a)|\cdot \frac{n!}{|\atom_\leq(b)|}$. 
\end{proof}
%%%%
\subsection*{Proofs from Section \ref{sec:measures0}}
%%%%%
\vspace{.2cm}

\noindent
For the proof of next lemma recall the tree structure $\mathbb{T}$ introduced in Section \ref{firsttree}. 
\vspace{.2cm}

\noindent{\bf Lemma \ref{MuPProb}.} {\em The map $\mu_P$ is  a probability distribution on $\atom(\mathcal{C}({\bf A}))$, that is, 
$$\sum_{\overline{\alpha}  \in Seq({\bf A})} \mu_P(\omega_{\overline{\alpha}}) = 1.$$
}

\begin{proof} 
Let $|\atom({\bf A})|=n$ and
let  $\mathbb{T}$ be as described in Section \ref{sec:atoms}. Attach to each node of $\mathbb{T}$ the following values:

\begin{itemize}
\item Level 0: recall that the root node is $(\top\mid \top)$. Thus attach $P(\top\mid \top)=1$.
\item Level 1: attach to each node $(\alpha_i\mid \top)$ of level 1 the value $P(\alpha_i)$.
By definition of atom and since $P$ is a probability measure, we have that 
$$
\sum_{\alpha \in \atom({\bf A})} P(\alpha) = 1.
$$ 
\item Level $i$: For each node $\beta_{i-1}$ at level $i-1$, let $\langle \beta_1, \ldots, \beta_{i-1}\rangle$ the partial sequence corresponding to the path from $(\top\mid\top)$ to $\beta_{i-1}$. Then attach to the $n-i$ children nodes 
the value $P(\beta \mid \neg \beta_1 \land \ldots \land \neg \beta_{i-1})$.
Note again that 
$$
\sum_{\beta \in \atom({\bf A}) \setminus \{\beta_1, \ldots, \beta_{i-1}\}} P(\beta \mid \neg \beta_1 \land \ldots \land \neg \beta_{i-1}) = 1.
$$ 
\end{itemize}

Recall also that, by construction of $\mathbb{T}$,  $\atom(\mathcal{C}({\bf A}))$ is in 1-1 correspondence with the paths from the root to the  leaf nodes of the tree. 
Furthermore, the above procedure attaches to each node $x$ of $\mathbb{T}$, the value 
$P(x)=\sum_{\omega_{\overline{\alpha}}\leq x}\mu_P(\omega_{\overline{\alpha}})$. Thus, in particular, for every leaf $l$, $P(l)=\mu_P(\omega_{\overline{\alpha_l}})$ where $\overline{\alpha_l}$ is the sequence in $Seq({\bf A})$ which uniquely corresponds to the path from $\top\mid\top$ to $l$. Thus, by the above construction, letting $L_\mathbb{T}$ the set of leafs of $\mathbb{T}$, one has
$$
1=\sum_{l\in L_\mathbb{T}}P(l)=\sum_{l\in L_\mathbb{T}}\mu_P(\omega_{\overline{\alpha_l}})=\sum_{\overline{\alpha}\in Seq({\bf A})}\mu_P(\omega_{\overline{\alpha}}).
$$
Thus, the claim is settled.
\end{proof}
\noindent{\bf Lemma \ref{lemma:prep1}.}
{\em Let $\alpha_{i_1},\ldots,\alpha_{i_t}\in \atom({\bf A})$. Then
\begin{enumerate}[(i)]
\item
$\mu_P(\llb \alpha_{i_1},\ldots,\alpha_{i_t} \rrb)=P(\alpha_{i_1})\cdot \frac{P(\alpha_{i_2})}{P(\neg \alpha_{i_1})}\cdot\ldots\cdot\frac{P(\alpha_{i_t})}{P(\neg\alpha_{i_1}\wedge\neg\alpha_{i_2}\wedge\ldots\wedge\neg\alpha_{i_{t-1}})}$.
\item $\mu_P(\llb \alpha_{i_1},\ldots,\alpha_{i_t} \rrb)=\mu_P(\llb \alpha_{i_1},\ldots,\alpha_{i_{j-1}},\alpha_{i_{j+1}},\ldots,\alpha_{i_t} \rrb)\cdot \frac{P(\alpha_{i_j})}{P(\neg\alpha_{i_1}\wedge\neg\alpha_{i_2}\wedge\ldots\wedge\neg\alpha_{i_{t-1}})}.$
\end{enumerate}}

\begin{proof}
(i). Let $\gamma_1,\ldots, \gamma_l$  be the atoms of ${\bf A}$ different from $ \alpha_{i_1},\ldots,\alpha_{i_t}$. Thus, $\omega_{\overline{\delta}}\in \llb\alpha_{i_1},\ldots,\alpha_{i_t}\rrb$ iff $\overline{\delta}=\langle\alpha_{i_1},\ldots,\alpha_{i_t}, \sigma \rangle$ for $\sigma$ any string which is obtained by permuting $l-1$ elements in $\{\gamma_1,\ldots, \gamma_l\}$. 

Therefore, letting $\Psi=\neg\alpha_{i_1}\wedge\ldots\wedge\neg\alpha_{i_t}$, $K=P(\alpha_{i_1})\cdot \frac{P(\alpha_{i_2})}{P(\neg \alpha_{i_1})}\cdot\ldots\cdot\frac{P(\alpha_{i_t})}{P(\neg\alpha_{i_1}\wedge\neg\alpha_{i_2}\wedge\ldots\wedge\neg\alpha_{i_{t-1}})}$ and $H=\sum_\sigma \frac{P(\sigma_1)}{P(\Psi)}\cdot\frac{P(\sigma_2)}{P(\Psi\wedge\neg\sigma_1)}\cdot\ldots\cdot \frac{P(\sigma_{l-1})}{P(\Psi\wedge\neg\sigma_1\wedge\ldots\wedge\neg\sigma_{n-2})}$, one has
$$
\mu_P(\llb \alpha_{i_1},\ldots,\alpha_{i_t} \rrb)=K\cdot H.
$$ 
We now prove by induction on $l$ that $H=1$.
\vspace{.2cm}

\noindent (Case $0$)
The basic case is for $l=2$. In this case $H=\frac{P(\gamma_1)}{P(\Psi)}+\frac{P(\gamma_2)}{P(\Psi)}=\frac{P(\gamma_1\vee\gamma_2)}{P(\Psi)}$. Thus, the claim trivially follows because $\Psi=\gamma_1\vee\gamma_2$.

\vspace{.2cm}

\noindent (Case $l$)
For any $j=1,\ldots, l$ we focus on the strings $\sigma$ whose first coordinate $\sigma_1=\gamma_j$. Therefore,
$$
\begin{array}{lll}
H&=&\frac{P(\gamma_1)}{P(\Psi)}\cdot\sum_{\sigma : \sigma_1=\gamma_1}(\frac{P(\sigma_2)}{P(\Psi\wedge\neg\gamma_1)}\cdot \ldots\cdot \frac{P(\sigma_{l-1})}{P(\Psi\wedge\neg\sigma_1\wedge\ldots\wedge\neg\sigma_{n-2})})+\ldots\\
&&\ldots+\frac{P(\gamma_l)}{P(\Psi)}\cdot\sum_{\sigma : \sigma_1=\gamma_l}(\frac{P(\sigma_2)}{P(\Psi\wedge\neg\gamma_1)}\cdot \ldots\cdot \frac{P(\sigma_{l-1})}{P(\Psi\wedge\neg\sigma_1\wedge\ldots\wedge\neg\sigma_{n-2})}).
\end{array}
$$
By inductive hypothesis, each term $\sum_{\sigma : \sigma_1=\gamma_j}(\frac{P(\sigma_2)}{P(\Psi\wedge\neg\gamma_1)}\cdot \ldots\cdot \frac{P(\sigma_{l-1})}{P(\Psi\wedge\neg\sigma_1\wedge\ldots\wedge\neg\sigma_{n-2})})$ equals $1$. Thus, 
$H=\sum_{j=1}^l \frac{P(\gamma_j)}{P(\Psi)}=1$ since $\Psi=\bigvee_{j=1}^l \gamma_j$.

\vspace{.2cm}

\noindent(ii). The claim follows from (i) and direct computation.
\end{proof}
%%%
\noindent{\bf Lemma \ref{lemma:cases0}.}
{\em Let $\alpha_1\mid b$ be a basic conditional and let  
$\neg b=\beta_1\vee\ldots,\vee \beta_k$ with $\beta_k=\alpha_n$. Then the following hold.
 For all $t\in \{1,\ldots, k-1\}$, 
$$
\mu_P(\mathbb{B}\llb\alpha_n,\beta_1,\ldots, \beta_t,\alpha_1\rrb)=\mu_P(\llb\alpha_n,\beta_1,\ldots, \beta_{t-1},\alpha_1\rrb)\cdot \frac{P(\beta_t)}{P(b)}.
$$}
\begin{proof}
We prove the claim by reverse induction on $t$. 
The basic case is hence for $t=k-1$. In that case, by construction of the set $\llb\alpha_n,\beta_{\pi(1)},\ldots, \beta_{\pi(t)},\alpha_1\rrb$, one has that $\llb\alpha_n,\beta_1,\ldots, \beta_{k-1},\alpha_1\rrb$  is a leaf of $\mathbb{B}$ (recall Definition \ref{def:treeT}). Thus,  
$$
\mathbb{B}\llb\alpha_n,\beta_1,\ldots, \beta_{k-1},\alpha_1\rrb=
\llb\alpha_n,\beta_1,\ldots,\beta_{k-1}\alpha_1\rrb.
$$ 
Therefore, $\mu_P(\mathbb{B}\llb\alpha_n,\beta_1,\ldots, \beta_{k-1},\alpha_1\rrb)=\mu_P(\llb\alpha_n,\beta_1,\ldots, \beta_{k-1},\alpha_1\rrb)$ and  by Lemma \ref{lemma:prep1} (ii), the latter equals 
$$
\mu_P(\llb\alpha_n,\beta_1,\ldots, \beta_{k-2},\alpha_1\rrb)\cdot \frac{P(\beta_{k-1})}{P(\neg \alpha_n\wedge\neg\beta_1\wedge\ldots\wedge\neg \beta_{k-1})}.
$$ 
Therefore, the claim follows by noticing that   $b=\neg \alpha_n\wedge\neg\beta_1\wedge\ldots\wedge\neg \beta_{k-1}$ as we put $\alpha_n=\beta_k$.

Now, let $t<k-1$. Thus, by construction of $\mathbb{B}$, $\mu_P(T\llb\alpha_n,\beta_1,\ldots, \beta_t,\alpha_1\rrb)=$
$$
\mu_P(\llb\alpha_n,\beta_1,\ldots, \beta_{t},\alpha_1\rrb)+\sum_{ \beta_l\not\in\{\beta_1,\ldots, \beta_{t}\}}\mu_P(\mathbb{B}\llb\alpha_n,\beta_1,\ldots,\beta_t,\beta_l,\alpha_1\rrb)
$$
and the inductive hypothesis  ensures that  
$$
\mu_P(\mathbb{B}\llb\alpha_n,\beta_1,\ldots,\beta_t,\beta_l,\alpha_1\rrb)=\mu_P(\llb\alpha_n,\beta_1,\ldots, \beta_t,\alpha_1\rrb)\cdot\frac{P(\beta_l)}{P(b)}.
$$
Thus,
$$
\mu_P(\llb\alpha_n,\beta_1,\ldots, \beta_{t},\alpha_1\rrb)+\sum_{\beta_l\not\in\{\beta_1,\ldots, \beta_{t}\}}\mu_P(\mathbb{B}\llb\alpha_n,\beta_1,\ldots,\beta_t,\beta_l,\alpha_1\rrb)=$$
$$=\mu_P(\llb\alpha_n,\beta_1,\ldots, \beta_t,\alpha_1\rrb)\cdot \left(1+\frac{P(d)}{P(b)}\right),
$$
where $d=\bigvee_{ \beta_l\not\in\{\beta_1,\ldots, \beta_{t}\}}\beta_l$. Again, Lemma \ref{lemma:prep1} (2) gives,
$$
\mu_P(\llb\alpha_n,\beta_1,\ldots, \beta_t,\alpha_1\rrb)=\mu_P(\llb\alpha_n,\beta_1,\ldots, \beta_{t-1},\alpha_1\rrb)\cdot \frac{P(\beta_t)}{P(\neg\alpha_n\wedge\neg\beta_1\ldots\wedge\neg \beta_t)}.
$$ 

Notice that $1+\frac{P(d)}{P(b)}=\frac{P(b\vee d)}{P(b)}$ and $b\vee d=\neg\alpha_n\wedge\neg\beta_1\ldots\wedge\neg \beta_t$. Thus,
\begin{equation}\label{eqmain1}
P(b\vee d)=P(\neg\alpha_n\wedge\neg\beta_1\ldots\wedge\neg \beta_t).
\end{equation}
 Therefore, 
$$
\begin{array}{lll}
\mu_P(\mathbb{B}\llb\alpha_n,\beta_1,\ldots, \beta_t,\alpha_1\rrb)&=&\mu_P(\llb\alpha_n,\beta_1,\ldots, \beta_{t},\alpha_1\rrb)+\displaystyle\sum_{ \beta_l\not\in\{\beta_1,\ldots, \beta_{t}\}}\mu_P(\mathbb{B}\llb\alpha_n,\beta_1,\ldots,\beta_t,\beta_l,\alpha_1\rrb).\\
&=&\mu_P(\llb\alpha_n,\beta_1,\ldots, \beta_t,\alpha_1\rrb)\cdot \left(1+\frac{P(d)}{P(b)}\right)\\
&=&\mu_P(\llb\alpha_n,\beta_1,\ldots, \beta_{t-1},\alpha_1\rrb)\cdot \frac{P(\beta_t)}{P(\neg\alpha_n\wedge\neg\beta_1\ldots\wedge\neg \beta_t)}\cdot \left(\frac{P(b\vee d)}{P(b)}\right)\\
&=&\mu_P(\llb\alpha_n,\beta_1,\ldots, \beta_{t-1},\alpha_1\rrb)\cdot \frac{P(\beta_t)}{P(b)}
\end{array}
$$
where the last equality follows from (\ref{eqmain1}).
\end{proof}
\noindent{\bf Lemma \ref{lemma:cases}.}
{\em Let $\alpha_1\mid b$ be a basic conditional and let $b\geq \alpha_1$. Then, 
\begin{enumerate}[(i)]
\item $\mu_P(\mathbb{S}_1)=P(\alpha_1)$.
\item For all $2\leq j\leq n$, $\mu_P(\mathbb{S}_j)=\displaystyle{\frac{P(\alpha_1)\cdot P(\alpha_j)}{P(b)}}$.
\end{enumerate}}

\begin{proof}
(i) immediately follows from Lemma \ref{lemma:cases0} (i) noticing that, indeed, $\mathbb{S}_1=\llb\alpha_1\rrb$ by Lemma \ref{lemma:property1Sj} (i). Let hence prove (ii). In particular, and without loss of generality, let us prove the claim for $\alpha_j=\alpha_n$. As in the statement, let us call $\beta_1,\ldots, \beta_k$ those atoms of ${\bf A}$ such that $\neg b=\bigvee_{i=1}^k\beta_i$. In particular,  without loss of generality due to Lemma \ref{lemma:property1Sj} (ii), let us put $\beta_k=\alpha_n$.

Notice that $\mu_P(\mathbb{S}_n)=\mu_P(\mathbb{B}\llb\alpha_n,\alpha_1\rrb)$ which, by construction and thanks to Lemma \ref{lemma:cases0}, equals 
$$
\mu_P(\llb\alpha_n,\alpha_1\rrb)+\sum_{j=1}^{k-1}\mu_P(\mathbb{B}\llb\alpha_n,\beta_j,\alpha_1\rrb)=P(\alpha_n)\cdot\frac{P(\alpha_1)}{P(\neg \alpha_n)}+\sum_{j=1}^{k-1}\mu_P(\llb\alpha_n,\alpha_1\rrb)\cdot \frac{P(\beta_j)}{P(b)},
$$ 

Therefore, since by Lemma \ref{lemma:prep1} (i), $\mu_P(\llb\alpha_n,\alpha_1\rrb)=P(\alpha_n)\cdot \frac{P(\alpha_1)}{P(\neg\alpha_n)}$, we get 
$$
\mu_P(\mathbb{S}_n)=P(\alpha_n)\cdot \frac{P(\alpha_1)}{P(\neg\alpha_n)}\left(1+\sum_{j=1}^{k-1}\frac{P(\beta_j)}{P(b)}\right)=P(\alpha_n)\cdot \frac{P(\alpha_1)}{P(\neg\alpha_n)}\left(\frac{P\left(b\vee\bigvee_{j=1}^{k-1}\beta_j\right)}{P(b)}\right).
$$
Finally, notice that  $b\vee\bigvee_{j=1}^{k-1}\beta_j=\neg\alpha_n$. Thus, $P\left(b\vee\bigvee_{j=1}^{k-1}\beta_j\right)=P(\neg\alpha_n)$ and
$$
\mu_P(\mathbb{S}_n)=P(\alpha_n)\cdot P(\alpha_1)\cdot\left(\frac{1}{P(b)}\right)
$$
which settles the claim.
\end{proof}

\subsection*{Proofs from Section \ref{sec:logic}}

\vspace{.2cm}

\noindent
{\bf Theorem \ref{thm:LogIso}.}
$\mathsf{CL}/_\equiv \cong\mathcal{C}({\bf L})$. \\

\noindent The proof needs some previous elaboration. For each satisfiable proposition $\varphi \in \mathsf{L}$, let $\varphi^*$ be its expression in full DNF. We assume that there is a unique such full DNF expression for each formula, i.e. $\varphi^*$ is the unique representative of the class $[\varphi]$ of CPL  formulas equivalent to $\varphi$  in full DNF. Thus, let  

\begin{itemize}
\item $\mathsf{L}^*$ be the set of satisfiable propositional formulas from $\mathsf{L}$ in full DNF,
\item  $V^* = \{(\varphi^* \mid \psi^*) :  \varphi^*  \in \mathsf{L}^*\cup\{\bot\}, \psi^* \in \mathsf{L}^*\}$,\footnote{It is well-known that contradictory formulas like $\bot=\varphi\wedge\neg\varphi$ cannot be expressed in full DNF. For this reason we need to add that symbol to $\mathsf{L}^*$ to be  a possible consequent of a basic conditional.} 
\item  ${\bf Free}(V^*)$ be the freely generated Boolean algebra with generators $V^*$ 

\end{itemize}

Let us define an equivalence relation on ${\bf Free}(V^*)$ as follows: $\Phi^* \equiv \Psi^*$ iff  $LBC \vdash \Phi^* \leftrightarrow \Psi^*$. (the latter due to rules (R1) and (R2) and CPL reasoning). 

Actually, $\equiv$ is indeed a congruence relation on  ${\bf Free}(V^*)$, that satisfies the five properties:

- $(\varphi^* \mid \varphi^*) \equiv \top$

- $(\varphi^* \mid \psi^*) \land (\gamma^* \mid \psi^*) \equiv ((\varphi \land \gamma)^* \mid \psi^*)$

- $\neg(\varphi^* \mid \psi^*) \equiv ((\neg \varphi)^* \mid \psi^*) $

- $(\varphi^* \mid \psi^*) \equiv ((\varphi \land \psi)^* \mid \psi^*)$

-  $(\varphi^* \mid \psi^*) \land (\psi^* \mid \gamma^*) \equiv (\varphi^* \mid \gamma^*)$, when $\vdash_{CPL }\varphi \to \psi \to \gamma$

Therefore, $\equiv$ and $\equiv_\mathfrak{C}$ as defined in Section \ref{sec:BAC} are the same congruence on ${\bf Free}(V^*)$  and, consequently, 
$$
{\bf Free}(V^*)/_{\equiv}\cong{\bf Free}(V^*)/_{\equiv_{\mathfrak{C}}}.
$$

Now, let ${\bf L}^*$ be the Lindenbaum algebra of the language $\mathsf{L}^*$. Notice that, as boolean algebra, ${\bf L}^*$ contains a bottom element, although the symbol $\bot$ does not belong to the language $\mathsf{L}^*$ as we previously remarked. Thus, $V^*=\mathbf{L}^*\times (\mathbf{L}^*)'$, whence they generate isomorphic free algebras, that is to say, ${\bf Free}(V^*)\cong {\bf Free}(\mathbf{L}^*\times (\mathbf{L}^*)')$. Therefore, passing at the quotients, one has ${\bf Free}(V^*)/_{\equiv_{\mathfrak{C}}}\cong {\bf Free}(\mathbf{L}^*\times (\mathbf{L}^*)')/_{\equiv_{\mathfrak{C}}}$ and the latter, by definition, equals $\mathcal{C}({\bf L}^*)$. Thus, we conclude that
$$
{\bf Free}(V^*)/_{\equiv_{\mathfrak{C}}}\cong\mathcal{C}({\bf L}^*) 
$$ 
Finally notice that, since ${\bf L}^*$ is actually isomorphic to the Lindenbaum algebra  ${\bf L}$ of $\mathsf{L}$, one finally has the following
\begin{lemma}\label{lemmaLog1}
$\mathcal{C}({\bf L})\cong{\bf Free}(V^*)/_{\equiv}$.
\end{lemma}
Now  consider the sublanguage $\mathsf{CL}^*$ of $\mathsf{CL}$ built from basic conditionals  $(\varphi^* \mid \psi^*)$. If $\Phi$ is an LBC-formula, we will denote by $\Phi^*$ the formula from  $\mathsf{CL}^*$ obtained by replacing all propositional formulas $\varphi$ in basic conditionals in $\Phi$ by their DNF expressions.

The relation $\equiv$ which is a congruence of ${\bf Free}(V^*)$ clearly also is an equivalence relation on $\mathsf{CL}^*$ and hence, since $\equiv$ is compatible with Boolean operations, the structure $\mathbf{CL}^*=(\mathsf{CL}^*/_\equiv, \wedge,\vee,\neg, \bot, \top)$ is a Boolean algebra.
\begin{lemma}\label{lemmaLog2}
${\bf CL}^*\cong {\bf Free}(V^*)/_{\equiv}$.
\end{lemma}

\begin{proof}
Since ${\bf CL}^*$ and $ {\bf Free}(V^*)/_{\equiv}$ are finite algebras we only need to exhibit a bijection between their supports. Notice that the elements of ${\bf CL}^*$ and those of $ {\bf Free}(V^*)/_{\equiv}$ are indeed the same elements, i.e., equivalence classes of Boolean combinations of basic conditionals whose antecedent and consequent are in DNF. Further the equivalence classes are determined by the same equivalence relation. Thus, clearly, the sets $\mathsf{CL}^*/_{\equiv}$ and $Free(V^*)/_{\equiv}$ are actually the same set and hence the claim follows.
\end{proof}
Now we can finally prove the desired isomorphism claimed in Theorem \ref{thm:LogIso}: $\mathsf{CL}/_\equiv \cong\mathcal{C}({\bf L})$

\begin{proof}
Due to Lemmas \ref{lemmaLog1} and \ref{lemmaLog2} it is enough to prove that ${\bf CL}\cong \mathbf{CL}^*$, that is, for each formula $\Phi$ of $\mathsf{CL}$, $\Phi$ is a theorem of LBC iff $\Phi^*$ can be proved from the DNF-instances $(A1)^*-(A5)^*$ of the axioms of LBC, axioms and rules of CPL, but without the use of (R1) and (R2). 

Every formula $\Phi^*$ is logically equivalent, in CPL, to $\Phi$ and also every $(Ai)$ is clearly logically equivalent to its DNF. Thus if $\Phi^*$ can be proved from $(A1)^*-(A5)^*$ without (R1) and (R2), then $\Phi^*$ is a theorem of LBC and hence so is $\Phi$.  

Conversely, assume that $\pi=\Psi_1,\ldots, \Psi_k$ (where $\Psi_k=\Phi$) is a proof of $\Phi$ in LBC. First of all, we can safely replace each $\Psi_i$ by its DNF $\Psi_i^*$ without loss of generality so defining a list of formulas $\pi^*$ all in DNF. Further notice that, in replacing $\Psi_i$ by $\Psi^*$, any occurrence of the rule $(R1)$ becomes 
\begin{itemize}
\item[(R1)$^*$] from $\vdash_{CPL}\varphi^*\to \psi^*$, derive $(\varphi^*\mid \chi^*)\to (\psi^*\mid \chi^*)$.
\end{itemize}
Then, since in DNF we have $\vdash_{CPL}\varphi^*\leftrightarrow \psi^*$ iff $\varphi^*= \psi^*$, (R1)$^*$ yields the following derived rule: if $\varphi^*= \psi^*$, derive $(\varphi^*\mid \chi^*)\leftrightarrow(\psi^*\mid \chi^*)$, i.e., essentially, for all $\varphi$, derive $(\varphi^*\mid \chi^*)\leftrightarrow(\varphi^*\mid \chi^*)$, which is trivial and hence can be omitted in $\pi^*$. The same obviously applies to (R2). 

Thus, summing up, $\pi^*$ is a proof of $\Phi^*$ made of formulas in $\mathsf{CL}^*$ and which does not use any occurrence of (R1) and (R2). Thus, the claim is settled. 
\end{proof}

\end{document}